\documentclass[Unicode]{cedram-alco}

\usepackage{amsrefs}
\usepackage{tabularray}
\usepackage{subcaption}
\usepackage[indent]{parskip}
\usepackage{xcolor}
\usepackage{mathtools, amssymb}
\usepackage{tikz}
\usetikzlibrary{positioning,fit}
\usepackage{ytableau}
\usepackage{appendix}
\usepackage{enumitem}
\usepackage{bm}
\usepackage{setspace}
\usepackage{fullpage}
\usepackage[T1]{fontenc}
\usepackage{mlmodern}

\newcommand{\ssA}{\mathsf{A}}
\newcommand{\ssD}{\mathsf{D}}
\newcommand{\ssE}{\mathsf{E}}
\newcommand{\C}{\mathbb{C}}
\newcommand{\scrO}{\mathcal{O}}
\newcommand{\R}{\mathbb{R}}
\DeclareMathOperator{\GL}{GL}
\DeclareMathOperator{\U}{U}
\DeclareMathOperator{\SL}{SL}
\newcommand{\gl}{\mathfrak{gl}}
\DeclareMathOperator{\Sp}{Sp}
\renewcommand{\sp}{\mathfrak{sp}}
\renewcommand{\sl}{\mathfrak{sl}}
\renewcommand{\O}{\operatorname{O}}
\DeclareMathOperator{\SO}{SO}
\newcommand{\so}{\mathfrak{so}}
\newcommand{\g}{\mathfrak{g}}
\renewcommand{\k}{\mathfrak{k}}
\newcommand{\h}{\mathfrak{h}}
\renewcommand{\q}{\mathfrak{q}}
\newcommand{\p}{\mathfrak{p}}
\newcommand{\F}{\mathbf{F}}
\newcommand{\E}{\mathcal{E}}
\renewcommand{\phi}{\varphi}
\DeclareMathOperator{\M}{M}
\DeclareMathOperator{\SM}{SM}
\DeclareMathOperator{\AM}{AM}
\DeclareMathOperator{\SSYT}{SSYT}
\newcommand{\la}{\lambda}
\newcommand{\cor}{{\rm cor}}
\newcommand{\Wedge}{\raisebox{1pt}{\scalebox{.8}{$\bigwedge\!$}}}

\DeclareFontFamily{U}{mathx}{}
\DeclareFontShape{U}{mathx}{m}{n}{<-> mathx10}{}
\DeclareSymbolFont{mathx}{U}{mathx}{m}{n}
\DeclareMathAccent{\widecheck}{0}{mathx}{"71}

\title
[Stanley decompositions of modules of covariants]
{Stanley decompositions of \\ modules of covariants}

\author{\firstname{William} \middlename{Q.} \lastname{Erickson}}
\address{Monmouth College\\ 
Dept. of Mathematics, Statistics, and Computer Science\\
700 E. Broadway\\
Monmouth, IL 61462 (USA)}
\email{william.q.erickson@gmail.com}

\author{\firstname{Markus}  \lastname{Hunziker}}
\address{Baylor University\\
 Dept. of Mathematics\\
 One Bear Place \#97328\\
Waco, TX 76798 (USA)}
\email{Markus\_Hunziker@baylor.edu}

  

\begin{document}

\begin{abstract}
Let $H$ be a complex reductive group, with finite-dimensional representations $W$ and $U$.
The \emph{module of covariants} for $W$ of type $U$ is the space of all $H$-equivariant polynomial maps $\phi: W \longrightarrow U$.
In this paper, we take $H$ to be one of the classical groups $\GL(V)$, $\O(V)$, or $\Sp(V)$, where $W$ is a direct sum of copies of $V$ and $V^*$, and $U$ is an arbitrary rational representation (with $U$ restricted to exterior powers of $V$ in the $\O(V)$ case).
Our main result gives uniform Stanley decompositions of these modules of covariants, with Stanley spaces parametrized by combinatorial objects we call \emph{jellyfish}.
As a corollary, we write down the Hilbert series as a finite sum of rational functions, each with a combinatorial interpretation in terms of lattice paths.
Notably, these results do not rely on the module being Cohen--Macaulay.
We further apply our methods to invariant rings for $\SL(V)$ and $\SO(V)$.
Our proofs rely on previous work by Jackson on standard monomial theory for dual reductive pairs, since classical modules of covariants can be viewed via Howe duality as Harish-Chandra modules of unitary highest weight representations of a certain real reductive group.
As a first step toward extending this program to arbitrary unitary highest weight representations (including those of the exceptional groups), we establish analogous results uniformly for the Wallach representations of type ADE.
\end{abstract}

\maketitle

%% Theorems should be defined using amsthm syntax
%% The commented out theorem styles are already defined, please do not redefine them.
\theoremstyle{plain} 

\theoremstyle{definition}

\theoremstyle{remark}

\section{Introduction}

\subsection*{Modules of covariants}
\label{sub:MOCs}

Modules of covariants are generalizations of invariant rings, and (under different terminology) played a key role in 19th-century classical invariant theory.
Let $H$ be a complex reductive group, and let $W$ and $U$ be finite-dimensional representations of $H$.
The \emph{module of covariants} for $W$ of type $U$ is the space of all $H$-equivariant polynomial functions $\phi: W \longrightarrow U$.
It is a module over the ring of invariants $\C[W]^H$ via multiplication of functions; in fact, $\C[W]^H$ is itself the module of covariants of trivial type.
In this paper, for the most part, we consider the case where $H$ is a classical group $\GL(V)$, $\O(V)$, or $\Sp(V)$, and $W$ is a direct sum of arbitrarily many copies of $V$ and $V^*$, and $U$ is an arbitrary rational representation of $H$.
This setting is the framework for Weyl's fundamental theorems of classical invariant theory~\cite{Weyl}.

In the 1970s, Schwarz~\cite{Schwarz} classified \emph{cofree} representations of simple algebraic groups, where every module of covariants is free, and hence has the \emph{Cohen--Macaulay} property. 
Roughly speaking, being Cohen--Macaulay is tantamount to having a structure that is combinatorially ``nice.''
For example, suppose that $M$ is a finitely generated graded module over a complex polynomial ring $S$.
If $M$ is Cohen--Macaulay, then there exists a system of homogeneous, algebraically independent elements $\theta_1, \ldots, \theta_\ell \in S$ such that $M$ is a graded free module over $\C[\theta_1, \ldots, \theta_\ell]$, and admits a \emph{Hironaka decomposition}
\begin{equation}
    \label{Hironaka}
    M = \bigoplus_{i \in I} \C[\theta_1, \ldots, \theta_\ell] \: \eta_i,
\end{equation}
where $I$ is a finite set indexing certain homogeneous elements $\eta_i \in M$.
A Hironaka decomposition~\eqref{Hironaka} immediately yields the Hilbert--Poincar\'e series of $M$ in the form
\[
P(M; t) = \frac{\sum_{i \in I} t^{\deg \eta_i}}{\prod_{j=1}^\ell (1-t^{\deg \theta_j})}.
\]
There was a flurry of interest in modules of covariants during the 1990s, with papers by  Brion~\cite{Brion}, Broer~\cite{Broer}, and Van den Bergh~\cites{Vandenberg91,VandenBergh} addressing the question of when modules of covariants have the Cohen--Macaulay property in the non-cofree case.
Currently, modules of covariants play a central role in the rapidly growing study of \emph{equivariant machine learning}; see the final paragraph of the introduction for more context connected to our main result.

\subsection*{Stanley decompositions}

Our original goal was to write down Hironaka decompositions for the classical modules of covariants that are Cohen--Macaulay.
The Cohen--Macaulay property is extremely rare among modules of covariants; 
this follows from the work of Brion~\cite{Brion}, and more explicitly from~\cite{Alexander}*{Ch.~4}.
Gradually we realized that we could obtain equally nice decompositions, and in complete generality, by allowing the $\theta$'s in~\eqref{Hironaka} to vary with each component.
This led us to look instead for \emph{Stanley decompositions}, which take the form
\begin{equation}
    \label{Stanley decomp general}
    M = \bigoplus_{i \in I} \C[\theta_{i,1}, \ldots, \theta_{i,{\ell_i}}] \: \eta_i,
\end{equation}
where, for each $i \in I$, the homogeneous elements $\theta_{i,1}, \ldots, \theta_{i,\ell_i} \in S$ are algebraically independent, and $\eta_i \in M$ is homogeneous.
(Stanley decompositions are so called because they first appeared in work by Stanley~\cite{Stanley82}*{Thm.~5.2} on commutative monoids.)
Each summand in~\eqref{Stanley decomp general} is called the \emph{Stanley space} corresponding to $i \in I$.
With respect to writing down the Hilbert--Poincar\'e series, a Stanley decomposition is effectively just as elegant as a Hironaka decomposition, since~\eqref{Stanley decomp general} yields
\begin{equation}
    \label{Hilbert series from Stanley decomp general}
    P(M;t) = \sum_{i \in I} \frac{t^{\deg \eta_i}}{\prod_{j=1}^{\ell_i} (1-t^{\deg \theta_{i,j}})}.
\end{equation}

In the special case of the invariant rings $\C[W]^H$, there are well-known Stanley decompositions, with Stanley spaces parametrized by families of nonintersecting lattice paths.
These were obtained in the 1990s (see~\cites{Sturmfels,Herzog,Conca94,GhorpadeKrattenthalerPfaffians}) in the context of determinantal rings (which are isomorphic to the invariant rings, by Weyl's fundamental theorems).
Since modules of covariants generalize invariant rings, our task was to find a generalization of lattice path families which could parametrize Stanley decompositions for all modules of covariants in a uniform manner.
(Note that the problem is not one of existence, but rather of combinatorial tractability; indeed, it was shown in~\cite{HVZ}*{Lemma~1.1} that if $S$ is a polynomial ring in $n$ variables over a field, then any finitely generated $\mathbb{Z}^n$-graded $S$-module admits a Stanley decomposition.)

\subsection*{Main result: a nontechnical overview}

The main result in this paper is a uniform combinatorial description of Stanley decompositions and Hilbert--Poincar\'e series for the classical modules of covariants, regardless of whether the modules are Cohen--Macaulay.
We introduce objects called \emph{jellyfish} that parametrize the Stanley spaces in our decompositions.
In particular, let
\[
W \coloneqq 
\begin{cases}
    V^{*p} \oplus V^q, & H = \GL(V),\\
    V^n, & H = \O(V) \text{ or } \Sp(V).
\end{cases}
\]
Recall that the irreducible rational representations $U_\sigma$ of the classical groups are labeled by certain weakly decreasing integer tuples $\sigma$.
Our object of interest is therefore 
\[
M_\sigma \coloneqq \Big\{ \text{$H$-equivariant polynomial functions } \phi: W \longrightarrow U_\sigma \Big\},
\] 
that is, the module of covariants of type $U_\sigma$.
Our main result is the following theorem, with notation to be explained immediately below.
(The statement of the theorem given below applies in the ``interesting'' range, that is, where the rank of $H$ is small enough that $W$ is not cofree.
Nonetheless, in the actual statement of Theorem~\ref{thm:Stanley decomps and HS}, we will also include the much simpler Stanley decomposition that arises in the range where all $M_\sigma$'s are free.)

\begin{thm*}[See the detailed statement in Theorem~\ref{thm:Stanley decomps and HS}]
    Let $H = \GL(V)$, $\O(V)$, or $\Sp(V)$.
    Let $U_\sigma$ be an irreducible rational representation of $H$; if $H = \O(V)$, then assume $U_\sigma$ is an exterior power of $V$.
    We have a Stanley decomposition
    \begin{equation}
    \label{main decomp intro}
    M_\sigma = \bigoplus_{(\F, T) \in \mathcal{J}(\sigma)} \C[ f_{ij} : (i,j) \in \F] \: f_{\cor(\F)} \cdot \phi_T,
    \end{equation}
    where $\mathcal{J}(\sigma)$ is the set of jellyfish of shape $\sigma$ \textup{(}see Definition~\ref{def:jellyfish}\textup{)}.
\end{thm*}

We illustrate the notation in~\eqref{main decomp intro} by means of an example.
Let $H = \Sp(V)$, where $\dim V = 2k$, and let $W = V^n$; in particular, we take $k=3$ and $n=8$, with $\sigma = (5,4,2)$.

\begin{itemize}
    \item Each $f_{ij} \in \C[W]^H$ is the quadratic contraction $f_{ij} : (v_1, \ldots, v_n) \mapsto \omega(v_i, v_j)$, where $\omega$ is the nondegenerate skew-symmetric bilinear form preserved by $\Sp(V)$.
    Note that the $f_{ij}$'s therefore range over all pairs $(i,j) \in \mathbf{P} \coloneqq \{(i,j) : 1 \leq i < j \leq n=8 \}$.
    Depicting $\mathbf{P}$ using matrix coordinates, we obtain an upper triangular staircase pattern, as in~\eqref{P and F intro} below.
    
    \item $\F$ is a subset of $\mathbf{P}$ formed by taking all the points lying weakly above the $k$th antidiagonal (\ie the points in the shaded isosceles triangle in~\eqref{P and F intro}), together with the points lying along some family of $k$ lattice paths connecting this antidiagonal to the right-hand edge of $\mathbf{P}$.
    In the example~\eqref{P and F intro}, we show a typical $\F$ consisting of the shaded points.
    (The shaded triangle somewhat resembles the bell of an actual jellyfish, and the lattice paths resemble the beginnings of tentacles.)

    \begin{equation} 
\label{P and F intro}
\tikzstyle{corner}=[rectangle,draw=black,thin, minimum size = 8pt, inner sep=2pt]
\tikzstyle{dot}=[circle,fill=black, minimum size = 4.5pt, inner sep=0pt]
\begin{tikzpicture}[scale=.3,baseline=(current bounding box.center),every node/.style={scale=.7}]
\foreach \x in {2,...,8}{\foreach \y in {\x,...,8}{\node [dot] at (10-\x,\y) {};}}
\node at (0,8) {$1$};
\node at (0,7) {$2$};
\node at (0,6) {$3$};
\node at (0,5) {$4$};
\node at (0,4) {$5$};
\node at (0,3) {$6$};
\node at (0,2) {$7$};
\node at (0,1) {$8$};
\node at (1,9) {$1$};
\node at (2,9) {$2$};
\node at (3,9) {$3$};
\node at (4,9) {$4$};
\node at (5,9) {$5$};
\node at (6,9) {$6$};
\node at (7,9) {$7$};
\node at (8,9) {$8$};
\node[scale=1.3] at (4.5,0) {$\mathbf{P}$};
\draw [densely dotted] (1.5,8.5) -- ++(7,0) -- ++(0,-7) -- ++(-1,0) -- ++(0,1) -- ++(-1,0) -- ++(0,1) -- ++(-1,0) -- ++(0,1) -- ++(-1,0) -- ++(0,1) -- ++(-1,0) -- ++(0,1) -- ++(-1,0) -- ++(0,1) -- ++(-1,0) -- ++(0,1);
\end{tikzpicture}
\hspace{10ex}
\begin{tikzpicture}[scale=.3,baseline=(current bounding box.center),every node/.style={scale=.7}]
\draw [white,fill=lightgray] (1.5,8.5) -- ++ (0,-1) -- ++ (1,0) -- ++ (0,-1) -- ++ (1,0) -- ++ (0,-1) -- ++ (1,0) -- ++ (0,1)  -- ++ (1,0) -- ++ (0,1) -- ++ (1,0) -- ++ (0,1) -- cycle;
\draw[line width=3pt, lightgray] (5,6) -- ++(0,-1) -- ++(2,0) -- ++(0,-1) -- ++(1,0) (6,7) -- ++(1,0) -- ++(0,-1) -- ++(1,0)  (7,8) -- ++(1,0) -- ++(0,-1);
\node at (8,8) [corner] {};
\node at (7,7) [corner] {};
\node at (7,5) [corner] {};
\foreach \x in {2,...,8}{\foreach \y in {\x,...,8}{\node [dot] at (10-\x,\y) {};}}
\node at (0,8) {$1$};
\node at (0,7) {$2$};
\node at (0,6) {$3$};
\node at (0,5) {$4$};
\node at (0,4) {$5$};
\node at (0,3) {$6$};
\node at (0,2) {$7$};
\node at (0,1) {$8$};
\node at (1,9) {$1$};
\node at (2,9) {$2$};
\node at (3,9) {$3$};
\node at (4,9) {$4$};
\node at (5,9) {$5$};
\node at (6,9) {$6$};
\node at (7,9) {$7$};
\node at (8,9) {$8$};
\node[scale=1.3] at (4.5,0) {$\mathbf{F} \subseteq \mathbf{P}$};
\draw [densely dotted] (1.5,8.5) -- ++(7,0) -- ++(0,-7) -- ++(-1,0) -- ++(0,1) -- ++(-1,0) -- ++(0,1) -- ++(-1,0) -- ++(0,1) -- ++(-1,0) -- ++(0,1) -- ++(-1,0) -- ++(0,1) -- ++(-1,0) -- ++(0,1) -- ++(-1,0) -- ++(0,1);
\end{tikzpicture}
\end{equation}
    
    \item $f_{\cor(\F)}$ denotes the product of those $f_{ij}$'s whose corresponding points lie on the \emph{corners} (\ie the east-to-south turns) of the lattice paths determined by $\F$.
    The subset $\F$ shown in~\eqref{P and F intro} has three corners, indicated by squares.

    \item $T$ is a semistandard tableau of shape $\sigma = (5,4,2)$, with maximum entry $n=8$.
    By depicting each row of the tableau $T$ as a path outside $\mathbf{P}$, where the tableau entries determine the row indices, we obtain a collection of $k$ paths attached to the right-hand edge of $\mathbf{P}$:
    \begin{equation}
    \label{T in intro}
    \ytableausetup{smalltableaux,centertableaux}
        T = \ytableaushort{22334,3458,56} \: \text{ is depicted as in~\eqref{example in intro}.  Explicitly:} 
    \hspace{3ex}
\begin{tikzpicture}[scale=.35,baseline=(current bounding box.center),every node/.style={scale=.7}]
\tikzstyle{dot}=[circle,fill=black, minimum size = 4.5pt, inner sep=0pt]
\draw [densely dotted] 
(7.5,8.5) -- ++ (1,0) -- ++(0,-7) -- ++(-1,0);
\draw[ultra thick, lightgray] (8,4) \foreach \x/\y in {1/0,1/-1}{
-- ++(\x,\y)
};
\draw[ultra thick, lightgray] (8,6)  \foreach \x/\y in {1/0,1/-1,1/-1,1/-3}{
-- ++(\x,\y)
};
\draw[ultra thick, lightgray] (8,7) \foreach \x/\y in {1/0,1/0,1/-1,1/0,1/-1}{
-- ++(\x,\y)
};
\foreach \y in {2,...,8}{\node [dot] at (8,\y) {};}

\node[draw=black,fill=white,inner sep=1.5pt] at (9,7) {2};
\node[draw=black,fill=white,inner sep=1.5pt] at (10,7) {2};
\node[draw=black,fill=white,inner sep=1.5pt] at (11,6) {3};
\node[draw=black,fill=white,inner sep=1.5pt] at (12,6) {3};
\node[draw=black,fill=white,inner sep=1.5pt] at (13,5) {4};

\node[draw=black,fill=white,inner sep=1.5pt] at (9,6) {3};
\node[draw=black,fill=white,inner sep=1.5pt] at (10,5) {4};
\node[draw=black,fill=white,inner sep=1.5pt] at (11,4) {5};
\node[draw=black,fill=white,inner sep=1.5pt] at (12,1) {8};

\node[draw=black,fill=white,inner sep=1.5pt] at (9,4) {5};
\node[draw=black,fill=white,inner sep=1.5pt] at (10,3) {6};

\end{tikzpicture}
\end{equation}
(In general, it is nontrivial to determine the points of attachment; see~\eqref{Sp end} for more examples.)
    
    \item $\phi_T : W \longrightarrow U_\sigma$ is an $H$-equivariant polynomial map induced canonically by the tableau $T$, to be defined in~\eqref{phi_T GL Sp}.
    The map $\phi_T$ has degree $|\sigma| \coloneqq \sum_i \sigma_i = 11$.
    
    \item A pair $(\F, T)$ is said to be a \emph{jellyfish of shape $\sigma$} if the lattice paths determined by $\F$ align with the paths determined by $T$.
    The set of jellyfish of shape~$\sigma$ is denoted by~$\mathcal{J}(\sigma)$.
\end{itemize}
In the present example, if $\F$ and $T$ are as given in~\eqref{P and F intro} and~\eqref{T in intro}, respectively, then this particular pair $(\F, T)$ does indeed belong to $\mathcal{J}(\sigma)$; we see this immediately by depicting $\F$ and $T$ together in a single diagram~\eqref{example in intro}.
Specifically, we check that the lattice paths in  $\F$ align with those representing $T$, thus forming ``tentacles.''
On the right-hand side of the diagram in~\eqref{example in intro}, we write down the Stanley space corresponding to this jellyfish $(\F, T)$ in the decomposition~\eqref{main decomp intro}:

\begin{equation}
\label{example in intro}
\tikzstyle{corner}=[rectangle,draw=black,thin, minimum size = 8pt, inner sep=2pt]
\tikzstyle{dot}=[circle,fill=black, minimum size = 4.5pt, inner sep=0pt]
\tikzstyle{endpt}=[circle,fill=lightgray, minimum size = 5pt, inner sep=0pt]
\begin{tikzpicture}[scale=.3,baseline=(current bounding box.center),every node/.style={scale=.7}]
\draw [white,fill=lightgray] (1.5,8.5) -- ++ (0,-1) -- ++ (1,0) -- ++ (0,-1) -- ++ (1,0) -- ++ (0,-1) -- ++ (1,0) -- ++ (0,1)  -- ++ (1,0) -- ++ (0,1) -- ++ (1,0) -- ++ (0,1) -- cycle;
\draw [densely dotted] (1.5,8.5) -- ++(7,0) -- ++(0,-7) -- ++(-1,0) -- ++(0,1) -- ++(-1,0) -- ++(0,1) -- ++(-1,0) -- ++(0,1) -- ++(-1,0) -- ++(0,1) -- ++(-1,0) -- ++(0,1) -- ++(-1,0) -- ++(0,1) -- ++(-1,0) -- ++(0,1);
\draw[ultra thick, lightgray] (5,6) -- ++(0,-1) -- ++(2,0) -- ++(0,-1) -- ++(1,0) \foreach \x/\y in {1/0,1/-1}{
-- ++(\x,\y) node[endpt]{}
};
\draw[ultra thick, lightgray] (6,7) -- ++(1,0) -- ++(0,-1) -- ++(1,0)  \foreach \x/\y in {1/0,1/-1,1/-1,1/-3}{
-- ++(\x,\y) node[endpt]{}
};
\draw[ultra thick, lightgray] (7,8) -- ++(1,0) -- ++(0,-1) \foreach \x/\y in {1/0,1/0,1/-1,1/0,1/-1}{
-- ++(\x,\y) node[endpt]{}
};
\node at (8,8) [corner] {};
\node at (7,7) [corner] {};
\node at (7,5) [corner] {};
\foreach \x in {2,...,8}{\foreach \y in {\x,...,8}{\node [dot] at (10-\x,\y) {};}}
\node at (0,8) {$1$};
\node at (0,7) {$2$};
\node at (0,6) {$3$};
\node at (0,5) {$4$};
\node at (0,4) {$5$};
\node at (0,3) {$6$};
\node at (0,2) {$7$};
\node at (0,1) {$8$};
\node at (1,9) {$1$};
\node at (2,9) {$2$};
\node at (3,9) {$3$};
\node at (4,9) {$4$};
\node at (5,9) {$5$};
\node at (6,9) {$6$};
\node at (7,9) {$7$};
\node at (8,9) {$8$};
\node[scale=1.3,align=center] at (4.5,-1) {a jellyfish \\ $(\mathbf{F}, T) \in \mathcal{J}(\sigma)$};
\end{tikzpicture}
\quad \leadsto \quad \C\!\underbrace{\left[\begin{aligned}&f_{12}, f_{13}, f_{14}, f_{15}, f_{16}, \\
& f_{23}, f_{24}, f_{25}, f_{34}, \\ &f_{17}, f_{18}, f_{28}, \\ 
&f_{26}, f_{27}, f_{37}, f_{38}, \\ &f_{35}, f_{45}, f_{46}, f_{47}, f_{57}, f_{58} \end{aligned}\right]}_{\mathclap{\text{$f_{ij}$'s corresponding to points $(i,j) \in \mathbf{F}$}}} \;  \underbrace{f_{18} f_{27} f_{47}}_{\mathclap{\substack{\text{corners} \\ \text{of $\mathbf{F}$}}}} \cdot \,\phi_T
\end{equation}

As a corollary of our main result,  using~\eqref{Hilbert series from Stanley decomp general}, we can write down the Hilbert--Poincar\'e series of $M_\sigma$.
As a preview of Corollary~\ref{cor:HS}, recalling that each $f_{ij}$ has degree 2 and $\phi_T$ has degree $|\sigma|$, we observe that the Stanley space in~\eqref{example in intro} contributes the term
\[
    \frac{(t^2)^{\text{\#corners of $\F$}}}{(1-t^2)^{\text{\#points in $\F$}}} \cdot t^{|\sigma|} = \frac{(t^2)^3}{(1-t^2)^{22}} \cdot t^{11}
\]
to the Hilbert--Poincar\'e series of $M_\sigma$.
In fact, by gathering together the jellyfish with respect to the endpoints of their lattice paths, we are able to write down the Hilbert--Poincar\'e series as a sum over just the sets of endpoints, rather than over all jellyfish.
Thus for the example above (upon programming the combinatorial definitions given in this paper), we obtain the following Hilbert--Poincar\'e series for $M_\sigma$:
\begin{align*}
 & \phantom{==} t^{|\sigma|} \cdot \hspace{-2ex}\sum_{\substack{\text{sets of} \\ \text{endpoints} \\ \mathbf{E}}} \left(\frac{\sum_{\text{$\F$ with endpoints $\mathbf{E}$}} \: (t^2)^{\text{\#corners of $\F$}}}{(1-t^2)^{\text{size of any $\F$ ending at $\mathbf{E}$}}} \right) \left( \text{\#tableaux $T$ that attach at $\mathbf{E}$} \right) \\
&=
t^{11} \bigg[ \underbrace{\left(\frac{1}{(1-t^2)^{18}}\right)(41580)}_{\text{$\mathbf{E}$ is topmost set of endpoints}} + \cdots + \underbrace{\left(\frac{1+t^2+t^4+t^6}{(1-t^2)^{27}}\right)(8316)}_{\text{$\mathbf{E}$ is bottommost set of endpoints}} \bigg] \\[2ex]
&= \frac{498960t^{11} - 887040 t^{13} + 460152 t^{15} + 44352 t^{17} - 121968 t^{19} + 44352 t^{21} - 5544 t^{23}}{(1-t^2)^{27}}.
\end{align*}
(The large coefficients are the price paid for having chosen an example big enough to yield an interesting diagram in~\eqref{example in intro}.)
In general, the Hilbert--Poincar\'e series can be verified using the generalized BGG resolutions constructed in~\cites{EW,EnrightHunziker04};
see also the Hilbert series results in Nishiyama--Zhu~\cite{NZ}*{\S4}, viewed from a very different approach.
By contrast, the jellyfish approach allows us to write these series explicitly as a finite sum of rational functions, whose numerators can be interpreted combinatorially in terms of lattice paths.

Section~\ref{sec:proofs} is devoted to the proof of Theorem~\ref{thm:Stanley decomps and HS}.
Our point of departure is the standard monomial theory for reductive dual pairs presented by Jackson~\cite{Jackson}.
Standard monomial theory was developed by Seshadri~\cite{Seshadri}, as a way of determining explicit linear bases for sections of line bundles over generalized flag varieties; see also~\cites{LakshimaiSeshadri,LakshmibaiRaghavan}.
In related work, De Concini--Procesi~\cite{DeConciniProcesi} wrote down canonical bases of standard monomials, in terms of (bi)tableaux, for the classical rings of invariants $\C[W]^H$; we also refer the reader to the treatments in~\cite{Procesi}*{Ch.~13} and~\cite{Lakshmibai}*{Ch.~10--12}, and our synthesis in~\cite{EricksonHunzikerTensors}.
In~\cite{Jackson}, standard monomial theory is applied to the ring $\C[W]^N$, where $N$ is a maximal unipotent subgroup of $H$; this ring can be viewed as the direct sum of the modules $M_\sigma$.
The key to our method is the use of Stanley--Reisner theory to carefully organize the standard monomials in $\C[W]^N$ into Stanley spaces.

\subsection*{Complications with the orthogonal group}
\label{sub:O complications}

Unlike the case where $H = \GL(V)$ or $\Sp(V)$, for $H = \O(V)$ it seems that a Stanley decomposition of $M_\sigma$ for arbitrary $\sigma$ (rather than only exterior powers of $V$) may be beyond the reach of our jellyfish approach.
There are several complications that arise in the $\O(V)$ case; these difficulties can already be glimpsed from the standard monomial theory in~\cite{Jackson}, discussed above.
Because $\O(V)$ is disconnected and certain irreducible representations contain two independent highest weight vectors, it is not true in general that $\C[W]^N$ is the multiplicity-free direct sum of the $M_\sigma$'s.
For this reason, Jackson~\cite{Jackson} imposes the condition $\dim V > 2n$ in the orthogonal case.
(The same condition is imposed in the closely related paper~\cite{HKL}.)
This condition automatically excludes all non-free modules of covariants, which are the only ones of real interest for us in the present paper.
Therefore we have adapted the arguments of~\cite{Jackson} (particularly the notion of ``split'' monomial generators) to the orthogonal group in our special case where $U_\sigma = \Wedge^m V$ (see Definition~\ref{def:split O} and Lemma~\ref{lemma:Ok standard monomials}).

An even more serious difficulty is the fact that the irreducible representations $U_\sigma$ of $\O(V)$ have an especially complicated structure as quotients of $\GL(V)$-representations, due to the \emph{orthogonal trace relations} which arise whenever the partition $\sigma$ has more than one column in its Young diagram; see~\cite{KingWelshOrthogonal}.
These trace relations seem to be incompatible with our approach using jellyfish.

\subsection*{Applications in this paper}

In Section~\ref{sec:Special cases}, we extend our jellyfish method to write down Stanley decompositions for the rings of invariants $\C[W]^H$ where $H = \SL(V)$ or $\SO(V)$; see Propositions~\ref{prop:SO invariants} and~\ref{prop:SL invariants}.
In Section~\ref{sec:linear bases}, via Howe duality, we view the classical modules of covariants $M_\sigma$ as $(\g, K)$-modules $L_{\la(\sigma)}$ of unitary highest weight representations of a real reductive group $G_\mathbb{R}$.
(See Theorem~\ref{thm:Howe duality}.)
In fact, for $H = \GL(V)$ or $\O(V)$, \emph{all} irreducible unitary highest weight representations of $G_\mathbb{R}$ arise in this way, as some $L_{\la(\sigma)}$ in the dual pair setting.
We flatten our jellyfish into arc diagrams whose degree sequence encodes the weight basis structure of these modules $L_{\la(\sigma)}$; see Proposition~\ref{prop:weight from degree sequence}.

An ultimate goal of the program in this paper is to find Stanley decompositions for arbitrary unitary highest weight representations.
In the dual pair setting, the modules $L_{\la(0)}$ (corresponding to the invariant rings $\C[W]^H$) are instances of \emph{Wallach representations}, which are defined for any Hermitian symmetric pair $(\g, \k)$.
Outside the dual pair setting, in the cases where $\g$ is simply laced, we show (Section~\ref{sec:ADE}) that Stanley decompositions for Wallach representations can be obtained using the same lattice path method we used for the classical groups: 
namely, for the $k$th Wallach representation, we parametrize the Stanley spaces by families of $k$ maximal nonintersecting lattice paths in the Hasse diagram of the poset of positive noncompact roots of $\g$.
In this way, we reinterpret the Hilbert--Poincar\'e series of the Wallach representations~\cites{EnrightHunziker04,EnrightHunzikerExceptional} in terms of lattice paths, for all Hermitian symmetric pairs of type ADE.
This marks a first step toward extending our jellyfish method outside the dual pair setting.

\subsection*{Further applications}

A further application of our main result is developed in our preprint~\cite{EricksonHunzikerBD2024}.
In that paper, for $H = \GL_k$ and $\Sp_{2k}$ (and conjecturally for $\O_k$), we give a uniform tableau-theoretic interpretation of the multiplicity of the modules of covariants $M_\sigma$, or equivalently, the Bernstein degree of the $(\g,K)$-modules $L_{\la(\sigma)}$.
(Recall that the \emph{multiplicity} of a graded module is obtained by evaluating the numerator of its reduced Hilbert--Poincar\'e series at $t=1$.)
Our result in~\cite{EricksonHunzikerBD2024} both refines and generalizes the main result of Nishiyama--Ochiai--Taniguchi~\cite{NOT}, which gave a formula for Bernstein degree in the range where $k$ is at most the real rank of $G_\mathbb{R}$.
By contrast, our combinatorial Bernstein degree formula in~\cite{EricksonHunzikerBD2024} is valid for all values of $k$, and leads to a closed determinantal formula as well.
The key fact, apparent from~\eqref{Hilbert series from Stanley decomp general}, is that the multiplicity of a graded module $M$ equals the number of Stanley spaces of maximal Krull dimension;
using this in the context of the present paper, we reinterpret the multiplicity of $M_\sigma$ by counting those tableaux whose initial column entries satisfy certain inequalities.

The main result in this paper suggests a strong connection to \emph{equivariant machine learning}, which has grown rapidly as a research area during the last three years.
For machine learning problems exhibiting symmetries (such as those arising naturally in physical sciences), restricting the learning class to equivariant functions between data spaces leads to dramatic improvement over ordinary approaches; 
for a detailed summary, see~\cite{Haddadin} and the references therein.
Therefore, a fundamental problem in equivariant machine learning is finding suitable parametrizations of equivariant functions between group representations.
For example, recent advances in equivariant machine learning have exploited Hironaka decompositions of modules of covariants that are Cohen--Macaulay~\cite{Haddadin}*{p.~88}; as mentioned above, however, such modules are quite rare.
The Stanley decompositions in Theorem~\ref{thm:Stanley decomps and HS} give an equally concrete parametrization, without requiring Cohen--Macaulayness; in fact, they broadly generalize the examples given in the recent \emph{Notices} article~\cite{BlumVillar}. We are therefore hopeful that Theorem~\ref{thm:Stanley decomps and HS} may be of some interest to experts in machine learning.

\section{Classical invariant theory and Stanley--Reisner rings}

Along with relevant exposition, the primary aim of this section is to define the functions $f_{ij}$ and $\phi_T$ that will appear in our main result.

\subsection*{The functions $f_{ij}$}
\label{sub:Weyl}

Let $\M_{k,\ell}$ denote the space of complex $k \times \ell$ matrices, and set $\M_k \coloneqq \M_{k,k}$.
Throughout the paper, $H$ always denotes one of the complex classical groups (the general linear group, the orthogonal group, or the symplectic group, respectively):
\begin{align*}
    \GL_k & \coloneqq \{ h \in \M_k : \det h \neq 0 \},\\
    \O_k & \coloneqq \{ h \in \GL_k : h^t = h^{-1} \},\\
    \Sp_{2k} & \coloneqq \{h \in \GL_{2k} : h^t J h = J \},
\end{align*}
where $J = \left( \begin{smallmatrix}
    0 & I\\-I & 0
\end{smallmatrix} \right)$ and $I$ is the $k \times k$ identity matrix.
Given a classical group $H$, we write $V$ to denote its defining representation $\C^k$ (or $\C^{2k}$, if $H = \Sp_{2k}$), and we write $V^*$ to denote the contragredient representation.
Note that $\O_k$ is the subgroup of $\GL_k$ that preserves the nongenerate symmetric bilinear form $b(v,w) = v^tw$ on $V$; likewise, $\Sp_{2k}$ is the subgroup of $\GL_k$ that preserves the nondegenerate skew-symmetric bilinear form $\omega(v,w) = v^t J w$ on $V$.
Hence for $H = \O_k$ or $\Sp_{2k}$, the form $b$ or $\omega$ induces an equivalence of representations $V \cong V^*$.

Let $W$ denote the direct sum of an arbitrary number of copies of $V$ and $V^*$.
By Weyl's first fundamental theorem of classical invariant theory~\cite{Weyl}, the ring of invariants $\C[W]^H$ is generated by the following contractions $f_{ij}$, ranging over all ordered pairs $(i,j)$ in the index set $\mathbf{P}$:
\renewcommand{\arraystretch}{1.3}
\begin{equation}
    \label{table P fij}
    \begin{array}{|c|c|c|c|}
    \hline
    H & W & \text{Generators $f_{ij}$ of $\C[W]^H$} & \mathbf{P} = \{(i,j) \in \mathbb{N}^2 : \ldots \} \\ \hline
    \GL_k & V^{*p} \oplus V^q & f_{ij} : (v^*_1, \ldots, v^*_p, v_1, \ldots, v_q) \mapsto v^*_i(v_j) & 1 \leq i \leq p, \: 1 \leq j \leq q \\ \hline
    \O_k & V^n & f_{ij} : (v_1, \ldots, v_n) \mapsto b(v_i, v_j) & 1 \leq i \leq j \leq n \\ \hline
    \Sp_{2k} & V^n & f_{ij} : (v_1, \ldots, v_n) \mapsto \omega(v_i, v_j) & 1 \leq i < j \leq n \\ \hline 
    \end{array}
\end{equation}
We choose the notation $\mathbf{P}$ due to the poset structure we will impose in Section~\ref{sec:proofs}.
It will often be convenient to introduce the following matrix coordinates on $W$:
\begin{equation}
    \label{table coordinates}
    \begin{array}{|c|c|c|c|}
    \hline
    H & \C[W] & \text{$H$-action on $\C[W]$} & f_{ij} \\ \hline
    \GL_k & \begin{aligned} &\C[\M_{p,k} \oplus \M_{k,q}] \\ = \; & \C[\{y_{ij}\} \cup \{x_{ij}\}] \end{aligned} & h \cdot f(Y,X) = (Yh, h^{-1}X) & \displaystyle\sum_{\ell=1}^k y_{i \ell} x_{\ell j} 
    \\ \hline
    \O_k & \begin{aligned} & \C[\M_{k,n}] \\ = \; & \C[x_{ij}] \end{aligned} & h \cdot f(X) = f(h^{-1}X) & \displaystyle\sum_{\ell = 1}^k x_{\ell i} x_{\ell j}
    \\ \hline
    \Sp_{2k} & \begin{aligned}
        & \C[\M_{2k, n}] \\
        = \; & \C[x_{ij}]
    \end{aligned} & h \cdot f(X) = f(h^{-1} X) & \displaystyle\sum_{\ell = 1}^k (x_{\ell i} x_{k+\ell, j} - x_{k+\ell,i} x_{\ell j} ) 
    \\ \hline 
    \end{array}
    \end{equation}

\subsection*{Semistandard tableaux and irreducible representations}

A \emph{partition} is a finite, weakly decreasing sequence of positive integers (called its \emph{parts}).
Given a partition $\sigma = (\sigma_1, \ldots, \sigma_m)$, we write $\ell(\sigma) \coloneqq m$ for its \emph{length}, and we write $|\sigma| \coloneqq \sum_i \sigma_i$ for its \emph{size}.
We write $(a^m) \coloneqq (a, \ldots, a)$, and we write $0$ for the empty partition, which has length $0$.
A \emph{Young diagram} of shape $\sigma$ is a left-justified arrangement of $|\sigma|$ many boxes, where the $i$th row from the top contains $\sigma_i$ many boxes.
A \emph{semistandard Young tableau} is a Young diagram in which each box is filled with a positive integer, such that the entries are weakly increasing along each row, and strictly increasing down each column.
Let $[n] \coloneqq \{1,2,\ldots,n\}$.
We write
\[
\SSYT(\sigma, n) \coloneqq \{\text{semistandard Young tableaux of shape $\sigma$, with entries taken from $[n]$}\}.
\]
Note that if $\ell(\sigma) > n$, then $\SSYT(\sigma, n) = \varnothing$.
Below is an example of a partition $\sigma$ along with a semistandard tableau of shape $\sigma$:
\[
\ytableausetup{smalltableaux,centertableaux}
\sigma = (4,3,3,2,1) \quad \leadsto \quad \ytableaushort{1225,234,446,55,6} \in \SSYT(\sigma,n), \text{ where $n \geq 6$.}
\]
We write $\sigma'_j$ to denote the number of boxes in the $j$th column of the Young diagram of $\sigma$; hence in the example above, we have $\sigma'_1 = 5$, $\sigma'_2 = 4$, and so on.
We denote the empty tableau (of shape 0) by the symbol $\varnothing$; thus we have $\SSYT(0, n) = \{ \varnothing \}$.
Given a tableau $T$, we write $T_{i,j}$ to denote the entry in the $i$th row from the top and the $j$th column from the left.

In the case of the group $\GL_k$, we will need to generalize partitions by allowing nonpositive parts, still in weakly decreasing order.
Upon fixing the total number of parts to be $k$, any such generalized partition $\sigma$ can be expressed uniquely as an ordered pair $\sigma = (\sigma^+, \sigma^-)$, where $\sigma^+$ and $\sigma^-$ are true partitions: in particular, $\sigma^+$ consists of the positive parts of $\sigma$, and $\sigma^-$ is the partition obtained by negating and reversing the negative parts of $\sigma$.
For example, if $k=9$ and $\sigma = (6,3,3,2,0,0,-1,-3,-5)$, then we have $\sigma^+ = (6,3,3,2)$ and $\sigma^- = (5,3,1)$.
We write $|\sigma| \coloneqq |\sigma^+| + |\sigma^-|$.

Let $\widehat{H}$ denote the set of irreducible rational representations of $H$ (up to equivalence).
It is well known that the elements of $\widehat{H}$ can be labeled by partitions (or pairs thereof), as follows:
\begin{equation}
\label{H hat}
\widehat{H} = \begin{cases}
    \{ \sigma = (\sigma^+, \sigma^-) : \ell(\sigma^+) + \ell(\sigma^-) \leq k\}, & H = \GL_k,\\[1ex]
    \{ \sigma: \sigma'_1 + \sigma'_2 \leq k \}, & H = \O_k,\\[1ex]
    \{ \sigma: \ell(\sigma) \leq k \}, & H = \Sp_{2k}.
\end{cases}
\end{equation}
Given a classical group $H$, we write $U_\sigma$ to denote a model for the irreducible representation of $H$ labeled by $\sigma \in \widehat{H}$.
If $H = \GL_k$ or $\Sp_{2k}$, then $U_\sigma$ is the irreducible representation of $H$ with highest weight $\sigma$ (in standard coordinates).
For $H = \O_k$, which is disconnected, the situation is more subtle; see~\cite{GW}*{pp.~438--9} for details.
Our object of interest is the \emph{module of covariants} of type $U_\sigma$, denoted by
\begin{equation}
    \label{M sigma}
    M_\sigma \coloneqq \left\{ \text{polynomial maps $\phi: W \longrightarrow U_\sigma$} \; \middle| \; \begin{array}{l} \phi(h \cdot w) = h \cdot \phi(w) \\ \text{for all $h \in H$ and $w \in W$} \end{array} \right\}.
\end{equation}

\subsection*{The functions $\phi_T$: general linear and symplectic groups} \label{sub:phi T}

Throughout this subsection, let $H = \GL_k$ or $\Sp_{2k}$.
Using the coordinate functions $x_{ij}$ and $y_{ij}$ given in~\eqref{table coordinates}, we define the following determinants in $\C[W]$, where $I$ and $J$ are assumed to be sets of positive integers:
\begin{equation}
    \label{table dets GL Sp}
    \begin{array}{|c|c|c|c|}
    \hline
    H & \text{det functions} & \mathcal{T}(\sigma) & \det_T \\ \hline
    \GL_k 
    & \begin{array}{l} \det_J \coloneqq \det\left[x_{ij}\right]_{\begin{subarray}{l} i \in \{k-|J|+1, \ldots, k\}, \\ j \in J, \end{subarray}} \\[2.5ex]
    \det^*_I \coloneqq \det \left[y_{ij}\right]_{\begin{subarray}{l} i \in I, \\ j \in \{1, \ldots, |I|\} \end{subarray}}\end{array} 
    & \begin{array}{r}\SSYT(\sigma^+, q) \\ \times \SSYT(\sigma^-, p) \end{array}
    &
    \quad \displaystyle \prod_{\mathclap{\substack{\text{columns} \\ \text{$J$ of $T^+$} \\ \phantom{.}}}} {\textstyle \det_J \cdot} \prod_{\mathclap{\substack{\text{columns} \\ \text{$I$ of $T^-$}}}} {\textstyle \det^*_I} 
    \\ \hline
    \Sp_{2k} 
    & \det_J \coloneqq \det\left[x_{ij}\right]_{
    \begin{subarray}{l}
        i \in \{k+1, \ldots, k+|J|\},\\
        j \in J
    \end{subarray}}
    & 
    \SSYT(\sigma, n) & \displaystyle\prod_{\mathclap{\substack{\text{columns} \\ \text{$J$ of $T$} \\ \phantom{.}}}}^{\phantom{'}} {\textstyle \det_J}
    \\ \hline 
    \end{array}
\end{equation}
In~\eqref{table dets GL Sp}, for each $\sigma \in \widehat{H}$ we introduce the set $\mathcal{T}(\sigma)$ consisting of certain semistandard tableaux (or pairs thereof);
in particular, for $H = \GL_k$, we write a generic element of $\mathcal{T}(\sigma)$ as a tableau pair $T = (T^+, T^-)$.
Each $T \in \mathcal{T}(\sigma)$ induces the function $\det_T$ defined in~\eqref{table dets GL Sp}, which is the product of det functions obtained by viewing each column of $T$ as the set of its entries.

Since $\GL_k$ and $\Sp_{2k}$ are connected, a choice of maximal unipotent subgroup $N \subseteq H$ uniquely determines a Borel subgroup, which then contains a unique maximal torus isomorphic to $(\C^\times)^k$.
Each $\mu = (\mu_1, \ldots, \mu_k) \in \mathbb{Z}^k$ gives rise to the character $\chi_\mu : (t_1, \ldots, t_k) \mapsto t_1^{\mu_1} \cdots t_k^{\mu_k}$.
A vector $v$ in an $H$-module is called a \emph{weight vector} with weight $\mu$ if, for all $t$ in the maximal torus, we have $t \cdot v = \chi_\mu(t) v$.
In this case we write ${\rm wt}_H(v) = \mu$.
In order to express the weight of a monomial
\[
\mathbf{m} = \prod_{i,j} x^{d_{ij}} \prod_{i,j} y^{e_{ij}} \in \C[W],
\]
we will write 
\begin{equation}
\label{deg notation}
\begin{array}{cc}
\begin{aligned}
    \deg_{x_{i \bullet}}(\mathbf{m}) & \coloneqq \sum_{j} d_{ij},\\
    \deg_{x_{\bullet j}}(\mathbf{m}) &\coloneqq \sum_{i} d_{ij},
\end{aligned}
&
\begin{aligned}
    \deg_{y_{i \bullet}}(\mathbf{m}) & \coloneqq \sum_{j} e_{ij},\\
    \deg_{y_{\bullet j}}(\mathbf{m}) &\coloneqq \sum_{i} e_{ij}.
\end{aligned}
\end{array}
\end{equation}
We write $\sigma^*$ to denote the highest weight of $U_\sigma^*$; in terms of the longest element $w_0$ of the Weyl group of $H$, one has $\sigma^* = -(w_0 \cdot \sigma)$.
Concrete details are given below:

\begin{itemize}
    \item For $H = \GL_k$, let $N$ consist of the upper triangular matrices with $1$'s along the diagonal.
    Then the maximal torus consists of the diagonal matrices
\[
    t = {\rm diag}(t_1, \ldots, t_k), \qquad t_i \in \C^\times.
\]
    Given the $H$-action on $\C[W]$ in~\eqref{table coordinates}, the torus acts on coordinate functions via
\[
    t \cdot y_{ij} = t_j y_{ij}, \qquad t \cdot x_{ij} = t_i^{-1} x_{ij}.
\]
Therefore with respect to this action, the weight of a monomial $\mathbf{m} \in \C[W]$ is given by 
\begin{equation}
    \label{wt H GL}
    {\rm wt}_H(\mathbf{m}) = \Big( \deg_{y_{\bullet 1}}(\mathbf{m}) - \deg_{x_{1 \bullet}}(\mathbf{m}), \ldots, \deg_{y_{\bullet k}}(\mathbf{m}) - \deg_{x_{k \bullet}}(\mathbf{m})\Big).
\end{equation}
Comparing this with~\eqref{table dets GL Sp}, for all $\sigma \in \widehat{H}$ and $T \in \mathcal{T}(\sigma)$ we have
\begin{equation}
    \label{wt det GL} 
    {\rm wt}_{H} (\textstyle \det_{T}) = (\sigma^-, \sigma^+) = \sigma^*,
\end{equation}
where $\sigma^*$ is the $k$-tuple obtained by reversing and negating $\sigma$.

\item For $H = \Sp_{2k}$, let $N$ consist of the matrices with block form $\left[ \begin{smallmatrix} U & A \\ 0 & L\end{smallmatrix}\right]$, where $U,L,A \in \M_k$ such that $U$ (\resp $L$) is upper triangular (\resp lower triangular) with 1's on the diagonal.
Then the maximal torus consists of diagonal matrices of the form
\[
t = {\rm diag}(t_1, \ldots, t_k, t_1^{-1}, \ldots, t_k^{-1}), \qquad t_i \in \C^\times.
\]
Given the $H$-action on $\C[W]$ in~\eqref{table coordinates}, the torus acts on coordinate functions as follows, for $1 \leq i \leq k$:
\[
    t \cdot x_{ij} = t_i^{-1} x_{ij}, \qquad
    t \cdot x_{k+i,j} = t_{i} x_{k+i,j}.
\]
Therefore with respect to this action, the weight of a monomial $\mathbf{m} \in \C[W]$ is given by 
\begin{equation}
    \label{wt H Sp}
    {\rm wt}_H(\mathbf{m}) = \scalebox{.8}{$\Big( \deg_{x_{k+1, \bullet}}(\mathbf{m}) - \deg_{x_{1, \bullet}}(\mathbf{m}), \deg_{x_{k+2, \bullet}}(\mathbf{m}) - \deg_{x_{2, \bullet}}(\mathbf{m}), \ldots, \deg_{x_{2k,\bullet}}(\mathbf{m}) - \deg_{x_{k, \bullet}}(\mathbf{m})\Big).$}
\end{equation}
Comparing this with~\eqref{table dets GL Sp}, for all $\sigma \in \widehat{H}$ and $T \in \mathcal{T}(\sigma)$ we have
\begin{equation}
    \label{wt det Sp}
    {\rm wt}_H (\textstyle \det_T) = \sigma = \sigma^*,
\end{equation}
where $\sigma = \sigma^*$ since every rational representation of $\Sp_{2k}$ is self-dual.
\end{itemize}

Assume still that $H = \GL_k$ or $\Sp_{2k}$.
Let $\sigma \in \widehat{H}$.
Let $u_{_{\rm LW}} \in U_\sigma$ be a lowest weight vector; that is, ${\rm wt}_H(u_{_{\rm LW}}) = w_0 \cdot \sigma$.
Let $u^*_{_{\rm LW}} \in U_\sigma^*$ be the linear functional sending $u_{_{\rm LW}}$ to 1, and sending every higher weight vector to 0.
Note that ${\rm wt}_H(u^*_{_{\rm LW}}) = -(w_0 \cdot \sigma) = \sigma^*$, and therefore $u_{_{\rm LW}}^*$ is a highest weight vector in $U_\sigma^*$.
Due to the $H$-equivariance property in~\eqref{M sigma}, any $\phi \in M_\sigma$ is uniquely determined by the function $u_{_{\rm LW}}^* \circ \phi \in \C[W]$; 
in other words, for all $w \in W$, its image $\phi(w) \in U_\sigma$ is completely determined by its $u_{_{\rm LW}}$-component.
Note that ${\rm wt}_H(u_{_{\rm LW}}^* \circ \phi) = \sigma^*$, since for $t$ in the maximal torus of $H$ we have
\begin{align}
\label{weight of u circ phi}
\begin{split}
    (t \cdot (u_{_{\rm LW}}^* \circ \phi))(w) &= u_{_{\rm LW}}^* (\phi(t^{-1} \cdot w)) \\
    &= u_{_{\rm LW}}^* ( t^{-1} \cdot \phi(w)) \\
    &= (t \cdot u_{_{\rm LW}}^*)(\phi(w)) = \chi_{\sigma^*}(t) \: (u_{_{\rm LW}}^* \circ \phi)(w).
\end{split}
\end{align}
Repeating the same calculation for an element of $N$ shows that $u_{_{\rm LW}}^* \circ \phi$ is $N$-invariant, since $u_{_{\rm LW}}^*$ is a highest weight vector in $U_\sigma^*$.
Conversely, any $N$-invariant function $f \in \C[W]^N$ with weight $\sigma^*$ determines a unique element $\phi \in M_\sigma$ such that $u_{_{\rm LW}}^* \circ \phi = f$.
(In fact, this map is the inverse of the canonical isomorphism $\Psi$ in~\eqref{Psi} below.)
Therefore, in light of~\eqref{wt det GL} and~\eqref{wt det Sp}, and the $N$-invariance of $\det^*_I$ and $\det_J$ in~\eqref{table dets GL Sp}, the functions $\det_T$ can play the role of this $f$;
indeed, for all $T \in \mathcal{T}(\sigma)$, we now define
\begin{equation}
    \label{phi_T GL Sp}
     \phi_T \coloneqq \text{the unique element $\phi \in M_\sigma$ such that $u_{_{\rm LW}}^* \circ \phi = \textstyle \det_T$}.
\end{equation}
In words, $\phi_T$ sends $w \in W$ to the vector in $U_\sigma$ whose lowest weight component equals $\det_T(w)$ and whose remaining weight components are determined by $H$-equivariance.
Note that with respect to the standard grading on $\C[W]$ whereby each coordinate function has degree 1, both $\det_T$ and $\phi_T$ have degree $|\sigma|$.

\subsection*{The functions $\phi_T$: orthogonal group}
\label{sub:O tableaux}

Let $H = \O_k$, and let $\sigma = (1^m)$ for some $0 \leq m \leq k$.
Since $\O_k$ is disconnected, the highest weight arguments used in the previous subsection are not valid, so we will construct the analogous maps $\phi_T$ directly.
Since $T \in \mathcal{T}(\sigma)$ is a single column with entries $T_1 < \cdots < T_m$, we will make the identification
\[
\mathcal{T}(\sigma) = \textstyle \binom{[n]}{m} \coloneqq \{ \text{$m$-element subsets of $[n]$} \},
\]
so that $T = \{T_1, \ldots, T_m\}$.
We define an $H$-equivariant map $\phi_T \in M_\sigma$ in the obvious way:
\begin{align}
\label{phi O}
\begin{split}
    \phi_T : W & \longrightarrow U_\sigma = \Wedge^m V,\\
    (v_1, \ldots, v_n) &\longmapsto v_{_{T_1}} \wedge \cdots \wedge v_{_{T_m}}.
    \end{split}
\end{align}
Note that $\phi_T$ is a polynomial function of degree $|\sigma| = m$.

\subsection*{Stanley--Reisner theory}
\label{sub:Stanley-Reisner}

 We follow the exposition given in~\cite{StanleyAC}*{Ch.~12}.
 Given a finite set $\mathbf{V}$, an \emph{abstract simplicial complex} on $\mathbf{V}$ is a collection $\Delta$ of subsets of $\mathbf{V}$ such that
 \begin{itemize}
     \item $\{v\} \in \Delta$ for all $v \in \mathbf{V}$;
     \item if $\mathbf{S} \in \Delta$ and $\mathbf{R} \subseteq \mathbf{S}$, then $\mathbf{R} \in \Delta$.
 \end{itemize}
 The set $\mathbf{V}$ is called the \emph{vertex set}; elements $v \in \mathbf{V}$ are called \emph{vertices}, and the elements of $\Delta$ are called \emph{faces}. 
A maximal face in $\Delta$ (with respect to inclusion) is called a \emph{facet}.
Let
\[
\mathcal{F}(\Delta) \coloneqq \{ \F: \text{$\F$ is a facet of $\Delta$} \}.
\]
We say that $\Delta$ is \emph{pure} if every facet has the same cardinality.  
A pure simplicial complex is said to be \emph{shellable} if there exists an ordering $\F_1, \ldots, \F_{r}$ of its facets with the following property: 
for all $i = 1, \ldots, r$, the facet $\F_i$ contains a unique minimal subset ${\rm res}(\F_i)$ not belonging to the subcomplex generated by $\F_1, \ldots, \F_{i-1}$.
Such an ordering is called a \emph{shelling} of $\Delta$, and ${\rm res}(\F_i)$ is called the \emph{restriction} of the facet $\F_i$.

To each $v \in \mathbf{V}$ we associate the indeterminate $z_v$.
Given a subset $\mathbf{S} \subseteq \mathbf{V}$, we write
\[
z_{\mathbf{S}} := \prod_{v \in \mathbf{S}} z_v \qquad \text{and} \qquad \C[\mathbf{S}] := \C[z_v : v \in \mathbf{S} ].
\]
Let $I_\Delta$ denote the ideal of $\C[\mathbf{V}]$ generated by all monomials $z_{\mathbf{S}}$ such that $\mathbf{S} \not\in \Delta$.
Such an $\mathbf{S}$ is called a \emph{nonface} of $\Delta$, and $I_\Delta$ is generated by the minimal nonfaces (\ie nonfaces that contain no proper nonface).
The \emph{support} of a monomial $\mathbf{z} \in \C[\mathbf{V}]$ is the set
\begin{equation}
    \label{support definition}
    {\rm supp}(\mathbf{z}) \coloneqq \{ v \in \mathbf{V} : \text{$z_v$ divides $\mathbf{z}$}\}.
\end{equation}
The quotient
\begin{equation}
    \label{C[Delta]}
    \C[\Delta] := \C[\mathbf{V}]/I_\Delta
\end{equation}
is called the \emph{Stanley--Reisner} ring of $\Delta$.  
The ring $\C[\Delta]$ has a linear basis consisting of (the images of) the monomials $\mathbf{z}$ such that ${\rm supp}(\mathbf{z}) \in \Delta$.
Moreover, each shelling of $\Delta$ induces a Stanley decomposition 
\begin{equation}
    \label{Stanley decomp SR ring}
    \C[\Delta] = \bigoplus_{\mathclap{\F \in \mathcal{F}(\Delta)}} \: \C[\F] \: z_{{\rm res}(\F)},
\end{equation}
where we identify monomials in $\C[\mathbf{V}]$ with their images in $\C[\Delta]$.

\subsection*{Posets and labelings}
\label{sub:posets}

Now suppose that the vertex set $\mathbf{V}$ is equipped with a partial order $\leq$.
We use standard terminology from order theory; see~\cite{Wachs}.
Suppose further that the poset $\mathbf{V}$ is \emph{bounded}, that is, contains a minimal and maximal element.
Recall that a subset $\mathbf{S} \subseteq \mathbf{V}$ is called a \emph{chain} if $\mathbf{S}$ is totally ordered with respect to $\leq$;
it is called an \emph{antichain} if, for all distinct $v,w \in \mathbf{S}$, neither $v \leq w$ nor $w \leq v$.
The \emph{width} of a subset $\mathbf{S} \subseteq \mathbf{V}$ is the cardinality of the largest antichain contained in $\mathbf{S}$; we denote this by ${\rm width}(\mathbf{S})$.
Equivalently, by Dilworth's theorem, ${\rm width}(\mathbf{S})$ is the smallest number of disjoint chains into which $\mathbf{S}$ can be partitioned.
We write $v \lessdot w$ to denote the \emph{covering} relation, meaning that $v < w$ and there is no element $x \in \mathbf{V}$ such that $v < x < w$.
A \emph{saturated} chain takes the form $v_1 \lessdot v_2 \lessdot \cdots \lessdot v_\ell$.

The \emph{Hasse diagram} of $\mathbf{V}$ is the graph with vertex set $\mathbf{V}$ and edge set $E := \{(v,w) \in \mathbf{V} \times \mathbf{V} : v \lessdot w\}$, oriented so that each edge $(v,w)$ is drawn upward from $v$ to $w$.
A \emph{labeling} of $\mathbf{V}$ is a function $\alpha : E \rightarrow \mathbb{Z}_{\geq 0}$ assigning a nonnegative integer to each edge of the Hasse diagram.
Each labeling $\alpha$ induces a lexicographic ordering on the set of saturated chains $v_1 \lessdot \cdots \lessdot v_\ell$, via the lexicographic order on their \emph{label sequences} $(\alpha(v_1,v_2), \ldots, \alpha(v_{\ell-1},v_\ell))$.

\begin{defi}[EL-labeling; see~\cite{BW83}*{Def.~2.1}]
\label{def:EL-labeling}
    Let $\mathbf{V}$ be a finite bounded poset.
    A labeling $\alpha$ of $\mathbf{V}$ is said to be an \emph{edge-lexicographic labeling} (EL-labeling for short) if, for all $v,w \in \mathbf{V}$ such that $v<w$,
    there exists a unique saturated chain $v \lessdot \cdots \lessdot w$ whose label sequence is nondecreasing and lexicographically precedes the label sequences of all other saturated chains $v \lessdot \cdots \lessdot w$.
\end{defi}

The \emph{order complex} $\Delta(\mathbf{V})$ is the abstract simplicial complex whose faces are the chains in $\mathbf{V}$; hence the facets of $\Delta(\mathbf{V})$ are the maximal chains in $\mathbf{V}$.
Let $\alpha$ be an EL-labeling of $\mathbf{V}$, and $\F$ a facet of $\Delta(\mathbf{V})$.
An element $v \in \F$ is said to be a \emph{descent} of $\F$ (with respect to $\alpha$) if $\F$ contains elements $u \lessdot v \lessdot w$ such that $\alpha(u,v) > \alpha(v,w)$.

\begin{lemma}[See~\cite{Bjorner80}*{Thm.~2.3} and~\cite{BW96}*{Thm.~5.8}]
\label{lemma:EL labeling induces shelling}
    An EL-labeling $\alpha$ of $\mathbf{V}$ induces a shelling of the order complex $\Delta(\mathbf{V})$, where the facets \textup{(}\ie maximal chains\textup{)} are ordered lexicographically by their label sequences with respect to $\alpha$.
    \textup{(}While an EL-labeling does not guarantee that the maximal chains are totally ordered, nevertheless, arbitrarily breaking ties results in a shelling order.\textup{)}
    With respect to a shelling induced by $\alpha$, the restriction of each facet is precisely its set of descents:
\begin{equation}
    \label{R equals descents}
    {\rm res}(\F) = \{ v : \textup{$v$ is a descent of $\F$}\}.
\end{equation}
\end{lemma}

\section{Main result: Stanley decompositions via jellyfish}
\label{sec:Stanley decomps main}

Let $M$ be a finitely generated graded $S$-module, where $S$ is a polynomial ring over $\C$.
Following~\cite{BKU}*{Def.~2.1}, a \emph{Stanley decomposition} of $M$ is a finite family $(S_i, \eta_i)_{i \in I}$
where $\eta_i \in M$ is homogeneous, and $S_i$ is a graded $\C$-algebra retract of $S$ such that $S_i \cap \operatorname{Ann} \eta_i = 0$, and 
\begin{equation}
\label{Stanley decomp module}
    M = \bigoplus_{i \in I} S_i \eta_i
\end{equation}
as a graded vector space.
Each summand in~\eqref{Stanley decomp module} is called the \emph{Stanley space} corresponding to $i \in I$.

\subsection*{Uniform overview}
In our context, in~\eqref{Stanley decomp module} we take $M = M_\sigma$ to be the module of covariants defined in~\eqref{M sigma}, and we take $S = \C[\mathbf{P}] \coloneqq \C[z_{ij} : (i,j) \in \mathbf{P}]$, acting on $M_\sigma$ via the algebra homomorphism
\begin{align}
    \label{pi star map}
    \begin{split}
    \pi^*: S & \longrightarrow \C[W]^H,\\
    z_{ij} & \longmapsto f_{ij},
    \end{split}
\end{align}
where the $f_{ij}$'s are the quadratic contractions defined in~\eqref{table P fij}.
(The map $\pi^*$ arises as a certain comorphism in the theory of Hermitian symmetric pairs; see~\eqref{pi star again} below for details.)
Hence the $S$-action on $M_\sigma$ is given by $z_{ij} \cdot \phi = f_{ij} \: \phi$.
Our main result is Theorem~\ref{thm:Stanley decomps and HS} below, stated in a uniform manner for the three classical groups $H = \GL_k$, $\O_k$, and $\Sp_{2k}$.
The upshot of the theorem is that each module of covariants $M_\sigma$ admits a Stanley decomposition where the Stanley spaces are parametrized by certain combinatorial objects we call \emph{jellyfish} of shape $\sigma$ (see Definition~\ref{def:jellyfish} below).
We summarize the key constructions and notation in the following paragraph.

We depict $\mathbf{P}$ as a grid of points, where $(i,j)$ lies in row $i$ and column $j$.
By a \emph{southeast lattice path} in $\mathbf{P}$, we mean a subset of $\mathbf{P}$ obtained by starting at some point and taking steps either south or east, that is, in the direction $(1,0)$ or $(0,1)$.
We define a \emph{northeast lattice path} similarly, replacing $(1,0)$ by $(-1,0)$.
If the starting point and endpoint coincide, then the lattice path contains only that single point.
The subset $\mathbf{A} \subseteq \mathbf{P}$ is the region lying above a certain antidiagonal (to be defined for each group $H$ below).
Then $\delta(\mathbf{A})$ or $\delta(\mathbf{P})$ denotes the set of points lying on a certain boundary of $\mathbf{A}$ or $\mathbf{P}$, respectively (to be defined for each group $H$ below).
In our construction of jellyfish to follow, we assume that
\[
k \leq \#\delta(\mathbf{P}),
\]
since we will construct $k$ nonintersecting lattice paths with endpoints in $\delta(\mathbf{P})$.
Our main result (Theorem~\ref{thm:Stanley decomps and HS}) will treat the somewhat trivial $k \geq \#\delta(\mathbf{P})$ case separately, without reference to jellyfish.
We define the set
\renewcommand{\arraystretch}{1}
\begin{equation}
    \label{F}
    \mathcal{F} \coloneqq \left\{\big(\mathbf{A} \backslash \delta(\mathbf{A})\big) \sqcup \coprod_{t=1}^k \mathbf{L}_t  : \begin{array}{l} 
    \text{each $\mathbf{L}_t$ is a lattice path} \\
    \text{from $\delta(\mathbf{A})$ to $\delta(\mathbf{P})$}
    \end{array}
    \right\},
\end{equation}
where for $H = \GL_k$ or $\Sp_{2k}$ the lattice paths are southeast, while for $\O_k$ they are northeast.
(The ``$\mathcal{F}$'' notation is meant to suggest ``facets,'' and will be justified in Section~\ref{sec:proofs}.)
We write $\cor(\F) \subseteq \F$ to denote the set of \emph{corners} of $\F \in \mathcal{F}$, meaning certain turns in the defining lattice paths (to be defined for each group $H$ below), and we use the shorthand
\begin{equation}
    \label{f cor}
    f_{\cor(\F)} \coloneqq \prod_{\mathclap{(i,j) \in \cor(\F)}} f_{ij}.
\end{equation}
(Although for $H = \GL_k$ the decomposition of $\F \in \mathcal{F}$ into lattice paths $\mathbf{L}_t$ is in general not unique, nonetheless it turns out that the corners of $\F$ are well defined.)
We define the collection
\begin{equation}
    \label{E}
    \E \coloneqq \left\{ 
    \{ \text{endpoint of $\mathbf{L}_t$} \}_{t=1}^k : \begin{array}{l} 
    \{\mathbf{L}_t\}_{t=1}^k \text{ is a family of nonintersecting}\\ \text{lattice paths from $\delta(\mathbf{A})$ to $\delta(\mathbf{P})$}
    \end{array}
    \right\},
\end{equation}
consisting of all possible sets $\mathbf{E} \subseteq \delta(\mathbf{P})$ of endpoints for the lattice paths appearing in~\eqref{F}.
For $\mathbf{E} \in \E$, we define
\begin{equation}
    \label{F_E}
    \mathcal{F}_{\mathbf{E}} \coloneqq \left\{ \F \in \mathcal{F} : \begin{array}{l} 
    \F = ( \mathbf{A} \backslash \delta(\mathbf{A})) \sqcup \coprod_{t=1}^k \mathbf{L}_t, \\
    \text{where each $\mathbf{L}_t$ is a lattice path} \\
    \text{from $\delta(\mathbf{A})$ to $\mathbf{E}$}
    \end{array}
    \right\}.
\end{equation}
For $H = \Sp_{2k}$ or $\O_k$ we will have $\mathcal{F} = \coprod_{\mathbf{E} \in \E} \mathcal{F}_{\mathbf{E}}$, but for $H = \GL_k$ it is possible that $\mathcal{F}_{\mathbf{E}} = \mathcal{F}_{\mathbf{E}'}$.
Given $\sigma \in \widehat{H}$, and recalling the tableau set $\mathcal{T}(\sigma)$ from~\eqref{table dets GL Sp}, we will define a map
\begin{equation}
    \label{end map}
    {\rm end} : \mathcal{T}(\sigma) \longrightarrow \E
\end{equation}
(see details for each group $H$ below);
very roughly speaking, ${\rm end}(T)$ is the set of endpoints most closely aligning with the initial column of the tableau $T$.
Then for each $\mathbf{E} \in \E$, we define the fiber
\begin{equation}
    \label{T_E}
    \mathcal{T}_{\mathbf{E}}(\sigma) \coloneqq \{ T \in \mathcal{T}(\sigma) : {\rm end}(T) = \mathbf{E} \}.
\end{equation}
Clearly $\mathcal{T}(\sigma) = \coprod_{\mathbf{E} \in \E} \mathcal{T}_{\mathbf{E}}(\sigma)$, although we may have $\mathcal{T}_{\mathbf{E}}(\sigma) = \varnothing$ for certain $\mathbf{E} \in \E$.

\begin{defi}[Jellyfish]
    \label{def:jellyfish}
    Let $H = \GL_k$, $\O_k$, or $\Sp_{2k}$, and let $\sigma \in \widehat{H}$ as in~\eqref{H hat}.
    A \emph{jellyfish of shape $\sigma$} is an element of the set
    \begin{align*}
        \mathcal{J}(\sigma) &\coloneqq \coprod_{\mathbf{E} \in \E} \mathcal{F}_{\mathbf{E}} \times \mathcal{T}_{\mathbf{E}}(\sigma), 
    \end{align*}
    with $\mathcal{F}_{\mathbf{E}}$ and $\mathcal{T}_{\mathbf{E}}(\sigma)$ as in~\eqref{F_E} and~\eqref{T_E}, and with case-by-case details to be described below.
\end{defi}

From now on, we treat the three classical groups in the order $\Sp_{2k}$, $\GL_k$, $\O_k$.
This is because the combinatorics are simpler for $\Sp_{2k}$ than for $\GL_k$, and because $\O_k$ requires special handling in several respects (see the introduction).
If we were to consider only the polynomial representations of $\GL_k$ (\ie those $U_\sigma$'s such that $\sigma^- = 0$), then the $\GL_k$ case would actually be quite straightforward;
however, allowing for arbitrary rational representations makes the situation surprisingly more delicate.

\subsection*{Details for the symplectic group}

Let $H = \Sp_{2k}$, with $\mathbf{P}$ as in~\eqref{table P fij}.
Thus $\mathbf{P}$ is depicted as a strictly upper triangular staircase with side length $n-1$.
Let
\begin{align*}
    \mathbf{A} &\coloneqq \{ (i,j) \in \mathbf{P} : i + j \leq 2(k + 1) \},\\
    \delta(\mathbf{A}) & \coloneqq \{ (i,j) \in \mathbf{A} : i + j = 2(k+1) \text{ or } j = n \},\\
    \delta(\mathbf{P}) &\coloneqq \{ (i,j) \in \mathbf{P} : j=n \}.
\end{align*}
Note that $\delta$ denotes the eastern boundary.
Below are two examples where $n=11$:
\begin{equation} 
    \label{Sp anatomy} 
    \tikzstyle{dot}=[circle,fill=black, minimum size = 4.5pt, inner sep=0pt]
    \tikzstyle{redglow}=[circle,fill=red, minimum size = 7.5pt, inner sep=0pt]
\begin{tikzpicture}          [scale=.3,baseline=(current bounding box.center),every node/.style={scale=.7}]
\node [redglow] at (9,11) {};
\node [redglow] at (8,10) {};
\node [redglow] at (7,9) {};
\node [redglow] at (6,8) {};
\node [scale=1.3, red] at (4,6) {$\delta(\mathbf{A})$};
\node at (1,12) {};
\foreach \x in {2,...,11}{\foreach \y in {\x,...,11}{\node [dot] at (13-\x,\y) {};}}
\draw[densely dotted, thick] (1.5,11.5) --++ (8,0) -- ++(0,-1) -- ++(-1,0) -- ++(0,-1) -- ++(-1,0) -- ++(0,-1) -- ++(-1,0) -- ++(0,-1) -- ++(-2,0) -- ++(0,1) -- ++(-1,0) -- ++(0,1) -- ++(-1,0) -- ++(0,1) -- ++(-1,0) -- cycle;
\draw[densely dotted, thick, cyan]
(10.5,11.5) -- ++ (1,0) -- ++(0,-10) -- ++(-1,0) -- cycle;
\node [scale=1.3,cyan] at (13, 5) {$\delta(\mathbf{P})$};
\node[scale=1.3] at (2,9) {$\mathbf{A}$};
\node[scale=1.3] at (6,1) {$k=4$};
\end{tikzpicture}
\qquad\qquad\qquad\qquad
\begin{tikzpicture}          [scale=.3,baseline=(current bounding box.center),every node/.style={scale=.7}]
\node [redglow] at (11,11) {};
\node [redglow] at (11,10) {};
\node [redglow] at (11,9) {};
\node [redglow] at (11,8) {};
\node [redglow] at (11,7) {};
\node [redglow] at (10,6) {};
\node [redglow] at (9,5) {};
\node [scale=1.3, red] at (7,3) {$\delta(\mathbf{A})$};
\node at (1,12) {};
\foreach \x in {2,...,11}{\foreach \y in {\x,...,11}{\node [dot] at (13-\x,\y) {};}}
\draw[densely dotted, thick] (1.5,11.5) --++ (10,0) -- ++(0,-5) -- ++(-1,0) -- ++(0,-1) -- ++(-1,0) -- ++(0,-1) -- ++(-2,0) -- ++(0,1) -- ++(-1,0) -- ++(0,1) -- ++(-1,0) -- ++(0,1) -- ++(-1,0) -- ++(0,1) -- ++(-1,0) -- ++(0,1) -- ++(-1,0) -- ++(0,1) -- ++(-1,0) -- cycle;
\draw[densely dotted, thick, cyan]
(10.5,11.5) -- ++ (1,0) -- ++(0,-10) -- ++(-1,0) -- cycle;
\node [scale=1.3,cyan] at (13, 5) {$\delta(\mathbf{P})$};
\node[scale=1.3] at (2,9) {$\mathbf{A}$};
\node[scale=1.3] at (6,1) {$k=7$};
\end{tikzpicture}
\end{equation}
Explicitly, we have 
\begin{equation}
    \label{delta A Sp}
    \delta(\mathbf{A}) = \{(t, b_t) : 1 \leq t \leq k \},
\end{equation}
where
\[
b_t 
= \begin{cases}
    n, & t < 2(k+1) - n,\\
    2(k+1)-t & \text{otherwise}.
\end{cases}
\]
Referring to~\eqref{F}, we observe that in the $\Sp_{2k}$ setting, each $\F \in \mathcal{F}$ admits a unique decomposition $\F = (\mathbf{A} \backslash \delta(\mathbf{A})) \sqcup \mathbf{L}_1 \sqcup \cdots \sqcup \mathbf{L}_k$, where each $\mathbf{L}_t$ is a southeast lattice path starting at $(t, b_t)$ and ending in $\delta(\mathbf{P})$.
A point $(i,j) \in \mathbf{L}_t$ is said to be a \emph{corner} of $\F$ if
\begin{equation}
\label{corner Sp}
    \{(i,j-1), \: (i+1, j) \} \subseteq \mathbf{L}_t \text{ and } \{(i+\ell, j-\ell) \in \mathbf{P} : \ell \geq 1 \} \not\subseteq \F.
\end{equation}
In other words, a corner is an east-to-south turn, such that the string of points directly southwest of it contains at least one point not in $\F$.
In~\eqref{Sp F} below, we show an element $\F \in \mathcal{F}$ for each of the two examples in~\eqref{Sp anatomy}.
We shade the points in $\F$, and we draw squares to indicate the elements of $\cor(\F)$:
\begin{equation} 
    \label{Sp F} 
    \tikzstyle{dot}=[circle,fill=black, minimum size = 4.5pt, inner sep=0pt]
    \tikzstyle{corner}=[rectangle,draw=black,thin, minimum size = 8pt, inner sep=2pt]
\begin{tikzpicture}          [scale=.3,baseline=(current bounding box.center),every node/.style={scale=.7}]
\draw [white,fill=lightgray] (1.5,11.5) -- ++ (0,-1) -- ++ (1,0) -- ++ (0,-1) -- ++ (1,0) -- ++ (0,-1) -- ++ (1,0) -- ++ (0,-1) -- ++ (1,0) -- ++ (0,1)  -- ++ (1,0) -- ++ (0,1)  -- ++ (1,0) -- ++ (0,1) -- ++ (1,0) -- ++ (0,1) -- cycle;
\draw[line width=3pt, lightgray] (9,11) -- ++(1,0) node [corner] {} -- ++(0,-1) -- ++(1,0) (8,10) -- ++(0,-1) -- ++(1,0) -- ++(0,-1) -- ++(2,0) (7,9) -- ++(0,-1) -- ++(1,0) -- ++(0,-1) -- ++(2,0) node [corner] {} -- ++(0,-1) -- ++(1,0) node [corner] {} -- ++(0,-1) (6,8) -- ++(0,-1) -- ++(1,0) -- ++(0,-1) -- ++(2,0) node [corner] {} -- ++(0,-2) -- ++(2,0);
\foreach \x in {2,...,11}{\foreach \y in {\x,...,11}{\node [dot] at (13-\x,\y) {};}}
\node[scale=1.3] at (6,2) {$k=4$};
\draw [densely dotted] 
(1.5, 11.5) -- ++(10,0) -- ++(0,-10) -- ++(-1,0) -- ++(0,1) -- ++(-1,0) -- ++(0,1) -- ++(-1,0) -- ++(0,1) -- ++(-1,0) -- ++(0,1) -- ++(-1,0) -- ++(0,1) -- ++(-1,0) -- ++(0,1) -- ++(-1,0) -- ++(0,1) -- ++(-1,0) -- ++(0,1) -- ++(-1,0) -- ++(0,1) -- ++(-1,0) -- ++(0,1);
\end{tikzpicture}
\qquad\qquad\qquad\qquad
\begin{tikzpicture}          [scale=.3,baseline=(current bounding box.center),every node/.style={scale=.7}]
\draw[white, fill=lightgray] (1.5,11.5) --++ (9,0) -- ++(0,-5) -- ++(-1,0) -- ++(0,-1) -- ++(-1,0) -- ++(0,-1) -- ++(-1,0) -- ++(0,1) -- ++(-1,0) -- ++(0,1) -- ++(-1,0) -- ++(0,1) -- ++(-1,0) -- ++(0,1) -- ++(-1,0) -- ++(0,1) -- ++(-1,0) -- ++(0,1) -- ++(-1,0) -- cycle;
\draw[line width=6pt, lightgray, line cap=round] (11,11) -- ++(0,0)
(11,10) -- ++(0,0)
(11,9) -- ++(0,0)
(11,8) -- ++(0,0);
\draw[line width=3pt, lightgray] (11,7) -- ++(0,-1)
(10,6) -- ++(0,-1) -- ++(1,0)
(9,5) -- ++(0,-1) -- ++(2,0) node [corner] {} -- ++(0,-2);
\foreach \x in {2,...,11}{\foreach \y in {\x,...,11}{\node [dot] at (13-\x,\y) {};}}
\node[scale=1.3] at (6,2) {$k=7$};
\draw [densely dotted] 
(1.5, 11.5) -- ++(10,0) -- ++(0,-10) -- ++(-1,0) -- ++(0,1) -- ++(-1,0) -- ++(0,1) -- ++(-1,0) -- ++(0,1) -- ++(-1,0) -- ++(0,1) -- ++(-1,0) -- ++(0,1) -- ++(-1,0) -- ++(0,1) -- ++(-1,0) -- ++(0,1) -- ++(-1,0) -- ++(0,1) -- ++(-1,0) -- ++(0,1) -- ++(-1,0) -- ++(0,1);
\end{tikzpicture}
\end{equation}

Observe that, due to the nonintersecting property of the lattice paths, for any set of endpoints $\mathbf{E} = \{(i_t, n) : 1 \leq t \leq k \} \in \E$ where the $i_t$'s are listed in increasing order, the maximum attainable value $\hat{\imath}_t$ for $i_t$ is given by
\begin{equation}
    \label{i hat t}
   \hat{\imath}_t = \begin{cases}
    t, & t < 2(k+1) - n,\\
    n - 2(k-t) - 1 & \text{otherwise}.
\end{cases}
\end{equation}
The topmost case in~\eqref{i hat t} corresponds to those $t$'s (whenever $k$ is sufficiently large) for which $\mathbf{L}_t$ is necessarily the singleton $\{(t, b_t)\}$; for example, in the right-hand example in~\eqref{Sp F}, this occurs for $t < 2(k+1) - n = 5$, and indeed the four topmost lattice paths are singletons.
We now define the ``end'' map~\eqref{end map} as follows, for $T \in \mathcal{T}(\sigma) = \SSYT(\sigma, n)$:
\begin{equation}
    \label{end T details Sp}
    {\rm end}(T) \coloneqq \{ (i_t, n) : 1 \leq t \leq k\}, \quad \text{where} \quad i_t = \min\{ T_{t,1}, \: \hat{\imath}_t \}.
\end{equation}
Since $T_{t,1}$ is undefined for $t > \ell(\sigma)$, in this range the equation~\eqref{end T details Sp} reduces to $i_t = \hat{\imath}_t$.
As an example, in~\eqref{Sp end} below, for a given tableau $T$ we indicate ${\rm end}(T)$ by highlighting its elements in red inside $\delta(\mathbf{P})$.
(See also Example~\ref{ex:GL toy}.)
\begin{equation} 
    \label{Sp end} 
    \tikzstyle{dot}=[circle,fill=black, minimum size = 4.5pt, inner sep=0pt]
    \tikzstyle{redglow}=[circle,fill=red, minimum size = 7.5pt, inner sep=0pt]
\begin{tikzpicture}          [scale=.3,baseline=(current bounding box.center),every node/.style={scale=.7}]
\node [redglow] at (0,2) {};
\node [redglow] (E3) at (0,4) {};
\node [redglow] (E2) at (0,6) {};
\node [redglow] (E1) at (0,10) {};
\foreach \y in {2,...,11}{\node [dot] at (0,\y) {};}
\draw[densely dotted, thick, cyan]
(-.5,11.5) -- ++ (1,0) -- ++(0,-10) -- ++(-1,0) -- cycle;
\node [scale=1.3, left] at (-4,6) {$T = \ytableaushort{2355,89,9}$};
\node[left] at (-.5,8) {$\hat{\imath}_1 \rightarrow$} ;
\node[left] at (-.5,6) {$\hat{\imath}_2 \rightarrow$} ;
\node[left] at (-.5,4) {$\hat{\imath}_3 \rightarrow$} ;
\node[left] at (-.5,2) {$\hat{\imath}_4 \rightarrow$} ;
\node[right] (T1) at (2,10) {$\ytableaushort{2} \leftarrow T_{1,1}$};
\node[right] (T2) at (2,4) {$\ytableaushort{8} \leftarrow T_{2,1}$};
\node[right] (T3) at (2,3) {$\ytableaushort{9} \leftarrow T_{3,1}$};
\draw [ultra thick, lightgray] (T1.west) -- (E1);
\draw [ultra thick, lightgray] (T2.west) -- (E2);
\draw [ultra thick, lightgray] (T3.west) -- (E3);
\node[scale=1.3] at (0,0) {$k=4$};
\node[scale=1.3, red] at (0,13) {${\rm end}(T)$};
\end{tikzpicture}
\qquad\qquad\qquad\qquad\begin{tikzpicture}          [scale=.3,baseline=(current bounding box.center),every node/.style={scale=.7}]
\node [redglow] (E1) at (0,11) {};
\node [redglow] (E2) at (0,10) {};
\node [redglow] (E3) at (0,9) {};
\node [redglow] (E4) at (0,8) {};
\node [redglow] (E5) at (0,6) {};
\node [redglow] (E6) at (0,4) {};
\node [redglow] (E7) at (0,3) {};
\foreach \y in {2,...,11}{\node [dot] at (0,\y) {};}
\draw[densely dotted, thick, cyan]
(-.5,11.5) -- ++ (1,0) -- ++(0,-10) -- ++(-1,0) -- cycle;
\node [scale=1.3, left] at (-4,6) {$T = \ytableaushort{235,468,57,68,79,8,9}$};
\node[left] at (-.5,11) {$\hat{\imath}_1 \rightarrow$} ;
\node[left] at (-.5,10) {$\hat{\imath}_2 \rightarrow$} ;
\node[left] at (-.5,9) {$\hat{\imath}_3 \rightarrow$} ;
\node[left] at (-.5,8) {$\hat{\imath}_4 \rightarrow$} ;
\node[left] at (-.5,6) {$\hat{\imath}_5 \rightarrow$} ;
\node[left] at (-.5,4) {$\hat{\imath}_6 \rightarrow$} ;
\node[left] at (-.5,2) {$\hat{\imath}_7 \rightarrow$} ;
\node[right] (T1) at (2,10) {$\ytableaushort{2} \leftarrow T_{1,1}$};
\node[right] (T2) at (2,8) {$\ytableaushort{4} \leftarrow T_{2,1}$};
\node[right] (T3) at (2,7) {$\ytableaushort{5} \leftarrow T_{3,1}$};
\node[right] (T4) at (2,6) {$\ytableaushort{6} \leftarrow T_{4,1}$};
\node[right] (T5) at (2,5) {$\ytableaushort{7} \leftarrow T_{5,1}$};
\node[right] (T6) at (2,4) {$\ytableaushort{8} \leftarrow T_{6,1}$};
\node[right] (T7) at (2,3) {$\ytableaushort{9} \leftarrow T_{7,1}$};
\draw [ultra thick, lightgray] (T1.west) -- (E1);
\draw [ultra thick, lightgray] (T2.west) -- (E2);
\draw [ultra thick, lightgray] (T3.west) -- (E3);
\draw [ultra thick, lightgray] (T4.west) -- (E4);
\draw [ultra thick, lightgray] (T5.west) -- (E5);
\draw [ultra thick, lightgray] (T6.west) -- (E6);
\draw [ultra thick, lightgray] (T7.west) -- (E7);
\node[scale=1.3] at (0,0) {$k=7$};
\node[scale=1.3, red] at (0,13) {${\rm end}(T)$};
\end{tikzpicture}
\end{equation}

\subsection*{Details for the general linear group}

Let $H = \GL_k$, with $\mathbf{P}$ as in~\eqref{table P fij}.
Thus $\mathbf{P}$ is depicted as a $p \times q$ rectangular grid.
Let
\begin{align*}
    \mathbf{A} &\coloneqq \{(i,j) \in \mathbf{P} : i + j \leq k + 1\},\\
    \delta(\mathbf{A}) &\coloneqq \{(i,j) \in \mathbf{A} : i + j = k + 1 \text{ or } i = p \text{ or } j = q \},\\
    \delta(\mathbf{P}) & \coloneqq \{(i,j) \in \mathbf{P} : i = p \text{ or } j = q \}.
\end{align*}
Note that $\delta$ denotes the southeastern boundary.
Below are three examples, where $p=7$ and $q=10$:
\begin{equation*} 
    \label{GL anatomy} 
    \tikzstyle{dot}=[circle,fill=black, minimum size = 4.5pt, inner sep=0pt]
    \tikzstyle{redglow}=[circle,fill=red, minimum size = 7.5pt, inner sep=0pt]
\begin{tikzpicture}          [scale=.3,baseline=(current bounding box.center),every node/.style={scale=.7}]
\node [redglow] at (4,7) {};
\node [redglow] at (3,6) {};
\node [redglow] at (2,5) {};
\node [redglow] at (1,4) {};
\draw[densely dotted, thick]
(.5,7.5) -- ++(4,0) --++(0,-1) -- ++(-1,0) --++(0,-1) -- ++(-1,0) --++(0,-1) -- ++(-1,0) --++(0,-1) -- ++(-1,0) -- cycle;
\foreach \x in {1,...,10}{\foreach \y in {1,...,7}{\node [dot] at (\x,\y) {};}}
\draw[densely dotted, thick, cyan]
(9.5,7.5) -- ++ (1,0) -- ++(0,-7) -- ++(-10,0) -- ++(0,1) -- ++(9,0) -- cycle;
\node [right, scale=1.3,cyan] at (10.5, 1) {$\delta(\mathbf{P})$};
\node[above,scale=1.3] at (1,7.5) {$\mathbf{A}$};
\node[above,scale=1.3,red] at (4,7.5) {$\delta(\mathbf{A})$};
\node[below,scale=1.3] at (5.5,0) {$k=4$};
\end{tikzpicture}
\quad
\begin{tikzpicture}          [scale=.3,baseline=(current bounding box.center),every node/.style={scale=.7}]
\node [redglow] at (8,7) {};
\node [redglow] at (7,6) {};
\node [redglow] at (6,5) {};
\node [redglow] at (5,4) {};
\node [redglow] at (4,3) {};
\node [redglow] at (3,2) {};
\node [redglow] at (2,1) {};
\node [redglow] at (1,1) {};
\draw[densely dotted, thick]
(.5,7.5) -- ++(8,0) --++(0,-1) -- ++(-1,0) --++(0,-1) -- ++(-1,0) --++(0,-1) -- ++(-1,0) --++(0,-1) -- ++(-1,0) --++(0,-1) -- ++(-1,0) --++(0,-1) -- ++(-1,0) --++(0,-1) -- ++(-2,0) -- cycle;
\foreach \x in {1,...,10}{\foreach \y in {1,...,7}{\node [dot] at (\x,\y) {};}}
\draw[densely dotted, thick, cyan]
(9.5,7.5) -- ++ (1,0) -- ++(0,-7) -- ++(-10,0) -- ++(0,1) -- ++(9,0) -- cycle;
\node [right, scale=1.3,cyan] at (10.5, 1) {$\delta(\mathbf{P})$};
\node[above,scale=1.3] at (1,7.5) {$\mathbf{A}$};
\node[above,scale=1.3,red] at (8,7.5) {$\delta(\mathbf{A})$};
\node[below,scale=1.3] at (5.5,0) {$k=8$};
\end{tikzpicture}
\quad
\begin{tikzpicture}          [scale=.3,baseline=(current bounding box.center),every node/.style={scale=.7}]
\node [redglow] at (10,7) {};
\node [redglow] at (10,6) {};
\node [redglow] at (10,5) {};
\node [redglow] at (9,4) {};
\node [redglow] at (8,3) {};
\node [redglow] at (7,2) {};
\node [redglow] at (6,1) {};
\node [redglow] at (5,1) {};
\node [redglow] at (4,1) {};
\node [redglow] at (3,1) {};
\node [redglow] at (2,1) {};
\node [redglow] at (1,1) {};
\draw[densely dotted, thick]
(.5,7.5) -- ++(10,0) --++(0,-3) -- ++(-1,0) --++(0,-1) -- ++(-1,0) --++(0,-1) -- ++(-1,0) --++(0,-1) -- ++(-1,0) --++(0,-1) -- ++(-6,0) -- cycle;
\foreach \x in {1,...,10}{\foreach \y in {1,...,7}{\node [dot] at (\x,\y) {};}}
\draw[densely dotted, thick, cyan]
(9.5,7.5) -- ++ (1,0) -- ++(0,-7) -- ++(-10,0) -- ++(0,1) -- ++(9,0) -- cycle;
\node [right, scale=1.3,cyan] at (10.5, 1) {$\delta(\mathbf{P})$};
\node[above,scale=1.3] at (1,7.5) {$\mathbf{A}$};
\node[above,scale=1.3,red] at (10,7.5) {$\delta(\mathbf{A})$};
\node[below,scale=1.3] at (5.5,0) {$k=12$};
\end{tikzpicture}
\end{equation*}

\noindent Given $\sigma \in \widehat{H}$, it will be useful to define
\begin{align}
    \label{k+ k-}
    \begin{split}
        k^+ & \coloneqq \max\{ \ell(\sigma^+), \: k - p \},\\
        k^- & \coloneqq k - k^+,
    \end{split}
\end{align}
so that $k = k^+ + k^-$.
The purpose of~\eqref{k+ k-} is to divide $\delta(\mathbf{A})$ into two sides:
due to the upcoming definition of ${\rm end}(T)$ in~\eqref{end T GL details}, it will turn out that $(\F, T) \in \mathcal{J}(\sigma)$ only if the $k^+$ many southernmost lattice paths in $\F$ have endpoints along the southern edge of $\delta(\mathbf{P})$, and the remaining $k^-$ many easternmost lattice paths have endpoints along the eastern edge of $\delta(\mathbf{P})$.
In particular, if $k-p$ is positive, then it equals the number of starting points along the southern edge of $\delta(\mathbf{A}) \cap \delta(\mathbf{P})$ that are necessarily their own endpoints.
Explicitly, then, we have
\[
\delta(\mathbf{A}) = \{(a_u, u) : 1 \leq u \leq k^+ \} \sqcup \{(t, b_t) : 1 \leq t \leq k^- \},
\]
where
\[
    a_u = 
    \begin{cases}
        p, & u \leq k-p,\\
        k-u+1 & \text{otherwise},
    \end{cases} \qquad\text{and}\qquad
    b_t = 
    \begin{cases}
        q, & t \leq k-q,\\
        k-t+1 & \text{otherwise}.
    \end{cases}
\]
By~\eqref{F}, for each $\F \in \mathcal{F}$ we have $\F = (\mathbf{A} \backslash \delta(\mathbf{A})) \sqcup \mathbf{L}$, for a (not necessarily unique) disjoint union
\[
\mathbf{L} = \mathbf{L}^+_1 \sqcup \cdots \sqcup \mathbf{L}^+_{k^+} \sqcup \mathbf{L}^-_1 \sqcup \cdots \sqcup \mathbf{L}^-_{k^-}
\]
where each $\mathbf{L}^+_u$ is a southeast lattice path from $(a_u, u)$ to $\delta(\mathbf{P})$, and $\mathbf{L}^-_t$ is a southeast lattice path from $(t,b_t)$ to $\delta(\mathbf{P})$.
For each of the $k$ lattice paths, a \emph{corner} is defined as it was for $H = \Sp_{2k}$ in~\eqref{corner Sp}.
In the diagrams below, we show an element $\F \in \mathcal{F}$ for each of our previous examples:
\begin{equation*} 
    \label{GL F} 
    \tikzstyle{dot}=[circle,fill=black, minimum size = 4.5pt, inner sep=0pt]
    \tikzstyle{corner}=[rectangle,draw=black,thin, minimum size = 8pt, inner sep=2pt]
\begin{tikzpicture}          [scale=.3,baseline=(current bounding box.center),every node/.style={scale=.7}]
\draw[white, fill=lightgray]
(.5,7.5) -- ++(3,0) --++(0,-1) -- ++(-1,0) --++(0,-1) -- ++(-1,0) --++(0,-1) -- ++(-1,0) -- cycle;
\draw[line width=3pt, lightgray, line cap=round] 
(1,4) --++(0,-1) --++(1,0) node [corner] {} -- ++(0,-2) -- ++(2,0)
(2,5) -- ++(1,0) node [corner] {} -- ++(0,-3) -- ++(2,0) -- ++(0,-1) -- ++(2,0)
(3,6) -- ++(2,0) node [corner] {} -- ++(0,-1) -- ++(2,0) node [corner] {} -- ++(0,-2) -- ++(1,0) node [corner] {} -- ++(0,-1) -- ++(1,0) node [corner] {} -- ++(0,-1)
(4,7) -- ++(4,0) node [corner] {} -- ++(0,-1) -- ++(1,0) node[corner] {} --++(0,-1) -- ++ (1,0);
\foreach \x in {1,...,10}{\foreach \y in {1,...,7}{\node [dot] at (\x,\y) {};}}

\node[below,scale=1.3] at (5.5,0) {$k=4$};
\draw [densely dotted] (.5,7.5) rectangle (10.5,.5);
\end{tikzpicture}
\qquad\qquad
\begin{tikzpicture}          [scale=.3,baseline=(current bounding box.center),every node/.style={scale=.7}]
\draw[white, fill=lightgray]
(.5,7.5) -- ++(7,0) --++(0,-1) -- ++(-1,0) --++(0,-1) -- ++(-1,0)--++(0,-1) -- ++(-1,0)--++(0,-1) -- ++(-1,0)--++(0,-1) -- ++(-1,0)--++(0,-1) -- ++(-2,0) -- cycle;
\draw[line width=6pt, lightgray, line cap=round]
(1,1) --++(0,0);
\draw[line width=3pt, lightgray, line cap=round] 
(2,1) --++(2,0)
(3,2) -- ++(2,0) -- ++(0,-1) -- ++(2,0)
(4,3) -- ++(2,0) -- ++(0,-1) -- ++(3,0) node [corner] {} -- ++(0,-1) -- ++(1,0)
(5,4) -- ++(3,0) node [corner] {} -- ++(0,-1) -- ++(2,0) node[corner] {} --++(0,-1)
(6,5) -- ++(4,0)
(7,6) -- ++(3,0)
(8,7) -- ++(2,0);
\foreach \x in {1,...,10}{\foreach \y in {1,...,7}{\node [dot] at (\x,\y) {};}}

\node[below,scale=1.3] at (5.5,0) {$k=8$};
\draw [densely dotted] (.5,7.5) rectangle (10.5,.5);
\end{tikzpicture}
\qquad\qquad
\begin{tikzpicture}          [scale=.3,baseline=(current bounding box.center),every node/.style={scale=.7}]
\draw[white, fill=lightgray]
(.5,7.5) -- ++(9,0) --++(0,-3) -- ++(-1,0) --++(0,-1) -- ++(-1,0)--++(0,-1) -- ++(-1,0)--++(0,-1) -- ++(-6,0) -- cycle;
\draw[line width=6pt, lightgray, line cap=round]
(1,1) --++(0,0)
(2,1) --++(0,0)
(3,1) --++(0,0)
(4,1) --++(0,0)
(5,1) --++(0,0)
(10,7) --++(0,0)
(10,6) --++(0,0)
(10,5) --++(0,0);
\draw[line width=3pt, lightgray, line cap=round] 
(6,1) -- ++(1,0)
(7,2) -- ++(1,0) --++ (0,-1)
(8,3) --++(2,0) node [corner] {} --++(0,-1)
(9,4) --++(1,0);
\foreach \x in {1,...,10}{\foreach \y in {1,...,7}{\node [dot] at (\x,\y) {};}}
\node[below,scale=1.3] at (5.5,0) {$k=12$};
\draw [densely dotted] (.5,7.5) rectangle (10.5,.5);
\end{tikzpicture}
\end{equation*}

\noindent In the $k=8$ example above, note that $\F$ can be decomposed into a different family of lattice paths:
specifically, we could extend the fourth path (counting from the north) all the way to the southeast corner of $\mathbf{P}$, and shorten the fifth path accordingly.
Nonetheless, in general, $\cor(\F)$ is well defined regardless of the lattice path decomposition one chooses.

Observe that, due to the nonintersecting condition on the lattice paths $\mathbf{L}^+_u$ and $\mathbf{L}^-_t$, each set of endpoints $\mathbf{E} \in \E$ has the following property: 
for all $1 \leq \ell \leq \min\{k,p,q\}$, the $\ell \times \ell$ square in the lower-right corner of $\mathbf{P}$ contains at most $\ell$ many points in $\mathbf{E}$.
That is, for $\mathbf{E} = \{(p,j_u) : 1 \leq u \leq k^+\} \sqcup \{(i_t, q) : 1 \leq t \leq k^-\}$, we have the tight inequalities
\begin{equation} 
    \label{endpoint spacing GL}
    \begin{array}{c} 
    \#\{ u : j-u > q - \ell \} + \#\{ t: i_t > p - \ell \} \leq \ell, \\[1ex]
    \text{for all } 1 \leq \ell \leq \#\{(i,j) \in \mathbf{P} : i+j = k+1 \}.
    \end{array}
\end{equation}
Consequently, upon fixing the $j_u$'s, the maximum attainable value $\hat{\imath}_t$ of $i_t$ (for all $1 \leq t \leq k^-$) is given by
\[
\hat{\imath}_t = t \text{ if } t < k-q,
\]
while if $t \geq k-q$, then the $\hat{\imath}_t$'s are the largest elements $\hat{\imath}_{\min\{1, \: k-q\}} < \cdots < \hat{\imath}_{k^-}$ of the set
\begin{equation}
    \label{i hat GL}
    \Big\{ i \in [p] : q-p+i \notin \{j_u : 1 \leq u \leq k^+ \} \Big\}.
\end{equation}
In other words, for $\max\{1, k-q\} \leq t \leq k^-$, the $\hat{\imath}_t$'s are the distinct elements chosen from $[p]$ to be as large as possible without violating~\eqref{endpoint spacing GL}; see the example~\eqref{GL end} below.
As a result, note that the sets
\[
 Q \coloneqq \{q-j_u : 1 \leq u \leq k^+ \} \text{ and }P \coloneqq \{p-\hat{\imath}_t : \min\{1, \: k-q\} \leq t \leq k^- \}
\]
are disjoint, and
\begin{equation}
    \label{P Q}
    Q \sqcup P \supseteq \{0, 1, \ldots, p-\hat{\imath}_{\max\{1, \: k-q\}} \}.
\end{equation}
In light of this, we define the ``end'' map~\eqref{end map} as follows, for $T = (T^+, T^-) \in \mathcal{T}(\sigma) = \SSYT(\sigma^+, q) \times \SSYT(\sigma^-, p)$:
\[
    {\rm end}(T) \coloneqq \Big\{ (p, j_u) : 1 \leq u \leq k^+ \Big\} \sqcup \Big\{ (i_t, q) : 1 \leq t \leq k^- \Big\},
\]
where
\begin{equation}
    \label{end T GL details}
    j_u = \begin{cases}
        u, & u \leq k-p,\\
        T^+_{u,1} & \text{otherwise},
    \end{cases} \qquad \text{and} \qquad 
    i_t = \min\{ T^-_{t,1}, \: \hat{\imath}_t \}.
\end{equation}
Since $T^-_{t,1}$ is undefined for $t > \ell(\sigma^-)$, in this range the right-hand side of~\eqref{end T GL details} reduces to $i_t = \hat{\imath}_t$.
As an example below, we indicate ${\rm end}(T)$ by highlighting its elements in red inside $\delta(\mathbf{P})$:
\begin{equation} 
    \label{GL end}
    \tikzstyle{dot}=[circle,fill=black, minimum size = 4.5pt, inner sep=0pt]
    \tikzstyle{redglow}=[circle,fill=red, minimum size = 7.5pt, inner sep=0pt]
\begin{tikzpicture}          [scale=.3,baseline=(current bounding box.center),every node/.style={scale=.7}]
\node [redglow] (D1) at (1,1) {};
\node [redglow] (D2) at (4,1) {};
\node [redglow] (D3) at (7,1) {};
\node [redglow] (D4) at (9,1) {};
\node [redglow] (D5) at (10,1) {};
\node [redglow] (E1) at (10,7) {};
\node [redglow] (E2) at (10,5) {};
\node [redglow] (E3) at (10,3) {};
\foreach \x in {1,...,10}{\node [dot] at (\x,1) {};}
\foreach \y in {2,...,7}{\node [dot] at (10,\y) {};}
\draw[densely dotted, thick, cyan]
(9.5,7.5) -- ++ (1,0) -- ++(0,-7) -- ++(-10,0) -- ++(0,1) -- ++(9,0) -- cycle;
\draw [densely dashed] (8,1) -- ++(0,1.5)
(6,1) -- ++(0,4) -- ++(1.25,0)
(5,1) -- ++(0,5) -- ++(2.25,0);
\node [scale=1.3, left] at (0,4) {$\begin{aligned}T &= \left(\:\ytableaushort{348,469,78,99,{10}}, \: \ytableaushort{125,46,7,\none,\none}\:\right) \\[2ex] k &= 8 \\ k^+ &= 5 \\ k^- &= 3 \end{aligned}$};
\node[left] at (9.5,6) {$\hat{\imath}_1 \rightarrow$} ;
\node[left] at (9.5,5) {$\hat{\imath}_2 \rightarrow$} ;
\node[left] at (9.5,3) {$\hat{\imath}_3 \rightarrow$} ;
\node[right] (T1) at (12,7) {$\ytableaushort{1} \leftarrow T^-_{1,1}$};
\node[right] (T2) at (12,4) {$\ytableaushort{4} \leftarrow T^-_{2,1}$};
\node[right] (T3) at (12,1) {$\ytableaushort{7} \leftarrow T^-_{3,1}$};
\draw [ultra thick, lightgray] (T1.west) -- (E1);
\draw [ultra thick, lightgray] (T2.west) -- (E2);
\draw [ultra thick, lightgray] (T3.west) -- (E3);
\node[below,align=center] (U1) at (3,-1) {$\ytableaushort{3}$\\$\uparrow$\\$T^+_{1,1}$};
\node[below] (U2) at (4,-1) {$\ytableaushort{4}$};
\node[below] (U3) at (7,-1) {$\ytableaushort{7}$};
\node[below] (U4) at (9,-1) {$\ytableaushort{9}$};
\node[below,align=center] (U5) at (10,-1) {$\ytableaushort{{10}}$\\$\uparrow$\\$T^+_{5,1}$};
\draw [loosely dotted, thick] (4,-4) -- (9,-4);
\draw [ultra thick, lightgray] (U1.north) -- (D1);
\draw [ultra thick, lightgray] (U2.north) -- (D2);
\draw [ultra thick, lightgray] (U3.north) -- (D3);
\draw [ultra thick, lightgray] (U4.north) -- (D4);
\draw [ultra thick, lightgray] (U5.north) -- (D5);
\node[scale=1.3, red] at (10,9) {${\rm end}(T)$};
\end{tikzpicture}
\end{equation}

\noindent The dashed lines in~\eqref{GL end} are meant to illustrate the definition of the $\hat{\imath}_t$'s in~\eqref{i hat GL}.
Specifically, on the southern edge of $\delta(\mathbf{P})$, take the easternmost points that are \emph{not} in the set $\{(p, j_u) : 1 \leq u \leq k^+\}$; then the $\hat{\imath}_t$'s are their corresponding points along the eastern edge of $\delta(\mathbf{P})$.
See also Example~\ref{ex:GL toy}.

\subsection*{Details for the orthogonal group}

Let $H = \O_k$, with $\mathbf{P}$ as in~\eqref{table P fij}.
Thus $\mathbf{P}$ is depicted as an upper triangular staircase with side length $n$.
Let
\begin{align*}
\mathbf{A} = \delta(\mathbf{A}) &\coloneqq \{(i,j) \in \mathbf{P} : i=j \},\\
\delta(\mathbf{P}) & \coloneqq \{ (i,j) \in \mathbf{P} : j=n \}.
\end{align*}
Note that $\delta$ denotes the eastern boundary.
Note also that, unlike the $\GL_k$ and $\Sp_{2k}$ cases above, $\mathbf{A}$ is defined independently of $k$, and we have $\# \delta(\mathbf{A}) = n$ rather than $k$.
By~\eqref{F}, since $\mathbf{A} = \delta(\mathbf{A})$, for each $\F \in \mathcal{F}$ we have a unique disjoint union $\F = \mathbf{L}_1 \sqcup \cdots \sqcup \mathbf{L}_k$ of northeast lattice paths from $\delta(\mathbf{A})$ to $\delta(\mathbf{P})$, such that each $\mathbf{L}_t$ has starting point $(a_t,a_t)$, for some $1 \leq a_1 < \cdots < a_k \leq n$.
For each $1 \leq t \leq k$, a point $(i,j) \in \mathbf{L}_t$ is said to be a \emph{corner} of $\F$ if the following is true:
\begin{equation}
\label{corner O}
    (i-1,j) \in \mathbf{L}_t \text{ and } (i+1,j) \notin \mathbf{L}_t.
\end{equation}
In other words, a corner either is a point at which a lattice path turns from east to north, or is the starting point of a lattice path whose first step is north.

For $1 \leq t \leq k$, set
\[
\hat{\imath}_t \coloneqq k-t+1.
\]
Recall that we state our main result for $H = \O_k$ only in the case where $\sigma = (1^m)$ for some $0 \leq m \leq k$;
for this reason, in~\eqref{phi O}, we made the identification $\mathcal{T}(\sigma) = \binom{[n]}{m}$.
For $T \in \binom{[n]}{m}$  with elements labeled in decreasing order $T_1 > \cdots > T_m$, define
\begin{equation}
    \label{end T details O}
    {\rm end}(T) \coloneqq \{(i_t, n) : 1 \leq t \leq k\}, \quad \text{where } i_t = \max\{ T_t, \hat{\imath}_t\}.
\end{equation}
Note that for $t > m$ this reduces to $i_t = \hat{\imath}_t$.
Below is an example where $n=6$, $k=4$, and $m=3$:
\begin{equation*} 
    \tikzstyle{dot}=[circle,fill=black, minimum size = 4.5pt, inner sep=0pt]
    \tikzstyle{redglow}=[circle,fill=red, minimum size = 7.5pt, inner sep=0pt]
    \tikzstyle{corner}=[rectangle,draw=black,thin, minimum size = 8pt, inner sep=2pt]
\begin{tikzpicture}          [scale=.3,baseline=(current bounding box.south),every node/.style={scale=.7}]
\foreach \x in {2,...,7}{\node [redglow] at (9-\x,\x) {};}
\node [left,scale=1.3] at (4.5,3) {$\mathbf{A} = \textcolor{red}{\delta(\mathbf{A})}$};
\foreach \x in {2,...,7}{\foreach \y in {\x,...,7}{\node [dot] at (9-\x,\y) {};}}
\draw[densely dotted, thick, cyan]
(6.5,7.5) -- ++ (1,0) -- ++(0,-6) -- ++(-1,0) -- cycle;
\node [scale=1.3,cyan] at (9, 5) {$\delta(\mathbf{P})$};
\end{tikzpicture}
\qquad\qquad
\begin{tikzpicture}          [scale=.3,baseline=(current bounding box.south),every node/.style={scale=.7}]
\draw[line width=3pt, lightgray]
(3,6) node [corner] {} -- ++(0,1) -- ++ (4,0)
(4,5) -- ++ (2,0) node [corner] {} -- ++(0,1) -- ++(1,0)
(5,4) -- ++ (2,0)
(7,2) node [corner] {} -- ++(0,1);
\foreach \x in {2,...,7}{\foreach \y in {\x,...,7}{\node [dot] at (9-\x,\y) {};}}
\node [left,scale=1.3,gray] at (4.5,3) {$\mathbf{F} \in \mathcal{F}$};
\draw [densely dotted] 
(1.5, 7.5) -- ++(6,0) -- ++(0,-6) -- ++(-1,0) -- ++(0,1) -- ++(-1,0) -- ++(0,1) -- ++(-1,0) -- ++(0,1) -- ++(-1,0) -- ++(0,1) -- ++(-1,0) -- ++(0,1) -- ++(-1,0) -- ++(0,1);
\end{tikzpicture}
\qquad\qquad\qquad
\begin{tikzpicture}          [scale=.3,baseline=(current bounding box.south),every node/.style={scale=.7}]
\node [redglow] at (0,6) {};
\node [redglow] (E1) at (0,5) {};
\node [redglow] (E2) at (0,3) {};
\node [redglow] (E3) at (0,2) {};
\foreach \y in {1,...,6}{\node [dot] at (0,\y) {};}
\draw[densely dotted, thick, cyan]
(-.5,6.5) -- ++ (1,0) -- ++(0,-6) -- ++(-1,0) -- cycle;
\node [scale=1.3, left] at (-5,4) {$T = \ytableaushort{1,4,5}$};
\node[right] (T1) at (2,6) {$\ytableaushort{1}$};
\node[right] (T2) at (2,3) {$\ytableaushort{4}$};
\node[right] (T3) at (2,2) {$\ytableaushort{5}$};
\draw [ultra thick, lightgray] (T1.west) -- (E1);
\draw [ultra thick, lightgray] (T2.west) -- (E2);
\draw [ultra thick, lightgray] (T3.west) -- (E3);
\node[scale=1.3, red] at (0,8) {${\rm end}(T)$};
\node[left] at (-.5,3) {$\hat{\imath}_1 \rightarrow$};
\node[left] at (-.5,4) {$\hat{\imath}_2 \rightarrow$};
\node[left] at (-.5,5) {$\hat{\imath}_3 \rightarrow$};
\node[left] at (-.5,6) {$\hat{\imath}_4 \rightarrow$};
\end{tikzpicture}
\end{equation*}

\subsection*{Main result}

We will prove the following theorem in Section~\ref{sec:proofs}.
For the sake of concreteness, note that we have the following:
\begin{equation}
\label{size P delta P}
\#\mathbf{P} = \begin{cases}
    pq, & H = \GL_k,\\
    \binom{n+1}{2}, & H = \O_k,\\
    \binom{n}{2}, & H = \Sp_{2k},
\end{cases}
\qquad \qquad
\#\delta(\mathbf{P}) = \begin{cases}
    p+q-1, & H = \GL_k,\\
    n, & H = \O_k,\\
    n-1, & H = \Sp_{2k}.
\end{cases}
\end{equation}

\begin{thm}
    \label{thm:Stanley decomps and HS}
    Let $H = \GL_k$, $\O_k$, or $\Sp_{2k}$.
    Let $\sigma \in \widehat{H}$ as in~\eqref{H hat}; but if $H = \O_k$, then let $\sigma = (1^m)$ for some $0 \leq m \leq k$.
    Let $M_\sigma$ be the module of covariants defined in~\eqref{M sigma}, with $f_{ij}$ as in~\eqref{table P fij}.
    
    If $k < \#\delta(\mathbf{P})$, then we have a Stanley decomposition
    \[
    M_\sigma = \bigoplus_{(\F, T) \in \mathcal{J}(\sigma)} \C[ f_{ij} : (i,j) \in \F] \: f_{\cor(\F)} \cdot \phi_T,
    \]
    where $\mathcal{J}(\sigma)$ is the set of jellyfish in Definition~\ref{def:jellyfish}, $f_{\cor(\F)}$ is the product of $f_{ij}$'s defined in~\eqref{f cor}, and $\phi_T : W \longrightarrow U_\sigma$ is the map defined in~\eqref{phi_T GL Sp} or \eqref{phi O}.
    
    If $k \geq \#\delta(\mathbf{P})$, then $M_\sigma$ is a free module over $\C[W]^H$, with Stanley decomposition
    \[
    M_\sigma = \bigoplus_{T \in \mathcal{T}(\sigma)} \C[f_{ij} : (i,j) \in \mathbf{P}]\: \phi_T.
    \]
\end{thm}

\begin{rema}
    In the $k < \#\delta(\mathbf{P})$ case of Theorem~\ref{thm:Stanley decomps and HS} for $H = \GL_k$ or $\Sp_{2k}$, it is still possible that $M_\sigma$ is free.
    This occurs if and only if $\mathcal{F}_{\mathbf{E}} = \{ \mathbf{P} \}$ for all $\mathbf{E} \in \E$, if and only if
    \[
    \begin{cases}
        \ell(\sigma^+) \leq k-p \text{ or } \ell(\sigma^-) \leq k-q, & H = \GL_k,\\
        \ell(\sigma) \leq 2k-n+1, & H = \Sp_{2k}.
    \end{cases}
    \]
\end{rema}

\begin{coro}
    \label{cor:HS}
    Assume the setting of Theorem~\ref{thm:Stanley decomps and HS}.
    For $\mathbf{E} \in \E$, let
    \begin{align*}
        d_{\mathbf{E}} & \coloneqq \textup{cardinality of any (equivalently, every) $\F \in \mathcal{F}_{\mathbf{E}}$},\\
        P_{\mathbf{E}}(t) &\coloneqq \dfrac{\sum_{\F \in \mathcal{F}_{\mathbf{E}}} (t^2)^{\# \cor(\F)}}{(1-t^2)^{d_{\mathbf{E}}}}.
    \end{align*}
    If $k < \#\delta(\mathbf{P})$, then we have the Hilbert--Poincar\'e series   
    \[
    P(M_\sigma; t) = t^{|\sigma|} \sum_{\mathbf{E} \in \E} \# \mathcal{T}_{\mathbf{E}}(\sigma) \cdot  P_{\mathbf{E}}(t).
    \]
    If $k \geq \#\delta(\mathbf{P})$, then we have the Hilbert--Poincar\'e series
    \[
    P(M_\sigma; t) = \frac{\# \mathcal{T}(\sigma) \cdot t^{|\sigma|}}{(1-t^2)^{\#\mathbf{P}}}.
    \]
    
\end{coro}

\begin{proof}
    Suppose $k < \#\delta(\mathbf{P})$.
    By Definition~\ref{def:jellyfish}, we have $\mathcal{J}(\sigma) = \coprod_{\mathbf{E} \in \E} \mathcal{F}_{\mathbf{E}} \times \mathcal{T}_{\mathbf{E}}(\sigma)$.
    Therefore the Stanley decomposition in Theorem~\ref{thm:Stanley decomps and HS} can be expanded as follows:
    \begin{equation}
    \label{triple decomp}
    M_\sigma = \bigoplus_{\mathbf{E} \in \E} \left( \bigoplus_{\F \in \mathcal{F}_{\mathbf{E}}} \: \bigoplus_{T \in \mathcal{T}_{\mathbf{E}}(\sigma)} \C[f_{ij} : (i,j) \in \F] f_{\cor(\F)} \cdot \phi_T \right).
    \end{equation}
    We claim that for all $\mathbf{E} \in \E$, the elements $\F \in \mathcal{F}_{\mathbf{E}}$ all have the same cardinality $d_{\mathbf{E}}$.
    For $H = \GL_k$ and $\Sp_{2k}$, this follows from the fact that the length of a southeast lattice path is determined by its starting point and endpoint.
    For $H = \O_k$, where the starting points vary along the diagonal $\delta(\mathbf{A})$, we observe that for a fixed endpoint $(i, n)$, its distance to any possible starting point $(a,a)$, such that $a \geq i$, equals the constant $n-i$.
    Finally, since each $f_{ij}$ has degree 2, and $\phi_T$ has degree $|\sigma|$, the decomposition~\eqref{triple decomp} yields the Hilbert--Poincar\'e series
    \begin{align*}
    P(M_\sigma; t) &= \sum_{\mathbf{E} \in \E} \left( \#\mathcal{T}_{\mathbf{E}}(\sigma) \sum_{\F \in \mathcal{F}_{\mathbf{E}} } \frac{(t^2)^{\# \cor(\F)}}{(1-t^2)^{\#\F}} \cdot t^{|\sigma|} \right) \\
    &= \sum_{\mathbf{E} \in \E} \left( \#\mathcal{T}_{\mathbf{E}}(\sigma) \: \frac{\sum_{\F \in \mathcal{F}_{\mathbf{E}} } (t^2)^{\# \cor(\F)}}{(1-t^2)^{d_{\mathbf{E}}}} \cdot t^{|\sigma|} \right)  = t^{|\sigma|} \sum_{\mathbf{E} \in \E} \#\mathcal{T}_{\mathbf{E}}(\sigma) \cdot P_{\mathbf{E}}(t),
    \end{align*}
    by~\eqref{Hilbert series from Stanley decomp general}.
    The result for the case $k \geq \#\delta(\mathbf{P})$ is immediate from Theorem~\ref{thm:Stanley decomps and HS}.
\end{proof}

Although it is straightforward to program the various constructions used in this section, it is illustrative to work out the following small examples by hand.

\begin{exam}[Symplectic group]
    \label{ex:Sp toy}
    Let
    \[
    H = \Sp_{2k}, \qquad k=3, \qquad n=6, \qquad \sigma = (1,1) = \ydiagram{1,1}.
    \]
    Then we have 
    \[
    \mathcal{T}(\sigma) = \SSYT(\sigma,n) =  \left\{ \ytableaushort{1,2}, \ytableaushort{1,3}, \ytableaushort{1,4},\ytableaushort{1,5}, \ytableaushort{1,6},
    \ytableaushort{2,3},
    \ytableaushort{2,4},
    \ytableaushort{2,5},
    \ytableaushort{2,6},
    \ytableaushort{3,4},
    \ytableaushort{3,5},
    \ytableaushort{3,6},
    \ytableaushort{4,5},
    \ytableaushort{4,6},
    \ytableaushort{5,6}
    \right\}.
    \]
    Let $T \in \mathcal{T}(\sigma)$, and suppose that  ${\rm end}(T) = \mathbf{E} = \{ (i_t, 6) : 1 \leq t \leq 3 \}$, with $i_1 < i_2 < i_3$.
    By~\eqref{i hat t}, we have $(\hat{\imath}_1, \hat{\imath}_2, \hat{\imath}_3) = (1, 3, 5)$.
    Thus $\min\{T_{1,1}, \: \hat{\imath}_1\} = 1$, and so by~\eqref{end T details Sp} we must have $i_1 = 1$.
    Moreover, since $\ell(\sigma) = 2$, the entry $T_{3,1}$ is undefined, so by~\eqref{end T details Sp} we have $i_3 = \hat{\imath}_3 = 5$.
    Finally, since $i_2 = \min\{T_{2,1}, \: \hat{\imath}_2\} = \min\{T_{2,1}, \: 3\}$, and since in any tableau we have $T_{2,1} \geq 2$, we conclude that
    \[
    i_2 = \begin{cases}
        2, & T_{2,1} = 2,\\
        3 & \text{otherwise}.
    \end{cases}
    \]
    Hence either $\mathbf{E} = \mathbf{E}_{125} \coloneqq \{(1,6), (2,6), (5,6)\}$ or $\mathbf{E} = \mathbf{E}_{135} \coloneqq \{(1,6), (3,6), (5,6)\}$.
    Therefore, for all other $\mathbf{E} \in \E$ we have $\mathcal{T}_{\mathbf{E}}(\sigma) = \varnothing$, so we may restrict our attention to just these two:
    \tikzstyle{dot}=[circle,fill=black, minimum size = 4.5pt, inner sep=0pt]
    \tikzstyle{corner}=[rectangle,draw=black,thin, minimum size = 8pt, inner sep=2pt]
\[
\begin{array}{lll}
\mathcal{F}_{\mathbf{E}_{125}} = 
        \left\{
        \begin{tikzpicture}          [scale=.3,baseline=(current bounding box.center),every node/.style={scale=.7}]
\draw [white,fill=lightgray] (1.5,6.5) -- ++ (4,0) -- ++ (0,-2) -- ++ (-1,0) -- ++ (0,-1) -- ++ (-1,0) -- ++ (0,1) -- ++ (-1,0) -- ++ (0,1)  -- ++ (-1,0) -- cycle;
\draw[line width=5pt, lightgray, line cap=round]
(6,6) -- ++(0,0)
(6,5) -- ++(0,0);
\draw[line width=3pt, lightgray] (5,4) -- ++(0,-1) -- ++(1,0) -- ++(0,-1);
\foreach \x in {2,...,6}{\foreach \y in {\x,...,6}{\node [dot] at (8-\x,\y) {};}}
\draw [densely dotted] (1.5, 6.5) -- ++ (5,0) -- ++ (0,-5) -- ++(-1,0) -- ++ (0,1) -- ++(-1,0) -- ++ (0,1) -- ++(-1,0) -- ++ (0,1) -- ++(-1,0) -- ++ (0,1) -- ++(-1,0) -- ++ (0,1);
\end{tikzpicture}
\quad
\begin{tikzpicture}          [scale=.3,baseline=(current bounding box.center),every node/.style={scale=.7}]
\draw [white,fill=lightgray] (1.5,6.5) -- ++ (4,0) -- ++ (0,-2) -- ++ (-1,0) -- ++ (0,-1) -- ++ (-1,0) -- ++ (0,1) -- ++ (-1,0) -- ++ (0,1)  -- ++ (-1,0) -- cycle;
\draw[line width=5pt, lightgray, line cap=round]
(6,6) -- ++(0,0)
(6,5) -- ++(0,0);
\draw[line width=3pt, lightgray] (5,4) -- ++(1,0) node [corner] {} -- ++(0,-2);
\foreach \x in {2,...,6}{\foreach \y in {\x,...,6}{\node [dot] at (8-\x,\y) {};}}
\draw [densely dotted] (1.5, 6.5) -- ++ (5,0) -- ++ (0,-5) -- ++(-1,0) -- ++ (0,1) -- ++(-1,0) -- ++ (0,1) -- ++(-1,0) -- ++ (0,1) -- ++(-1,0) -- ++ (0,1) -- ++(-1,0) -- ++ (0,1);
\end{tikzpicture}
        \right\},
        & &
        \mathcal{T}_{\mathbf{E}_{125}}(\sigma) = \left\{ \ytableaushort{1,2} \right\},\\[6ex]
%%%%%
\mathcal{F}_{\mathbf{E}_{135}} = 
        \left\{
        \begin{tikzpicture}          [scale=.3,baseline=(current bounding box.center),every node/.style={scale=.7}]
\draw [white,fill=lightgray] (1.5,6.5) -- ++ (4,0) -- ++ (0,-2) -- ++ (-1,0) -- ++ (0,-1) -- ++ (-1,0) -- ++ (0,1) -- ++ (-1,0) -- ++ (0,1)  -- ++ (-1,0) -- cycle;
\draw[line width=5pt, lightgray, line cap=round]
(6,6) -- ++(0,0);
\draw[line width=3pt, lightgray] (5,4) -- ++(0,-1) -- ++(1,0) -- ++(0,-1)
(6,5) -- ++(0,-1);
\foreach \x in {2,...,6}{\foreach \y in {\x,...,6}{\node [dot] at (8-\x,\y) {};}}
\draw [densely dotted] (1.5, 6.5) -- ++ (5,0) -- ++ (0,-5) -- ++(-1,0) -- ++ (0,1) -- ++(-1,0) -- ++ (0,1) -- ++(-1,0) -- ++ (0,1) -- ++(-1,0) -- ++ (0,1) -- ++(-1,0) -- ++ (0,1);
\end{tikzpicture}
        \right\},
        & &
        \mathcal{T}_{\mathbf{E}_{135}}(\sigma) = \underbrace{\left\{
        \ytableaushort{1,3}, \ytableaushort{1,4},\ytableaushort{1,5}, \ldots,     \ytableaushort{4,5},
    \ytableaushort{4,6},
    \ytableaushort{5,6}
    \right\}}_{\mathclap{\text{the 14 remaining tableaux in $\mathcal{T}(\sigma)$}}}.        
\end{array}
\]
This gives us an explicit understanding of the set 
\[
\mathcal{J}(\sigma) = \Big(\mathcal{F}_{\mathbf{E}_{125}} \times \mathcal{T}_{\mathbf{E}_{125}}(\sigma) \Big) \sqcup \Big(\mathcal{F}_{\mathbf{E}_{135}} \times \mathcal{T}_{\mathbf{E}_{135}}(\sigma) \Big),
\]
which contains exactly $(2 \times 1) + (1 \times 14) = 16$ jellyfish.
To obtain the Hilbert--Poincar\'e series of $M_\sigma$ via Corollary~\ref{cor:HS}, we count the number of shaded points in any $\F \in \mathcal{F}_{\mathbf{E}}$ to obtain $d_{\mathbf{E}_{125}} = 14$ and $d_{\mathbf{E}_{135}} = 15$, and we count the corners to obtain the numerators as follows:
\[
\begin{array}{lll}
    P_{\mathbf{E}_{125}}(t) = \dfrac{(t^2)^0 + (t^2)^1}{(1-t^2)^{14}} = \dfrac{1+t^2}{(1-t^2)^{14}}, & & \#\mathcal{T}_{\mathbf{E}_{125}}(\sigma) = 1,\\[4ex]
    P_{\mathbf{E}_{135}}(t) = \dfrac{(t^2)^0}{(1-t^2)^{15}} = \dfrac{1}{(1-t^2)^{15}}, & & \#\mathcal{T}_{\mathbf{E}_{125}}(\sigma) = 14.
    \end{array}
\]
Thus since $|\sigma| = 2$, the Hilbert--Poincar\'e series is given by
\[
P(M_{\sigma}; t) = t^2 \left( \frac{1+t^2}{(1-t^2)^{14}} + \frac{14}{(1-t^2)^{15}} \right) = \frac{15t^2 - t^6}{(1-t^2)^{15}}.
\]

\end{exam}

\begin{exam}[General linear group]
    \label{ex:GL toy}
    Let
    \[
    H = \GL_k, \qquad k=3, \qquad p=3, \: q=4, \qquad \sigma = (1,0,-1).
    \]
    Note that $U_\sigma$ is the nontrivial component of the adjoint representation of $\GL_3$.
    We have $\mathcal{T}(\sigma) = \{ ( \ytableaushort{x}, \ytableaushort{y} ) : 1 \leq x \leq 4, \: 1 \leq y \leq 3 \}$.
    Since $k^+ = 1$ and $k^- = 2$, we will denote elements of $\E$ with the shorthand $\mathbf{E}_{abc} \coloneqq \{ (3,a), \: (b, 4), \: (c,4) \}$.
    There are six elements $\mathbf{E} \in \E$ such that $\mathcal{T}_{\mathbf{E}}(\sigma)$ is nonempty:
    \tikzstyle{dot}=[circle,fill=black, minimum size = 4.5pt, inner sep=0pt]
    \tikzstyle{corner}=[rectangle,draw=black,thin, minimum size = 8pt, inner sep=2pt]
\[
\begin{array}{lll}
\mathcal{F}_{\mathbf{E}_{113}} = 
        \Bigg\{
        \begin{tikzpicture}  [scale=.3,baseline=(current bounding box.center),every node/.style={scale=.7}]
\draw [white,fill=lightgray] (.5,3.5) -- ++ (2,0) -- ++ (0,-1) -- ++ (-1,0) -- ++ (0,-1) -- ++ (-1,0) -- cycle;
\draw[line width=5pt, lightgray, line cap=round]
(1,1) -- ++(0,0);
\draw[line width=3pt, lightgray] (2,2) -- ++(0,-1) -- ++(2,0)
(3,3) -- ++ (1,0);
\foreach \x in {1,...,4}{\foreach \y in {1,...,3}{\node [dot] at (\x,\y) {};}}

\draw [densely dotted] (.5,3.5) rectangle (4.5,.5);
\node at (0,0) {};
\end{tikzpicture}
\;
\begin{tikzpicture}  [scale=.3,baseline=(current bounding box.center),every node/.style={scale=.7}]
\draw [white,fill=lightgray] (.5,3.5) -- ++ (2,0) -- ++ (0,-1) -- ++ (-1,0) -- ++ (0,-1) -- ++ (-1,0) -- cycle;
\draw[line width=5pt, lightgray, line cap=round]
(1,1) -- ++(0,0);
\draw[line width=3pt, lightgray] (2,2) -- ++(1,0) node [corner] {} -- ++(0,-1) -- ++ (1,0)
(3,3) -- ++ (1,0);
\foreach \x in {1,...,4}{\foreach \y in {1,...,3}{\node [dot] at (\x,\y) {};}}

\draw [densely dotted] (.5,3.5) rectangle (4.5,.5);
\node at (0,0) {};
\end{tikzpicture}
\;
\begin{tikzpicture}  [scale=.3,baseline=(current bounding box.center),every node/.style={scale=.7}]
\draw [white,fill=lightgray] (.5,3.5) -- ++ (2,0) -- ++ (0,-1) -- ++ (-1,0) -- ++ (0,-1) -- ++ (-1,0) -- cycle;
\draw[line width=5pt, lightgray, line cap=round]
(1,1) -- ++(0,0);
\draw[line width=3pt, lightgray] (2,2) -- ++(2,0) node [corner] {} -- ++(0,-1)
(3,3) -- ++ (1,0);
\foreach \x in {1,...,4}{\foreach \y in {1,...,3}{\node [dot] at (\x,\y) {};}}

\draw [densely dotted] (.5,3.5) rectangle (4.5,.5);
\node at (0,0) {};
\end{tikzpicture}
\;
        \Bigg\},
        & &
        \mathcal{T}_{\mathbf{E}_{113}}(\sigma) = \Big\{ \left(\ytableaushort{1}, \ytableaushort{1} \right) \Big\},\\[3ex]
%%%%%%
\mathcal{F}_{\mathbf{E}_{123}} = 
        \Bigg\{
        \begin{tikzpicture}  [scale=.3,baseline=(current bounding box.center),every node/.style={scale=.7}]
\draw [white,fill=lightgray] (.5,3.5) -- ++ (2,0) -- ++ (0,-1) -- ++ (-1,0) -- ++ (0,-1) -- ++ (-1,0) -- cycle;
\draw[line width=5pt, lightgray, line cap=round]
(1,1) -- ++(0,0);
\draw[line width=3pt, lightgray] (2,2) -- ++(0,-1) -- ++(2,0)
(3,3) -- ++ (0,-1) -- ++ (1,0);
\foreach \x in {1,...,4}{\foreach \y in {1,...,3}{\node [dot] at (\x,\y) {};}}

\draw [densely dotted] (.5,3.5) rectangle (4.5,.5);
\node at (0,0) {};
\end{tikzpicture}
\;
\begin{tikzpicture}  [scale=.3,baseline=(current bounding box.center),every node/.style={scale=.7}]
\draw [white,fill=lightgray] (.5,3.5) -- ++ (2,0) -- ++ (0,-1) -- ++ (-1,0) -- ++ (0,-1) -- ++ (-1,0) -- cycle;
\draw[line width=5pt, lightgray, line cap=round]
(1,1) -- ++(0,0);
\draw[line width=3pt, lightgray] (2,2) -- ++(0,-1) -- ++(2,0)
(3,3) -- ++ (1,0) node [corner] {} -- ++ (0,-1);
\foreach \x in {1,...,4}{\foreach \y in {1,...,3}{\node [dot] at (\x,\y) {};}}

\draw [densely dotted] (.5,3.5) rectangle (4.5,.5);
\node at (0,0) {};
\end{tikzpicture}
\;
\begin{tikzpicture}  [scale=.3,baseline=(current bounding box.center),every node/.style={scale=.7}]
\draw [white,fill=lightgray] (.5,3.5) -- ++ (2,0) -- ++ (0,-1) -- ++ (-1,0) -- ++ (0,-1) -- ++ (-1,0) -- cycle;
\draw[line width=5pt, lightgray, line cap=round]
(1,1) -- ++(0,0);
\draw[line width=3pt, lightgray] (2,2) -- ++(1,0) node [corner] {} -- ++(0,-1) -- ++ (1,0)
(3,3) -- ++ (1,0) node [corner] {} -- ++ (0,-1);
\foreach \x in {1,...,4}{\foreach \y in {1,...,3}{\node [dot] at (\x,\y) {};}}

\draw [densely dotted] (.5,3.5) rectangle (4.5,.5);
\node at (0,0) {};
\end{tikzpicture}
\;
        \Bigg\},
        & &
        \mathcal{T}_{\mathbf{E}_{123}}(\sigma) = \Big\{ \left(\ytableaushort{1}, \ytableaushort{2} \right), \;  \left(\ytableaushort{1}, \ytableaushort{3} \right) \Big\},\\[3ex]
%%%%%%%
\mathcal{F}_{\mathbf{E}_{213}} = 
        \Bigg\{
        \begin{tikzpicture}  [scale=.3,baseline=(current bounding box.center),every node/.style={scale=.7}]
\draw [white,fill=lightgray] (.5,3.5) -- ++ (2,0) -- ++ (0,-1) -- ++ (-1,0) -- ++ (0,-1) -- ++ (-1,0) -- cycle;
\draw[line width=3pt, lightgray] 
(1,1) -- ++ (1,0)
(2,2) -- ++(1,0) -- ++(0,-1) -- ++(1,0)
(3,3) -- ++ (1,0);
\foreach \x in {1,...,4}{\foreach \y in {1,...,3}{\node [dot] at (\x,\y) {};}}

\draw [densely dotted] (.5,3.5) rectangle (4.5,.5);
\node at (0,0) {};
\end{tikzpicture}
\;
\begin{tikzpicture}  [scale=.3,baseline=(current bounding box.center),every node/.style={scale=.7}]
\draw [white,fill=lightgray] (.5,3.5) -- ++ (2,0) -- ++ (0,-1) -- ++ (-1,0) -- ++ (0,-1) -- ++ (-1,0) -- cycle;
\draw[line width=3pt, lightgray] 
(1,1) -- ++ (1,0)
(2,2) -- ++(2,0) node [corner] {} -- ++(0,-1)
(3,3) -- ++ (1,0);
\foreach \x in {1,...,4}{\foreach \y in {1,...,3}{\node [dot] at (\x,\y) {};}}

\draw [densely dotted] (.5,3.5) rectangle (4.5,.5);
\node at (0,0) {};
\end{tikzpicture}
\;
        \Bigg\},
        & &
        \mathcal{T}_{\mathbf{E}_{213}}(\sigma) = \Big\{ \left(\ytableaushort{2}, \ytableaushort{1} \right) \Big\},\\[3ex]
%%%%%%%
\mathcal{F}_{\mathbf{E}_{223}} = 
        \Bigg\{
        \begin{tikzpicture}  [scale=.3,baseline=(current bounding box.center),every node/.style={scale=.7}]
\draw [white,fill=lightgray] (.5,3.5) -- ++ (2,0) -- ++ (0,-1) -- ++ (-1,0) -- ++ (0,-1) -- ++ (-1,0) -- cycle;
\draw[line width=3pt, lightgray] 
(1,1) -- ++ (1,0)
(2,2) -- ++(1,0) -- ++(0,-1) -- ++(1,0)
(3,3) -- ++ (1,0) -- ++(0,-1);
\foreach \x in {1,...,4}{\foreach \y in {1,...,3}{\node [dot] at (\x,\y) {};}}

\draw [densely dotted] (.5,3.5) rectangle (4.5,.5);
\node at (0,0) {};
\end{tikzpicture}
\;
        \Bigg\},
        & &
        \mathcal{T}_{\mathbf{E}_{223}}(\sigma) = \Big\{ \left(\ytableaushort{2}, \ytableaushort{2} \right), \: \left(\ytableaushort{2}, \ytableaushort{3} \right) \Big\},\\[3ex]
%%%%%%%%
\mathcal{F}_{\mathbf{E}_{313}} = 
        \Bigg\{
        \begin{tikzpicture}  [scale=.3,baseline=(current bounding box.center),every node/.style={scale=.7}]
\draw [white,fill=lightgray] (.5,3.5) -- ++ (2,0) -- ++ (0,-1) -- ++ (-1,0) -- ++ (0,-1) -- ++ (-1,0) -- cycle;
\draw[line width=3pt, lightgray] 
(1,1) -- ++ (2,0)
(2,2) -- ++(2,0) -- ++(0,-1)
(3,3) -- ++ (1,0);
\foreach \x in {1,...,4}{\foreach \y in {1,...,3}{\node [dot] at (\x,\y) {};}}

\draw [densely dotted] (.5,3.5) rectangle (4.5,.5);
\node at (0,0) {};
\end{tikzpicture}
\;
        \Bigg\},
        & &
        \mathcal{T}_{\mathbf{E}_{313}}(\sigma) = \Big\{ \left(\ytableaushort{3}, \ytableaushort{1} \right), \: \left(\ytableaushort{3}, \ytableaushort{2} \right), \: \left(\ytableaushort{3}, \ytableaushort{3} \right) \Big\},\\[3ex]
%%%%%%%%
\mathcal{F}_{\mathbf{E}_{412}} = 
        \Bigg\{
        \begin{tikzpicture}  [scale=.3,baseline=(current bounding box.center),every node/.style={scale=.7}]
\draw [white,fill=lightgray] (.5,3.5) -- ++ (2,0) -- ++ (0,-1) -- ++ (-1,0) -- ++ (0,-1) -- ++ (-1,0) -- cycle;
\draw[line width=3pt, lightgray] 
(1,1) -- ++ (3,0)
(2,2) -- ++(2,0)
(3,3) -- ++ (1,0);
\foreach \x in {1,...,4}{\foreach \y in {1,...,3}{\node [dot] at (\x,\y) {};}}

\draw [densely dotted] (.5,3.5) rectangle (4.5,.5);
\node at (0,0) {};
\end{tikzpicture}
\;
        \Bigg\},
        & &
        \mathcal{T}_{\mathbf{E}_{412}}(\sigma) = \Big\{ \left(\ytableaushort{4}, \ytableaushort{1} \right), \: \left(\ytableaushort{4}, \ytableaushort{2} \right), \: \left(\ytableaushort{4}, \ytableaushort{3} \right) \Big\}.
\end{array}
\]

    \noindent Hence the number of jellyfish in $\mathcal{J}(\sigma)$ equals $3(1) + 3(2) + 2(1) + 1(2) + 1(3) + 1(3) = 19$.
    Using Corollary~\ref{cor:HS}, we obtain the Hilbert--Poincar\'e series
\[
    \textstyle \frac{1+2t^2}{(1-t^2)^{10}} + 2 \cdot \frac{1+t^2 + t^4}{(1-t^2)^{11}} + \frac{1+t^2}{(1-t^2)^{11}} + 2 \cdot \frac{1}{(1-t^2)^{12}} + 3 \cdot \frac{1}{(1-t^2)^{12}} + 3 \cdot \frac{1}{(1-t^2)^{12}} = \displaystyle \frac{12- 4 t^4}{(1-t^2)^{12}}.
    \]
    
\end{exam}

\section{Proof of Theorem~\ref{thm:Stanley decomps and HS}}
\label{sec:proofs}

Our proof of Theorem~\ref{thm:Stanley decomps and HS} relies on viewing the module of covariants $M_\sigma$ as the multiplicity space for the $U_\sigma^*$-isotypic component inside $\C[W]$.
These multiplicity spaces, in turn, are irreducible representations of a certain Lie algebra $\g$ corresponding to $H$ via a principle known as \emph{Howe duality}~\cite{Howe89}.
We devote the first subsection below to recording the details of the Howe duality setting.
We largely follow the conventions in~\cite{DES91}*{\S7}; see also the exposition on reductive dual pairs in~\cite{NOT}*{\S3}.
In analytic contexts, Howe duality arises in the theory of \emph{theta liftings}; we refer the reader to~\cites{KV,Gelbart,LeeZhu,NZ}, for example.
The explicit action of $\g$ via differential operators is difficult to find in the literature, and so we include those details in Appendix~\ref{app:Howe}.

\subsection*{Howe duality} 

Let $H = \GL_k$, $\O_k$, or $\Sp_{2k}$ acting on the space $\C[W]$ given in~\eqref{table coordinates}.
Let $\mathcal{D}(W)$ denote the Weyl algebra of polynomial-coefficient differential operators on $\C[W]$, and let $\mathcal{D}(W)^{H}$ denote the subalgebra of $H$-invariant operators. 
There is a complex Lie algebra $\g$ that embeds in $\mathcal{D}(W)^H$ as a Lie subalgebra and is spanned by a generating set of the associative algebra $\mathcal{D}(W)^H$.
Let $\omega: \g \longrightarrow \mathcal{D}(W)^H$ denote the injective homomorphism.
Then $\g$ acts on $\C[W]$ via its image $\omega(\g)$ in $\mathcal{D}(W)^H$.
Similarly, if $\sigma \in \widehat{H}$,
then $\mathcal{D}(W)^{H}$ and hence $\g$
acts on the multiplicity space
\begin{align*}
{\rm Hom}_{H}(U_\sigma^*, \C[W]) &\cong M_\sigma,\\
\psi & \mapsto \Big[w \mapsto \sum_i \psi(u_i^*)(w) u_i \Big],
\end{align*}
where $\{u_i\}$ is a basis for $U_\sigma$ and $\{u_i^*\}$ is the dual basis for $U_\sigma^*$.
The action of an operator $D\in \mathcal{D}(W)^{H}$ on $M_\sigma$ is given by
$D\cdot[w \mapsto \sum_i f_i(w) u_i]=[w \mapsto \sum_{i} (Df_i)(w) u_i]$.
The algebra $\mathcal{D}(W)^H$ acts irreducibly on $M_\sigma$, and hence $\g$ does as well.

One can say more about the Lie algebra $\g$ described above.
In fact, $\g$ is the complexified Lie algebra of a real reductive Lie group $G_\R$, with maximal compact subgroup $K_\R$ such that $G_\R/K_\R$ is a Hermitian symmetric space.
For each classical group $H$, this real group $G_\R$ is given in the table in Theorem~\ref{thm:Howe duality} below.
In that table, we use the following realizations:
\begin{align*}
\U(p,q)& \coloneqq \left\{g\in \GL(p+q,\C) : g \begin{pmatrix}  I_p &0 \\   0 & -I_q\end{pmatrix}g^* =  \begin{pmatrix}  I_p &0 \\   0 & -I_q\end{pmatrix}\right\},\\[5pt]
\Sp(2n,\R)&\coloneqq\left\{g\in \GL(2n,\C) : g \begin{pmatrix}  0 & I_n \\  -I_n& 0\end{pmatrix}g^t =  \begin{pmatrix}  0 & I_n \\  -I_n & 0\end{pmatrix}\right\}\cap\U(n,n),\\[5pt]
\O^*(2n) & \coloneqq \left\{g\in \GL(2n,\C) : g \begin{pmatrix}  0 & I_n \\  I_n & 0\end{pmatrix}g^t =  \begin{pmatrix}  0 & I_n \\   I_n & 0\end{pmatrix}\right\}\cap\U(n,n).\label{eq: Sp(2n,R)}
\end{align*}
Note that we have 
$\U(p,q) \cap \U(p+q) \cong \U(p)\times \U(q)$, and $\Sp(2n,\R)\cap \U(2n) \cong \U(n)$, and $\O^*(2n)\cap \U(2n)\cong \U(n)$,
where in the last two cases $\U(n)$ 
is embedded block-diagonally as follows: 
\[
\left\{\begin{pmatrix}
    a & 0 \\
    0 & (a^{-1})^{t}
\end{pmatrix}
: a\in \U(n)\right\}\cong \U(n).
\]
It follows that
$\Sp(2n,\R)\subseteq \SL(2n,\C)$ and $\O^*(2n) \subseteq \SL(2n,\C)$. 
For this reason, many authors write $\SO^*(2n)$ to denote $\O^*(2n)$.
In the table in Theorem~\ref{thm:Howe duality}, $\operatorname{Mp}(2n, \R)$ denotes the metaplectic double cover of $\Sp(2n, \R)$, while $\widetilde{\GL}_n$ denotes the double cover defined in Appendix~\ref{sub:appendix O}.

\begin{rema} 
Our definition of $\Sp(2n,\R)$ above
differs from the more standard definition 
by a Cayley transform as follows.
Let $\mathbf{c}\in \SL(2n,\C)$ be the matrix given by
\[
\mathbf{c} \coloneqq \frac{1}{\sqrt{2}}
\begin{pmatrix}I_n &\sqrt{-1}I_n\\\sqrt{-1}I_n& I_n
\end{pmatrix}.
\]
Then $\mathbf{c} \Sp(2n,\R)\mathbf{c}^{-1}\subseteq \SL(2n,\R)$, with equality if and only if $n=1$.
\end{rema}

Let $\k$ denote the complexified Lie algebra of $K_\R$, and let $K$ denote the complexification of $K_\R$.
Roughly speaking, a \emph{$(\g,K)$-module} is a 
complex vector space carrying representations of both $\g$ and $K$, such that $K$ acts locally finitely and the actions of $\g$
and $K$ are compatible; see~\cite{SchmidNotes}*{Def.~3.2.3} for details.
In the Howe duality setting, the natural $\k$-action on $M_\sigma$ (namely, the differential of the right matrix multiplication on $\C[W]$) is different from that obtained by restricting the $\g$-action to $\k$.
Indeed, in order for the $\k$-action to extend to a $\g$-action, we must tensor $M_\sigma$ with a certain one-dimensional $\k$-module $F_{-kc\zeta}$, where $c$ is a certain constant intrinsic to $G_\R / K_\R$ (see~\cite{BaiHunziker}*{\S2}), and $\zeta$ is the fundamental weight for $\g$ that is orthogonal to the compact roots (see details in Theorem~\ref{thm:Howe duality}).
In this way, the $\k$-action on
\begin{equation}
    \label{M tilde}
    \widetilde{M}_\sigma \coloneqq M_\sigma \otimes F_{-kc\zeta}
\end{equation}
integrates to a $K$-action in a manner compatible with the $\g$-action, and hence 
$\widetilde{M}_\sigma$ can be viewed as a $(\g, K)$-module.
Further, $\widetilde{M}_\sigma$ is a highest weight $\g$-module, which follows from work of Harish-Chandra~\cites{HC55,HC56} since $G_\R/K_\R$ is Hermitian symmetric.
We denote this highest weight by $\la(\sigma)$.
To write down $\la(\sigma)$ in coordinates, we introduce the following shorthand for a weight of $\g$:
\begin{align*}
    \overrightarrow{\sigma} &\coloneqq \text{vector obtained from $\sigma$ by padding with 0's on the right},\\
    \overleftarrow{\sigma} &\coloneqq \text{vector obtained by reversing the coordinates of $\overrightarrow{\sigma}$},\\
    \mu - k &\coloneqq \text{vector obtained from $\mu$ by subtracting $k$ from every coordinate}.
\end{align*}
In what follows, we write weights for $\g$ in standard epsilon coordinates on the Cartan subalgebra consisting of diagonal matrices.
In particular, a weight for $\g = \gl_{p+q}$ is written in the form $(a_1, \ldots, a_p \mid b_1, \ldots, b_q)$;
a weight for $\g = \sp_{2n}$ or $\so_{2n}$ is written as $(a_1, \ldots, a_n)$.

\begin{thm}[\cite{Howe89}, \cite{KV}]
\label{thm:Howe duality}
Assume one of the three settings in the following table:

\begin{center}
\resizebox{\linewidth}{!}
{ 
\begin{tblr}{colspec={|Q[m,c]|Q[m,c]|Q[m,c]|Q[m,c]|Q[m,c]|Q[m,c]|Q[m,c]|Q[m,c]|},stretch=1.5}

\hline

$H$ & $\C[W]$ & {$\Sigma =$ \\ $\{\sigma \in \widehat{H} : \ldots \}$} & $G_\R$ & $(\g, K)$ & $\lambda(\sigma)$ & $c$ & $\zeta$ \\

\hline[2pt]

$\GL_k$ &
$\C[\M_{p,k} \oplus \M_{k,q}]$ & {$\ell(\sigma^+) \leq q$, \\ $\ell(\sigma^-) \leq p$} & $\operatorname{U}(p,q)$ & $(\gl_{p+q}, \GL_p \times \GL_q)$ & $(-\overleftarrow{\sigma^{-}} - k \mid \overrightarrow{\sigma^+})$ & $1$ & $(1^p \mid 0^q)$ \\

\hline

$\O_k$ & 
$\C[\M_{k,n}]$ & $\ell(\sigma) \leq n$ & ${\rm Mp}(2n,\R)$ & $(\sp_{2n}, \widetilde{\GL}_n)$ & $- \overleftarrow{\sigma} -\frac{k}{2}$ & $1/2$ & $(1^n)$ \\

\hline

$\Sp_{2k}$ &
$\C[\M_{2k,n}]$ & $\ell(\sigma) \leq n$ & $\O^*(2n)$ & $(\so_{2n}, \GL_n)$ & $- \overleftarrow{\sigma} -k$ & $2$ & $\left((\frac{1}{2})^n\right)$ \\

\hline

\end{tblr}
}
\end{center}

\noindent We have the following multiplicity-free decomposition as a $(\g,K) \times H$-module:
\begin{equation}
    \label{Howe decomp}
    \C[W] \cong \bigoplus_{\sigma \in \Sigma} 
    \widetilde{M}_\sigma
   \otimes U_{\sigma}^*,
\end{equation}
where $\Sigma \coloneqq \{ \sigma \in \widehat{H} : M_\sigma \neq 0\}$.
Furthermore, the map $\sigma \mapsto \lambda(\sigma)$ is an injective map from $\Sigma$ into the set of $\k$-dominant integral weights, such that as a $(\g, K)$-module
\begin{equation}
\label{covariants = L lambda}
 \widetilde{M}_\sigma \cong L_{\la(\sigma)}
 \coloneqq \textup{simple $\g$-module with highest weight $\la(\sigma)$},
\end{equation}
and $L_{\la(\sigma)}$ is unitarizable with respect to $\g_\R$.
\end{thm}

See Appendix~\ref{app:Howe} for further details, where we record explicitly the Lie algebra homomorphism $\omega: \g \longrightarrow \mathcal{D}(W)^H$, using a block matrix form that we find especially helpful in understanding the differential operators by which $\g$ acts on $\C[W]$.
Upon comparing~\eqref{table dets GL Sp} with the $H$-spectrum $\Sigma \subseteq \widehat{H}$ given in Theorem~\ref{thm:Howe duality}, we observe that for all $\sigma \in \widehat{H}$,
\begin{equation}
    \label{Sigma T nonempty}
    \sigma \in \Sigma \text{ if and only if } \mathcal{T}(\sigma) \neq \varnothing.
\end{equation}

\subsection*{General linear and symplectic groups}

We now proceed toward the proof of Theorem~\ref{thm:Stanley decomps and HS}.
In the following discussion, let $H = \GL_k$ or $\Sp_{2k}$.
Let $N$ be the maximal unipotent subgroup of $H$ defined in the discussion preceding~\eqref{wt H GL} and~\eqref{wt H Sp}.
In \cite{Jackson}, the ring $\C[W]^N$ is called the \emph{ring of covariants}.
This ring effectively gathers all the modules of covariants into a single object, as follows.
Taking the $N$-invariants in the Howe duality decomposition~\eqref{Howe decomp}, we have
\begin{align*}
\C[W]^N &\cong \bigoplus_{\sigma \in \Sigma} \widetilde{M}_\sigma \otimes (U_\sigma^*)^N,\\
& \cong \bigoplus_{\sigma \in \Sigma} \widetilde{M}_\sigma \otimes \C u^*_{_{\rm LW}},
\end{align*}
where the functional $u^*_{_{\rm LW}}$ is dual to a fixed lowest weight vector $u_{_{\rm LW}} \in U_\sigma$, as in~\eqref{weight of u circ phi}.
In particular, recall that $u^*_{_{\rm LW}}$ is a highest weight vector in $U_\sigma^*$.
Hence as a $\g$-module, we have
\[
\C[W]^N = \bigoplus_{\sigma \in \Sigma} \C[W]^N_{\sigma^*} \cong \bigoplus_{\sigma \in \Sigma} \widetilde{M}_\sigma, 
\]
where 
\[
    \C[W]^N_{\sigma^*} \coloneqq \Big\{ f \in \C[W]^N : {\rm wt}_H(f) = \sigma^* \Big\},
\]
with a canonical isomorphism of $\g$-modules given by
    \begin{align}
    \label{Psi}
    \begin{split}
        \Psi: \widetilde{M}_\sigma &\longrightarrow \C[W]^N_{\sigma^*},\\
        \phi & \longmapsto u^*_{_{\rm LW}} \circ \phi.
        \end{split}
    \end{align}
Comparing~\eqref{Psi} with~\eqref{phi_T GL Sp}, it follows that $\Psi(\phi_T) = \det_T$, for all $T \in \mathcal{T}(\sigma)$.
Thus, if $\mathbf{f} \in \C[W]^H$ denotes a monomial in the $f_{ij}$'s, then we have
\begin{equation}
    \label{Psi on f det}
    \Psi(\mathbf{f} \cdot \phi_T) = \mathbf{f} \cdot {\textstyle \det_T}.
\end{equation}

Define the set of indeterminates
\[
Z \coloneqq 
\begin{cases}
    \{ z_{ij} : (i,j) \in \mathbf{P} \} \cup \left\{ w_J : J \in \bigcup_{t=1}^k \binom{[n]}{t} \right\}, & H = \Sp_{2k},\\[1ex]
    \{ z_{ij} : (i,j) \in \mathbf{P} \} \cup  \left\{w_J : J \in \bigcup_{t=1}^k \binom{[q]}{t} \right\} \cup \left\{ w^*_I : I \in \bigcup_{t=1}^k \binom{[p]}{t} \right\}, & H = \GL_k.
\end{cases}
\]
The \emph{tableau order} is the partial order among nonempty sets of positive integers, whereby $J \leq J'$ if and only if (when viewed as top-justified columns with increasing entries) $J$ can appear to the left of $J'$ in a semistandard tableau.

The following definition distinguishes certain monomials in $\C[Z]$; following Jackson~\cite{Jackson}, we associate these monomials with combinatorial objects called \emph{splits}.

\begin{defi}[Monomial of a split for $H = \Sp_{2k}$ or $\GL_k$; see Definition 3.6.15--16 in~\cite{Jackson}]\
\label{def:Sp split}

    An \emph{$\Sp_{2k}$-split} consists of four (possibly empty) subsets of $[n]$, namely $A = \{a_i\}_{i=1}^v$, $B = \{b_i\}_{i=1}^{r}$, $C = \{c_i\}_{i=1}^r$, $D = \{d_i\}_{i=1}^t$, such that $t+v \leq k$ and $r + t = k+1$, where
    \[
    d_t < \cdots < d_1 < c_r < \cdots < c_1 < b_1 < \cdots < b_r \text{ and } d_1 < a_v < \cdots < a_1.
    \]
The \emph{monomial} of an $\Sp_{2k}$-split is the following element of $\C[Z]$:
 \begin{equation}
    \label{monomial of Sp split}
 \left(\prod_{i=1}^r z_{c_i, b_i} \right) \textstyle w_{D \sqcup A}.
 \end{equation}

 A \emph{$\GL_k$-split} consists of four (possibly empty) subsets of $[q]$, namely $A = \{a_i\}_{i=1}^w$, $B = \{b_i\}_{i=1}^s$, $C = \{c_i\}_{i=1}^{s}$, $D = \{ d_i\}_{i=1}^{u}$, and four (possibly empty) subsets of $[p]$, namely $A^* = \{a^*_i\}_{i=1}^v$, $B^* = \{b^*_i\}_{i=1}^r$, $C^* = \{c^*_i\}_{i=1}^{r}$, $D^* = \{ d^*_i\}_{i=1}^{t}$, such that $t+v \leq k$, $u+w \leq k$, $r+s+t+u = k+1$, and $r+s > 0$, where
    \[d_u < \cdots < d_1 < c_s < \cdots < c_1 < b^*_1 < \cdots < b^*_s \text{ and } d_1 < a_w \text{ and } c_1 < b_1,
    \]
    \[
    d^*_t < \cdots < d^*_1 < c^*_r < \cdots < c^*_1 \leq b_1 < \cdots < b_r \text{ and } d^*_1 < a^*_v \text{ and } c^*_1 < b^*_1.
    \]
The \emph{monomial} of a $\GL_k$-split is the following element of $\C[Z]$:
 \begin{equation}
    \label{monomial of split GL}
 \left(\prod_{i=1}^r z_{c^*_i, b_i} \prod_{i=1}^s z_{b^*_i, c_i} \right) \textstyle w_{D \sqcup A} \: w^*_{D^* \sqcup A^*}.
 \end{equation} 
\end{defi}

\begin{lemma}
    \label{lemma:splits J J'}
    Let $\mathbf{z} \in \C[Z]$ be a monomial in the $z_{ij}$'s.
    If $\mathbf{z} w_{J}$ is divisible by the monomial of an $\Sp_{2k}$-split, then so is $z_{{\rm supp}(\mathbf{z})} w_{J'}$, for all $J' \leq J$ in the tableau order.
    Likewise, if $\mathbf{z} w_J w^*_I$ is divisible by the monomial of a $\GL_k$-split, then so is $z_{{\rm supp}(\mathbf{z})} w_{J'} w^*_{I'}$, for all $J' \leq J$ and $I' \leq I$ in the tableau order.
\end{lemma}

\begin{proof}
    Let $H = \Sp_{2k}$, and suppose that $\mathbf{z} w_J$ is divisible by the monomial of a split $(A, B, C, D)$.
    Since the monomial of a split is squarefree in the $z_{ij}$'s, we can pass from $\mathbf{z}$ to $z_{{\rm supp}(\mathbf{z})}$.
    By~\eqref{monomial of Sp split}, we have $J = D \sqcup A$.
    If $J' \leq J$, then $J' = D' \sqcup A'$, where $\#D' = \#D = t$ and $\#A' \geq \#A$.
    In particular, $D'$ consists of elements $d'_t < \cdots < d'_1$ such that each $d'_i \leq d_i$.
    Thus by Definition~\ref{def:Sp split}, $(A', B, C, D')$ is also an $\Sp_{2k}$-split, which completes the proof for $\Sp_{2k}$.
    The proof for $\GL_k$ is identical.
\end{proof}

In the following lemma, we introduce a map $\widetilde{\pi}^* : \C[Z] \longrightarrow \C[W]^N$.
The notation is due to the fact that upon restricting the domain of $\widetilde{\pi}^*$ to $S = \C[\mathbf{P}]$, one recovers the map $\pi^* : S \longrightarrow \C[W]^H$ from~\eqref{pi star map}.

\begin{lemma}[see \cite{Jackson}*{Thms.~3.6.17, 3.8.12, and 3.8.27}] 
    \label{lemma:Jackson Grobner}
Let $H = \Sp_{2k}$ or $\GL_k$.
Define the algebra homomorphism
\begin{align}
    \label{Omega}
    \begin{split}
    \widetilde{\pi}^* : \C[Z] & \longrightarrow \C[W]^N,\\
    z_{ij} & \longmapsto f_{ij},\\
    w_J & \longmapsto {\textstyle \det_J},\\
   \textup{(if $H = \GL_k$)} \;\; w^*_I & \longmapsto {\textstyle \det^*_I},\\
    \end{split}
\end{align}
with $f_{ij}$ and the \textup{det} functions as defined above in~\eqref{table P fij} and~\eqref{table dets GL Sp}, respectively.
Then $\ker \widetilde{\pi}^*$ is the ideal generated by the following elements:
\begin{enumerate}
    \item all monomials of $H$-splits;
    \item all products $w_J w_{J'}$ such that $J$ and $J'$ are incomparable in the tableau order;
    \item \textup{(}if $H = \GL_k$\textup{)} all products $w^*_I w^*_{I'}$ such that $I$ and $I'$ are incomparable in the tableau order;
    \item \textup{(}if $H = \GL_k$\textup{)} all products $w^*_I w_J$ such that $\#I + \#J > k$.
\end{enumerate}
Moreover, the images of the monomials lying outside $\ker \widetilde{\pi}^*$ furnish a linear basis for $\C[W]^N$, and are said to be \emph{standard monomials} in $\C[W]^N$.
\end{lemma}

We equip $\mathbf{P}$ with the \emph{product order}, that is, the partial order $\leq$ whereby
\begin{equation}
    \label{partial order on P}
    (i,j) \leq (i', j') \text{ if and only if } i \leq i' \text{ and } j \leq j'.
\end{equation}
Note that $\mathbf{P}$ is a bounded poset:
for $\Sp_{2k}$, the minimal and maximal elements are $(1,2)$ and $(n-1,n)$, respectively, and for $\GL_k$, they are $(1,1)$ and $(p,q)$.
Observe that an antichain in $\mathbf{P}$ is a subset of the form $\{(i_1, j_1), \ldots, (i_\ell, j_\ell)\}$ where $i_1 < \cdots < i_\ell$ and $j_1 > \cdots > j_\ell$.
Given $\mathbf{S} \subseteq \mathbf{P}$, we will use the shorthand
\[
\mathbf{S}|_{i > a, \: j>b } \coloneqq \{ (i,j) \in \mathbf{S} : i>a \text{ and } j>b \}.
\]
For each $T \in \mathcal{T}(\sigma)$, we define the following abstract simplicial complex $\Delta_T$:
\begin{align}
\label{Delta T}
\begin{split}
    (H = \Sp_{2k}) \quad & \Delta_T \coloneqq
        \left\{ \mathbf{S} \subseteq \mathbf{P} : \begin{array}{l}\operatorname{width} \left( \mathbf{S}|_{i > T_{t,1}} \right) \leq k-t \\ \textup{for all } 0 \leq t \leq \ell(\sigma) \end{array} \right\},\\[1ex]
        (H = \GL_{k}) \quad & \Delta_T \coloneqq
        \left\{ \mathbf{S} \subseteq \mathbf{P} : \begin{array}{l} \operatorname{width} \big( \mathbf{S}|_{i > T^-_{t,1}, \: j > T^+_{u,1}} \big) \leq k-t-u \\ \textup{for all } 0 \leq t \leq \ell(\sigma^-) \textup{ and } 0 \leq u \leq \ell(\sigma^+) \end{array} \right\},
        \end{split}
    \end{align}
    where we put $T_{0,1} \coloneqq 0$.
    Analogously to the function $\det_T$ in~\eqref{table dets GL Sp}, define
    \[
    w_T \coloneqq \begin{cases}
        \displaystyle \prod_{\substack{\text{columns} \\ \text{$J$ of $T$}}}{\textstyle w_J} , & H = \Sp_{2k}, \\[5ex]
        \displaystyle \prod_{\substack{\text{columns} \\ \text{$J$ of $T^+$}}}{\textstyle w_J} \prod_{\substack{\text{columns} \\ \text{$I$ of $T^-$}}} {\textstyle w^*_I} , & H = \GL_k.
    \end{cases}
    \]

\begin{lemma}
    \label{lemma: standard monomials}
    Let $H = \Sp_{2k}$ or $\GL_k$.
    We have
    \[
    \C[Z] \backslash \ker \widetilde{\pi}^* = \bigoplus_{\sigma \in \Sigma} \left(\bigoplus_{T \in \mathcal{T}(\sigma)} \C[\Delta_T] \: w_T \right).
    \]
\end{lemma}

\begin{proof}
    Every monomial in $\C[Z]$ is of the form $\mathbf{z} \mathbf{w}$, where $\mathbf{z}$ is a monomial in the $z_{ij}$'s, and where $\mathbf{w}$ is a monomial in the $w_J$'s (and $w^*_I$'s, if $H = \GL_k$).
    We have $\mathbf{zw} \notin \ker \widetilde{\pi}^*$ if and only if \begin{enumerate}[label=(\alph*)]
        \item $\mathbf{zw}$ is not divisible by the monomial of an $H$-split, and
        \item $\mathbf{w} = w_T$ for some $T \in \coprod_{\sigma \in \Sigma} \mathcal{T}(\sigma)$;       
    \end{enumerate}
    condition (a) is due to the generators~(1) of $\ker \widetilde{\pi}^*$ in Lemma~\ref{lemma:Jackson Grobner}, and condition (b) is due to the generators~(2)--(4).
    Therefore we have
    \[
    \C[Z] \backslash \ker \widetilde{\pi}^* = \bigoplus_{\sigma \in \Sigma} \Bigg( \bigoplus_{T \in \mathcal{T}(\sigma)} \underbrace{\Bigg(\bigoplus_{\substack{\mathbf z: \\ \mathbf{z}w_T \text{ not div.} \\ \text{by monom.} \\ \text{of $H$-split}}} \C \, \mathbf{z} \Bigg)}_{(*)} w_T \Bigg),
    \]
    and it suffices to prove that $(*) = \C[\Delta_T]$, as defined in~\eqref{Delta T}.
    We prove this below separately for the $\Sp_{2k}$ and $\GL_k$ cases.

    Let $H = \Sp_{2k}$, let $\sigma \in \Sigma$, and fix $T \in \mathcal{T}(\sigma)$, with the initial column denoted by $T_{\bullet, 1}$.
    By Lemma~\ref{lemma:splits J J'}, $\mathbf{z} w_T$ is divisible by the monomial of a split if and only if $z_{\operatorname{supp}(\mathbf{z})} w_{T_{\bullet,1}}$ is divisible by the monomial of a split.
    In this case, we have $T_{\bullet, 1} = D \sqcup A$ in~\eqref{monomial of Sp split}, and so there exists some $t$ such that $0 \leq t \leq \ell(\sigma)$ and
    \[
    T_{t,1} = d_1 < c_r < \cdots < c_1 < b_1 < \cdots < b_r,
    \]
    where $r+t = k+1$ (and if $t=0$, then omit the $d_1$), and $\{(c_i, b_i)\}_{i=1}^r$ is an antichain contained in $\operatorname{supp}(\mathbf{z})$.
    Equivalently, $\operatorname{width}(\operatorname{supp}(\mathbf{z})|_{i > T_{t,1}}) \geq k-t+1$.
    It follows that $\mathbf{z}$ occurs in $(*)$ if and only if
    \[
    \operatorname{width}(\operatorname{supp}(\mathbf{z})|_{i > T_{t,1}}) \leq k-t, \qquad \text{for all } 0 \leq t \leq \ell(\sigma).
    \]
    But by~\eqref{Delta T}, this is equivalent to the condition $\operatorname{supp}(\mathbf{z}) \in \Delta_T$.
    Thus, since $(*)$ has a basis consisting of those monomials $\mathbf{z}$ such that $\operatorname{supp}(\mathbf{z}) \in \Delta_T$, it follows from the definition of a Stanley--Reisner ring that $(*) = \C[\Delta_T]$ as claimed.
    This completes the proof for $H = \Sp_{2k}$.

    The proof for $H = \GL_k$ is essentially identical, with the following modifications.
    By Lemma~\ref{lemma:splits J J'}, $\mathbf{z} w_T$ is divisible by the monomial of a split if and only if $z_{\operatorname{supp}(\mathbf{z})} w_{T^+_{\bullet,1}} w^*_{T^-_{\bullet,1}} $ is divisible by the monomial of a split.
    In this case, we have $T^+_{\bullet,1} = D \sqcup A$ and $T^-_{\bullet,1} = D^* \sqcup A^*$ in~\eqref{monomial of split GL}, and so there exist some $u$ and $t$ such that
    \begin{align*}
        T^+_{u,1} &  = d_1 < c_s < \cdots < c_1 < b_1 < \cdots < b_r,\\
        T^-_{t,1} &= d^*_1 < c^*_r < \cdots < c^*_1 < b^*_1 < \cdots < b^*_s,
    \end{align*}
    where $r+s+t+u = k+1$, and $\{(c^*_i, b_i)\}_{i=1}^{r} \cup \{(b^*_i, c_i)\}_{i=1}^{s}$ is an antichain contained in $\operatorname{supp}(\mathbf{z})$.
    Equivalently, $\operatorname{width}(\operatorname{supp}(\mathbf{z})|_{i > T^-_{t,1}, \: j> T^+_{u,1}}) \geq k-t-u+1$.
    The rest of the proof is the same as in the $\Sp_{2k}$ case above.
\end{proof}

\begin{proof}[Proof of Theorem~\ref{thm:Stanley decomps and HS} \textup{($H = \GL_k$ or $\Sp_{2k}$)}]

For each group, the proof will proceed as follows.
Suppose that $k< \#\delta(\mathbf{P})$.
For each $\mathbf{E} \in \E$ we will define a certain abstract simplicial complex $\Delta_{\mathbf{E}}$ on a subset of $\mathbf{P}$, in such a way that 
\begin{align}
    \Delta_T &= \Delta_{{\rm end}(T)} \text{ for all }T \in \mathcal{T}(\sigma); \label{Delta T = Delta end T} \\
    \mathcal{F}(\Delta_{\mathbf{E}}) & = \mathcal{F}_{\mathbf{E}} \text{ as defined in~\eqref{F_E}, for all $\mathbf{E} \in \E$}. \label{F Delta E = F E}
\end{align}
Then we will exhibit a certain shelling of each $\Delta_{\mathbf{E}}$, in such a way that
\begin{equation}
    \label{res = cor}
    \operatorname{res}(\F) = \cor(\F) \text{ as defined in~\eqref{corner Sp}, for all $\F \in \mathcal{F}(\Delta_{\mathbf{E}})$}.
\end{equation}
From this it will follow that
\begin{align*}
    \C[Z] \backslash \ker \widetilde{\pi}^* &= \bigoplus_{\sigma \in \Sigma} \left(\bigoplus_{T \in \mathcal{T}(\sigma)} \C[\Delta_T] \: w_T \right) && \text{by Lemma~\ref{lemma: standard monomials}} \\
    &= \bigoplus_{\sigma \in \Sigma} \left(\bigoplus_{T \in \mathcal{T}(\sigma)} \C[\Delta_{{\rm end}(T)}] \: w_T \right) && \text{by~\eqref{Delta T = Delta end T}} \\
    &= \bigoplus_{\sigma \in \Sigma} 
    \left( \bigoplus_{\mathbf{E} \in \E} 
    \left(
    \bigoplus_{T \in \mathcal{T}_{\mathbf{E}}(\sigma)} \C[\Delta_\mathbf{E}] \: w_T
    \right)
    \right) && \text{by~\eqref{T_E}} \\
    &= \bigoplus_{\sigma \in \Sigma} 
    \left( \bigoplus_{\mathbf{E} \in \E} 
    \left(
    \bigoplus_{T \in \mathcal{T}_{\mathbf{E}}(\sigma)} \left(
    \bigoplus_{\F \in \mathcal{F}(\Delta_{\mathbf{E}})}
    \C[\F] \: z_{{\rm res}(\F)}
    \right) w_T
    \right)
    \right) && \text{by~\eqref{Stanley decomp SR ring}} \\
    &= \bigoplus_{\sigma \in \Sigma} 
    \left( \bigoplus_{\mathbf{E} \in \E} 
    \left(
    \bigoplus_{T \in \mathcal{T}_{\mathbf{E}}(\sigma)} \left(
    \bigoplus_{\F \in \mathcal{F}_{\mathbf{E}}}
    \C[\F] \: z_{\cor(\F)}
    \right) w_T
    \right)
    \right) && \text{by~\eqref{F Delta E = F E} and~\eqref{res = cor}} \\
    &= \bigoplus_{\sigma \in \Sigma} \left(
    \bigoplus_{(\F, T) \in \mathcal{J}(\sigma)}
    \C[\F] \: z_{\cor(\F)} \: w_T
    \right) && \text{by Definition~\ref{def:jellyfish}}.
\end{align*}
Recalling from Lemma~\ref{lemma:Jackson Grobner} that the images (under $\widetilde{\pi}^*$) of the monomials in $\C[Z] \backslash \ker \widetilde{\pi}^*$ furnish a linear basis for $\C[W]^N$, we conclude that the direct sums above are preserved by $\widetilde{\pi}^*$.
Therefore, applying $\widetilde{\pi}^*$ to the last displayed equation above, we have
\[
\C[W]^N = \bigoplus_{\sigma \in \Sigma} \left( \bigoplus_{(\F, T) \in \mathcal{J}(\sigma)} \C[f_{ij} : (i,j) \in \F] \: f_{\cor(\F)} \: {\textstyle \det_T} \right).
\]
By~\eqref{wt det GL} and~\eqref{wt det Sp}, the $\sigma$-component above is precisely $\C[W]^N_{\sigma^*}$, so that
\[
\C[W]^N_{\sigma^*} =  \bigoplus_{(\F, T) \in \mathcal{J}(\sigma)} \C[f_{ij} : (i,j) \in \F] \: f_{\cor(\F)} \: {\textstyle \det_T}.
\]
Applying the inverse of the isomorphism $\Psi$ in~\eqref{Psi on f det}, we obtain the following decomposition of $\widetilde{M}_\sigma$, which (as a graded vector space) is isomorphic to $M_\sigma$:
\[
    M_\sigma = \bigoplus_{(\F, T) \in \mathcal{J}(\sigma)} \C[f_{ij} : (i,j) \in \F] \: f_{\cor(\F)} \: {\textstyle \phi_T},
\]
which establishes the $k < \#\delta(\mathbf{P})$ case in Theorem~\ref{thm:Stanley decomps and HS}.
Note that this result is valid for all $\sigma \in \widehat{H}$ (not just $\Sigma$):
on one hand, by definition (see Theorem~\ref{thm:Howe duality}), we have $\sigma \in \Sigma$ if and only if $M_\sigma \neq 0$, and on the other hand, by observation~\eqref{Sigma T nonempty}, we have $\sigma \in \Sigma$ if and only if $\mathcal{J}(\sigma) \neq \varnothing$.

If $k \geq \#\delta(\mathbf{P})$, then observe that there are no $H$-splits (see Definition~\ref{def:Sp split}).
Consequently, in Lemma~\ref{lemma: standard monomials}, each $\Delta_T$ is just the power set of $\mathbf{P}$, and has a unique facet $\mathbf{P}$.
Thus in the case $k \geq \#\delta(\mathbf{P})$, we have
\[
\C[Z] \backslash \ker \widetilde{\pi}^* = \bigoplus_{\sigma \in \Sigma} 
\left(
\bigoplus_{T \in \mathcal{T}(\sigma)}
\C[\mathbf{P}] \: w_T
\right),
\]
and now the argument from above yields the $k \geq \#\delta(\mathbf{P})$ result in Theorem~\ref{thm:Stanley decomps and HS}.

It remains to fill in the initial details above, namely the definition of $\Delta_{\mathbf{E}}$ and the verification of its three properties~\eqref{Delta T = Delta end T}--\eqref{res = cor}.
We do this below separately for the $\Sp_{2k}$ and $\GL_k$ cases.

Let $H = \Sp_{2k}$.
From now on, let $\mathbf{E} = \{(i_t, n) : 1 \leq t \leq k\} \in \E$, where $0 \eqcolon i_0 < i_1 < \cdots < i_k \leq n-1$.
Define the following abstract simplicial complex:
\begin{equation}
    \label{Delta_E Sp}
    \Delta_{\mathbf{E}} \coloneqq \left\{ \mathbf{S} \subseteq \mathbf{P} : \begin{array}{l}\operatorname{width} \left( \mathbf{S}|_{i > i_t} \right) \leq k-t \\ \textup{for all } 0 \leq t \leq k \end{array} \right\}.
\end{equation}

To prove~\eqref{Delta T = Delta end T}, namely that $\Delta_T = \Delta_{{\rm end}(T)}$, let $T \in \mathcal{T}(\sigma)$ with ${\rm end}(T) = \mathbf{E}$;
we must show that for all $\mathbf{S} \subseteq \mathbf{P}$,
\begin{equation}
    \label{Delta T = Delta end T concrete Sp}
    \begin{array}{c}
    {\rm width}(\mathbf{S}|_{i> T_{t,1}}) \leq k - t \text{ for all $0 \leq t \leq \ell(\sigma)$} \\
    \Updownarrow \\
    {\rm width}(\mathbf{S}|_{i> i_t}) \leq k - t \text{ for all $0 \leq t \leq k$}.
    \end{array}
\end{equation}
By definition, we have both $T_{0,1} = 0$ and $i_0 = 0$, so the $t=0$ case in~\eqref{Delta T = Delta end T concrete Sp} is done.
Next, suppose $1 \leq t \leq \ell(\sigma)$.
By~\eqref{end T details Sp}, if $T_{t,1} \leq \hat{\imath}_t$, then $i_t = T_{t,1}$ and so we are done.
Thus suppose that $T_{t,1} > \hat{\imath}_t$, so that $i_t = \hat{\imath}_t$; in this case, we claim that both propositions in the biconditional~\eqref{Delta T = Delta end T concrete Sp} are automatically true.
On one hand, if $t \geq 2(k+1) - n$, then by~\eqref{end T details Sp} we have $i_t = \hat{\imath}_t = n - 2(k-t) - 1$.
Observing that for all $0 \leq a \leq n$ we have
\begin{equation}
    \label{width general Sp}
    {\rm width}(\mathbf{P}|_{i > a}) = \lfloor (n-a) / 2 \rfloor,
\end{equation}
and setting $a = i_t = n - 2(k-t) - 1$, we obtain
\[
    {\rm width}(\mathbf{P}|_{i > i_t}) = k - t.
\]
Thus the top proposition in~\eqref{Delta T = Delta end T concrete Sp} is true for all $\mathbf{S} \subseteq \mathbf{P}$;
likewise, since the width~\eqref{width general Sp} weakly decreases as $a$ increases, and since $T_{t,1} > i_t$, the bottom proposition in~\eqref{Delta T = Delta end T concrete Sp} is also true for all $\mathbf{S} \subseteq \mathbf{P}$, so we are done.
On the other hand, if $t < 2(k+1) - n$, then $i_t = \hat{\imath}_t = t \geq n-2(k-t) - 1$;
thus again, both propositions in~\eqref{Delta T = Delta end T concrete Sp} are true for all $\mathbf{S} \subseteq \mathbf{P}$, so we are done.
The last remaining case is $t > \ell(\sigma)$;
in this range, by~\eqref{end T details Sp} we have $i_t = \hat{\imath}_t$, and so again the bottom proposition in~\eqref{Delta T = Delta end T concrete Sp} is true for all $\mathbf{S} \subseteq \mathbf{P}$.
This completes the proof of~\eqref{Delta T = Delta end T}.

Next we will prove~\eqref{F Delta E = F E}, namely that the facet set $\mathcal{F}(\Delta_{\mathbf{E}})$ coincides with the set $\mathcal{F}_{\mathbf{E}}$ in~\eqref{F_E}.
Let $\F \in \mathcal{F}_{\mathbf{E}}$.
Each subset $\F|_{i > i_t}$ can be partitioned into $k-t$ many chains: namely, take its intersections with $\mathbf{L}_{t+1}, \mathbf{L}_{t+2}, \ldots, \mathbf{L}_{k}$, where for each $\ell > i_t$ (in order to cover $\mathbf{A}$) the starting point of $\mathbf{L}_\ell$ is extended from $(t, b_t)$ to $(t, t+1)$.
It follows that each ${\rm width} (\F|_{i > i_t}) \leq k-t$, and so $\F \in \Delta_{\mathbf{E}}$.
To show maximality, let $\F' = \F \sqcup \{(i',j')\}$ for some $(i',j') \in \mathbf{P} \backslash \F$.
Consider the antidiagonal $\mathbf{D} = \{ (i,j) \in \mathbf{P} : i+j = i'+j' \}$ containing $(i',j')$.
There is some $t$, where $0 \leq t \leq k$, such that $\mathbf{D}$ intersects precisely the lattice paths $\mathbf{L}_{t+1}, \mathbf{L}_{t+2}, \ldots, \mathbf{L}_{k}$, and thus we have $\mathbf{D} \subseteq \mathbf{P}|_{i > i_t}$.
But $\mathbf{D} \cap \F$ is an antichain of size $k-t$, and so $\mathbf{D} \cap \F'$ is an antichain of size $k-t+1$ contained in $\mathbf{P}|_{i > i_t}$.
Thus we have ${\rm width}(\F'|_{i > i_t}) > k-t$, and so $\F' \notin \Delta_{\mathbf{E}}$.
Therefore, $\F$ is a maximal element (\ie a facet) of $\Delta_{\mathbf{E}}$, and hence $\mathcal{F}_{\mathbf{E}} \subseteq \mathcal{F}(\Delta_{\mathbf{E}})$.
Conversely, any facet $\F \in \mathcal{F}(\Delta_{\mathbf{E}})$ can be decomposed into the disjoint union of saturated chains $\mathbf{C}_1, \ldots, \mathbf{C}_k$, where the minimal element of $\mathbf{C}_t$ is the minimal element $(t, t+1)$ of $\mathbf{P} \backslash (\mathbf{C}_1 \sqcup \cdots \sqcup \mathbf{C}_{t-1})$, and the maximal element of $\mathbf{C}_t$ is the maximal element $(i_t, n)$ of $\mathbf{P}|_{i \leq i_t}$.
But then $\F = \coprod_{t=1}^k \mathbf{C}_{t} = (\mathbf{A} \backslash \delta(\mathbf{A})) \sqcup \coprod_{t=1}^k \mathbf{L}_t$, where $\mathbf{L}_t = \mathbf{C}_t \backslash (\mathbf{A} \backslash \delta(\mathbf{A}))$ is a southeast lattice path from $(t, b_t)$ to $(i_t, n)$, and hence by~\eqref{F} we have $\F \in \mathcal{F}_{\mathbf{E}}$.
Thus we have $\mathcal{F}(\Delta_{\mathbf{E}}) \subseteq \mathcal{F}_{\mathbf{E}}$,  and the claim~\eqref{F Delta E = F E} follows.

Finally, to show~\eqref{res = cor}, namely that ${\rm res}(\F) = \cor(\F)$ for all $\F \in \mathcal{F}(\Delta_{\mathbf{E}})$, we need to exhibit a shelling of $\Delta_{\mathbf{E}}$ such that the restrictions coincide with the corners defined in~\eqref{corner Sp}.
(Note that $\Delta_{\mathbf{E}}$ is pure, since its facets all have the same cardinality; see the proof of Corollary~\ref{cor:HS}.)
Let $\F = (\mathbf{A} \backslash \delta(\mathbf{A})) \sqcup (\mathbf{L}_1 \sqcup \cdots \sqcup \mathbf{L}_k) \in \mathcal{F}(\Delta_{\mathbf{E}})$ as in~\eqref{F}.
Denote the southwest boundary of a subset $\mathbf{S} \subseteq \mathbf{P}$ by
\begin{equation}
    \label{omega}
    \gamma(\mathbf{S}) \coloneqq \{ (i,j) \in \mathbf{S} : (i+1, j-1) \notin \mathbf{S} \}.
\end{equation}
Note that each $\mathbf{L}_t$ is contained in the poset interval between $(t, b_t)$ and $(i_t, n)$.
In fact, since $\mathbf{L}_t$ cannot intersect $\mathbf{L}_{t+1}$, each path $\mathbf{L}_t$ is contained in the subset $\mathbf{Q}_t$ defined recursively for $t = k, k-1, \ldots, 2, 1$ as follows:
\[
    \mathbf{Q}_t \coloneqq \Big\{ (i,j) \in \mathbf{P} : (t, b_t) \leq (i,j) \leq (i_t, n) \Big\} \Big\backslash \bigcup_{\ell=t+1}^k \! \gamma(\mathbf{Q}_{\ell}).
\]
Now equip each subposet $\mathbf{Q}_t$ with the labeling $\alpha_t$ defined as follows:
\begin{align}
\label{alpha labeling}
\begin{split}
    \alpha_t\Big((i-1,j), \: (i, j) \Big) &= 0,\\[1ex]
    \alpha_t\Big((i,j-1), \: (i, j) \Big) &= \begin{cases}
        0, & (i, j) \in \gamma(\mathbf{Q}_t),\\
        1 & \text{otherwise}.
    \end{cases}
\end{split}
\end{align}
(Note that if $t < 2(k+1) - n$, then $\mathbf{Q}_t$ consists of the single point $(t,b_t) = (i_t, n)$, so there are no edges to label.)
If $(i,j) < (i',j')$, then there is a unique saturated chain $(i,j) \lessdot \cdots \lessdot (i', j')$ with label sequence $(0,\ldots,0)$, obtained by stepping south whenever possible, and east otherwise.
Therefore $\alpha_t$ is an EL-labeling of $\mathbf{Q}_t$, by Definition~\ref{def:EL-labeling}.
Thus by Lemma~\ref{lemma:EL labeling induces shelling}, $\alpha_t$ induces a shelling order on the facets of $\Delta(\mathbf{Q}_t)$, or equivalently, on the lattice paths $\mathbf{L}_t$ from $(t,b_t)$ to $(i_t, n)$;
moreover, by~\eqref{R equals descents}, with respect to this shelling order, ${\rm res}(\mathbf{L}_t)$ is the set of descents of $\mathbf{L}_t$.
With respect to $\alpha$, a descent of $\mathbf{L}_t$ corresponds to a $(1,0)$ in its label sequence.
In turn, a $(1,0)$ occurs in the label sequence whenever $\mathbf{L}_t$ contains three points of the form $(i,j-1)$, $(i, j)$, and $(i+1, j)$ such that $(i, j) \notin \gamma(\mathbf{Q}_t)$.
Note that $(i, j) \in \gamma(\mathbf{Q}_t)$ if and only if $(i + \ell, j-\ell) \in \mathbf{L}_{t+\ell}$ for all $1 \leq \ell \leq k-t$.
Comparing this with the definition~\eqref{corner Sp} of a corner, we have that ${\rm res}(\mathbf{L}_t) = \cor(\F) \cap \mathbf{L}_t$.
Note that $\gamma(\mathbf{Q}_t)$ is the facet of $\Delta(\mathbf{Q}_t)$ with label sequence $(0, \ldots, 0)$;
thus, since $\gamma(\mathbf{Q}_t) \cap \mathbf{Q}_{t-1} =\varnothing$, the shellings induced by the $\alpha_t$'s extend to a shelling of $\Delta_{\mathbf{E}}$, where we order the facets $\F \in \mathcal{F}(\Delta_{\mathbf{E}})$ lexicographically by concatenating the label sequences of $\mathbf{L}_1, \ldots, \mathbf{L}_k$ in that order.
With respect to this shelling order, we have ${\rm res}(\F) = \cor(\F)$.
This establishes the claim~\eqref{res = cor}, and completes the proof of Theorem~\ref{thm:Stanley decomps and HS} for $H = \Sp_{2k}$.

Now let $H = \GL_k$; again, we must prove~\eqref{Delta T = Delta end T}--\eqref{res = cor}.
The arguments are the same as in the $\Sp_{2k}$ case above, with the following modifications.

Recall the parameters $k^+$ and $k^-$ from~\eqref{k+ k-}.
From now on, let $\mathbf{E} = \{(p, j_u) : 1 \leq u \leq k^+\} \sqcup \{(i_t, q) : 1 \leq t \leq k^-\} \in \E$, where $0 \coloneqq j_0 < j_1 < \cdots < j_{k^+} \leq q$ and $0 \coloneqq i_0 < i_1 < \cdots < i_{k^-} \leq p$.
Define
\[
\Delta_{\mathbf{E}} \coloneqq \left\{ \mathbf{S} \subseteq \mathbf{P} : \begin{array}{l}\operatorname{width} \left( \mathbf{S}|_{i > i_t, \: j > j_u} \right) \leq k-t-u \\ \text{for all } 0 \leq t \leq k^- \text{ and } 0 \leq u \leq k^+ \end{array} \right\}.
\]
To prove~\eqref{Delta T = Delta end T}, let $T = (T^+, T^-) \in \mathcal{T}(\sigma)$ with ${\rm end}(T) = \mathbf{E}$;
we must show that for all $\mathbf{S} \subseteq \mathbf{P}$,
\begin{equation}
    \label{Delta T = Delta end T concrete GL}
    \begin{array}{c}
    {\rm width}(\mathbf{S}|_{i> T^-_{t,1}, \: j > T^+_{u,1}}) \leq k - t - u \text{ for all $0 \leq t \leq \ell(\sigma^-)$ and $0 \leq u \leq \ell(\sigma^+)$} \\
    \Updownarrow \\
    {\rm width}(\mathbf{S}|_{i> i_t, \: j > j_u}) \leq k - t - u \text{ for all $0 \leq t \leq k^-$ and $0 \leq u \leq k^+$}.
    \end{array}
\end{equation}
Since by definition $i_0 = T^-_{0,1} = 0$ and $j_0 = T^+_{0,1} = 0$, the case $t=u=0$ in~\eqref{Delta T = Delta end T concrete GL} is done.
Next observe that for all $0 \leq a \leq p$ and $0 \leq b \leq q$,
\begin{equation}
    \label{width observation GL}
    {\rm width}(\mathbf{P}|_{i>a, \: j> b}) = \min\{p-a, \: q-b\}.
\end{equation}
Moreover, $i_t \geq t$ and $j_u \geq u$.
It follows that if $t \leq k-q$, then $k - t - u \geq q-u \geq q - j_u$; therefore, ${\rm width}(\mathbf{P}|_{i > i_t, \: j > j_u}) \leq k-t-u$ for all $t$.
By symmetry, if $u \leq k-p$, then ${\rm width}(\mathbf{P}|_{i>i_t, \: j > j_u}) \leq k-t-u$ for all $u$.
Thus, if $t \leq k-q$ or $u \leq k-p$, then the bottom proposition in~\eqref{Delta T = Delta end T concrete GL} is automatically true;
in either case, since each $T^-_{t,1} \geq i_t$ and $T^+_{u,1} \geq j_u$, the top proposition in~\eqref{Delta T = Delta end T concrete GL} is automatically true as well, and we are done.
Therefore we may assume that $t > k-q$ and $u > k-p$;
thus by~\eqref{end T GL details} we have $j_u = T^+_{u,1}$, and $i_t = \min\{T^-_{t,1}, \hat{\imath}_t\}$, where $\hat{\imath}_t$ is given by~\eqref{i hat GL}.
If $T^-_{t,1} \leq \hat{\imath}_t$, then $i_t = T^-_{t,1}$, and we are done.
On the other hand, if $T^-_{t,1} > \hat{\imath}_t$, then $i_t = \hat{\imath}_t$ as in~\eqref{i hat GL}.
By~\eqref{width observation GL}, we have ${\rm width}(\mathbf{P}|_{i>\hat{\imath}_t, \: j > j_u}) = m \coloneqq \min\{p-\hat{\imath}_t, \: q-j_u\}$.
By~\eqref{P Q}, we have
\[
\underbrace{\{0,1,\ldots,m-1\}}_{\text{cardinality $m$}} \subseteq \underbrace{\{ p-\hat{\imath}_{t'} : t' > t \} \sqcup \{q - j_{u'} : u' > u\}}_{\text{cardinality $k-t-u$}},
\]
and therefore the bottom proposition in~\eqref{Delta T = Delta end T concrete GL} is true.
Since we assumed that $T^-_{t,1} > \hat{\imath}_t$, the top proposition is true as well.
This establishes~\eqref{Delta T = Delta end T concrete GL}, and thus the claim~\eqref{Delta T = Delta end T}.

To prove~\eqref{F Delta E = F E}, let $\F \in \mathcal{F}_{\mathbf{E}}$.
Note that the subset $\F|_{i > i_t, \: j> j_u}$ can be partitioned into $k-t-u$ many chains: namely, its intersections with $\mathbf{L}^+_{u+1}, \ldots, \mathbf{L}^+_{k^+}$ and $\mathbf{L}^-_{t+1}, \ldots, \mathbf{L}^-_{k^-}$, upon extending these chains to cover $\mathbf{A}$ as necessary.
Thus ${\rm width}(\F|_{i > i_t, \: j> j_u}) \leq k - t- u$, and so $\F \in \Delta_{\mathbf{E}}$.
The maximality of $\F$ follows from the same argument as in the $\Sp_{2k}$ case above, where we try to add a point $(i', j') \notin \F$ and observe that its antidiagonal $\mathbf{D}$ intersects precisely the paths $\mathbf{L}^+_{u+1}, \mathbf{L}^+_{u+2}, \ldots, \mathbf{L}^+_{k^+}$ and $\mathbf{L}^-_{t+1}, \mathbf{L}^-_{t+2}, \ldots, \mathbf{L}^-_{k^-}$ for some $t$ and $u$;
then $\mathbf{D} \subseteq \mathbf{P}|_{i> i_t, \: j > j_u}$, and the rest of the argument goes through.
Conversely, just as in the $\Sp_{2k}$ case, any facet $\F \in \mathcal{F}(\Delta_{\mathbf{E}})$ can be decomposed into a disjoint union of saturated chains, which takes the form~\eqref{F} for some element of $\mathcal{F}_{\mathbf{E}}$.

\label{res cor GL} Finally, to prove~\eqref{res = cor}, let $\gamma(\mathbf{S})$ be as in~\eqref{omega}.
For $u = 1,\ldots, k^+$, recursively define
\[
\mathbf{Q}^+_u \coloneqq \{ (i,j) \in \mathbf{P} : (a_u, u) \leq (i,j) \leq (p, j_u) \} \backslash \gamma(\mathbf{Q}^+_{u-1}),
\]
and for $t = k^-, \ldots, 2, 1$, recursively define
\[
\mathbf{Q}^-_t \coloneqq \{ (i,j) \in \mathbf{P} : (t, b_t) \leq (i,j) \leq (i_t, q) \} \backslash \gamma(\mathbf{Q}^-_{t+1}),
\]
where the role of $\gamma(\mathbf{Q}^-_{t+1})$ is played by $\gamma(\mathbf{Q}^+_{k^+})$ in the base case $t = k^-$.
Equip each subposet $\mathbf{Q}^+_u$ and $\mathbf{Q}^-_t$ with the EL-labeling $\alpha^+_u$ or $\alpha^-_t$, respectively, in the same way as in~\eqref{alpha labeling}. 
The rest of the proof is identical to the $\Sp_{2k}$ case above, upon extending the shellings induced by the $\alpha^-_t$'s and $\alpha^+_u$'s to all of $\Delta_{\mathbf{E}}$, by concatenating the label sequences from northeast to southwest.
\end{proof}

\subsection*{The orthogonal group}

Whereas for $\GL_k$ and $\Sp_{2k}$ we appealed to the results of Jackson~\cite{Jackson} in Lemma~\ref{lemma:Jackson Grobner}, there is no analogous result for $\O_k$.
In particular, the treatment of $\O_k$ in~\cite{Jackson} is restricted to the case $k > 2n$ (see pp.~64 and 75), in which case $W$ was cofree and there was no occasion to define analogues of splits or their monomials (see Definition~\ref{def:Sp split} above).
For this reason, we briefly develop the requisite theory below in the special case $\sigma = (1^m)$.

\subsubsection*{Fundamental theorems of tensor invariants for $\O_k$}

We follow the exposition in Lehrer--Zhang~\cite{LZ}*{\S3.1}.
Let $s$ be a positive integer.
Let $\mathcal{M}(2s)$ denote the set of matchings on $[2s]$, where a \emph{matching} is a set
    \begin{equation*}
        \label{matching}
    \bm{\mu} = \{ \{i_1, j_1\}, \ldots, \{i_s, j_s\} \}
    \end{equation*}
of two-element subsets such that $\{i_1, \ldots, i_s, j_1, \ldots, j_s\} = [2s]$.
To each matching $\bm{\mu} \in \mathcal{M}(2s)$, we associate the following complete contraction in $(V^{\otimes 2s})^*$:
    \begin{equation}
        \label{theta_mu}
        \theta_{\bm{\mu}} : v_1 \otimes \cdots \otimes v_{2s} \longmapsto \prod_{\mathclap{\{i,j\} \in \bm{\mu}}} b(v_i, v_j),
    \end{equation}
 where $b$ is the nondegenerate symmetric bilinear form on $V$ that is preserved by~$\O_k$.
By the first fundamental theorem of tensor invariants for $\O_k$ (see~\cite{LZ}*{Thm.~3.1} or~\cite{GW}*{Thm.~5.3.5}), we have
    \begin{equation}
        \label{FFT O}
        [ (V^{\otimes 2s})^* ]^{\O_k} = {\rm span} \Big\{ \theta_{\bm{\mu}} : \bm{\mu} \in \mathcal{M}(2s) \Big\}.
    \end{equation}
    
To describe the relations among these $\theta_{\bm{\mu}}$'s, let $X = \{x_1, \ldots, x_{k+1}\}$ and $Y = \{y_1, \ldots, y_{k+1}\}$ be disjoint subsets of $[2s]$.
We assume the labeling convention $x_1 < \cdots < x_{k+1}$ and $y_1 < \cdots < y_{k+1}$.
Let $\mathfrak{S}_{k+1}$ denote the symmetric group on $k+1$ letters, and let $\tau \in \mathfrak{S}_{k+1}$.
Let $\bm{\nu}$ be a matching on $[2s] \backslash (X \sqcup Y)$.
Define the following matching in $\mathcal{M}(2s)$ by letting $\tau$ act on $X$:
    \[
    \tau(X, Y, \bm{\nu}) \coloneqq \left\{ \{x_{\tau(\ell)}, y_{\ell}\} \right\}_{\ell=1}^{k+1} \sqcup \bm{\nu}.
    \]
Then define
    \begin{equation}
        \label{Theta XY nu}
        \Theta(X,Y,\bm{\nu}) \coloneqq \sum_{\mathclap{\tau \in \mathfrak{S}_{k+1}}} {\rm sgn}(\tau) \: \theta_{\tau(X,Y,\bm{\nu})},
    \end{equation}
where ${\rm sgn}(\tau) \in \{\pm 1\}$ is the signature of $\tau$.
We collect all such elements in the following set:
    \[
        \label{R}
        \mathcal{R} \coloneqq \left\{ \Theta(X,Y,\bm{\nu}) : \begin{array}{l}
            \text{$X,Y$ are disjoint $(k+1)$-element subsets of $[2s]$},\\
            \text{$\bm{\nu}$ is a matching on $[2s] \backslash (X \sqcup Y)$}
        \end{array}
        \right\}.
    \]
    By the second fundamental theorem of tensor invariants for $\O_k$ (see~\cite{LZ}*{Thm.~3.4}), any linear relation among the $\theta_{\bm \mu}$'s is a linear consequence of the relations $\Theta = 0$, for all $\Theta \in \mathcal{R}$.
    
    Combining~\eqref{theta_mu} with~\eqref{Theta XY nu}, we expand each $\Theta \in \mathcal{R}$ as follows:
    \begin{align}
    \Theta(X, Y, \bm{\nu}) : v_1 \otimes \cdots \otimes v_{2s} \longmapsto & \; \left[\sum_{\tau \in \mathfrak{S}_{k+1}} {\rm sgn}(\tau) \: \prod_{\ell = 1}^{k+1} b(v_{x_{\tau(\ell)}}, v_{y_\ell}) \right]    
    \prod_{\{i,j\} \in \bm{\nu}} \!\!\! b(v_i, v_j) \nonumber \\[1ex]
    = & \; \det \! \Big[ b(v_{x_\ell}, v_{y_{m}}) \Big]_{1 \leq \ell, m \leq k+1} \cdot    
    \prod_{\{i,j\} \in \bm{\nu}} \!\!\! b(v_i, v_j). \label{det for Theta}
    \end{align}
    By the straightening law of De Concini--Procesi~\cite{DeConciniProcesi}*{Lemma~5.3}, based on that of Doubilet--Rota--Stein~\cite{DRS}*{Thm.~1, p.~198}, the determinant of the matrix $[ b(v_{x_\ell}, v_{y_{m}})]$ in~\eqref{det for Theta} can be written as a linear combination of determinants of matrices $[b(v_{x'_\ell}, v_{y'_m})]$ such that $x'_\ell \leq y'_\ell$ for all $1 \leq \ell \leq k + 1$.
    Hence it suffices to consider the subset
    \[
        \mathcal{R}' \coloneqq 
        \Big\{ \Theta(X,Y,\bm{\nu}) \in \mathcal{R} : 
            \text{$x_\ell \leq y_\ell$ for all $1 \leq \ell \leq k+1$}
        \Big\},
    \]
    and we may sharpen the second fundamental theorem as follows:
    \begin{equation}
        \label{SFT O}
        \begin{array}{l}
        \text{any linear relation among the $\theta_{\bm \mu}$'s is a linear consequence} \\
        \text{of the relations $\Theta = 0$, for all $\Theta \in \mathcal{R}'$}.
        \end{array}
    \end{equation}
    Note that if $s \leq k$, then $\mathcal{R} = \mathcal{R}' = \varnothing$, and thus by~\eqref{SFT O} the $\theta_{\bm{\mu}}$'s furnish a basis for $[(V^{\otimes 2s})^*]^{\O_k}$.

\subsubsection*{Proving the main result for $\O_k$}

We introduce the following orthogonal analogue of Definition~\ref{def:Sp split}.

\begin{defi}[Monomial of a split for $H = \O_k$]
\label{def:split O}
    An \emph{$\O_{k}$-split of type $\sigma = (1^m)$} consists of four (possibly empty) subsets of $[n]$, namely $A = \{a_i\}_{i=1}^v$, $B = \{b_i\}_{i=1}^r$, $C = \{c_i\}_{i=1}^r$, $D = \{d_i\}_{i=1}^t$, such that $t+v=m$ and $r + t = k+1$, where
    \[
    b_1 < \cdots < b_r < d_1 < \cdots < d_t \text{ and } a_1 < \cdots < a_v < d_1 \text{ and } c_1 < \cdots < c_r \text{ and each } b_i \leq c_i.
    \]
The \emph{monomial} of an $\O_k$-split is
 \begin{equation}
    \label{monomial of O split}
 \left( \prod_{i=1}^r f_{b_i, c_i} \right) \phi_{A \sqcup D},
 \end{equation}
 with $f_{ij}$ and $\phi_T$ as defined in~\eqref{table P fij} and~\eqref{phi O}, respectively.
\end{defi}

\begin{lemma}
    \label{lemma:Ok standard monomials}
    Let $H = \O_k$, and put $\sigma = (1^m)$, where $0 \leq m \leq k$.
    Let $\mathbf{f}$ denote a monomial in the $f_{ij}$'s.
    A linear basis for $M_\sigma$ is given by the set of all monomials $\mathbf{f} \phi_T$ which are not divisible by the monomial of an $\O_k$-split.
\end{lemma}

\begin{proof}
    
    Let $\mathbf{d} = (d_1, \ldots, d_n)$ be an $n$-tuple of nonnegative integers, and put $d \coloneqq d_1 + \cdots + d_n$.
    Let $\C[V^n]_{\mathbf{d}}$ denote the multigraded component of $\C[V^n]$ spanned by monomials $\mathbf{m}$ such that $\deg_{x_{\bullet j}}(\mathbf{m}) = d_j$ for all $1 \leq j \leq n$.
    Likewise, let $M_{\sigma, \mathbf{d}}$ denote the multigraded component of $M_\sigma$ spanned by those maps $\phi$ such that $\deg_{x_{\bullet j}}(\phi) = d_j$ for all $1 \leq j \leq n$.
    Then we have
    \begin{equation}
        \label{M sigma d decomp}
        M_\sigma = \bigoplus_{\mathbf{d} \in \mathbb{N}^n} M_{\sigma, \mathbf{d}}.
    \end{equation}
    Define the surjective $\O_k$-equivariant map
    \[
    \Phi: [(V^{\otimes d+m})^*]^{\O_k} \longrightarrow M_{\sigma, \mathbf{d}}
    \]
    to be the composition of the following canonical maps:
    \begin{align}
        \label{Phi Ok proof}
        \begin{split}
        [(V^{\otimes d+m})^*]^{\O_k} \cong [(V^*)^{\otimes d} \otimes V^{\otimes m}]^{\O_k} \longrightarrow & \left[\left(S^{d_1}(V^*) \otimes \cdots \otimes S^{d_n}(V^*) \right) \otimes \Wedge^m V \right]^{\O_k} \\
        \cong & \;(\C[V^n]_{\mathbf{d}} \otimes \Wedge^m V)^{\O_k} \cong M_{\sigma, \mathbf{d}}.
        \end{split}
    \end{align}
    The first isomorphism in~\eqref{Phi Ok proof} is induced by the canonical isomorphism of $\O_k$-representations $V \cong V^*$ via $v \mapsto b( - , v)$.
    The arrow in the first line of~\eqref{Phi Ok proof} denotes the obvious projection given by symmetrizing each block of $d_j$ tensor factors, and then antisymmetrizing the final $m$ tensor factors.
    We call these the \emph{$\mathbf{d}$-symmetrization} and the \emph{$m$-antisymmetrization} by $\Phi$, respectively.
    The final isomorphism in~\eqref{Phi Ok proof} is given by $f \otimes u \mapsto [ w \mapsto f(w)u]$ for all $f \in \C[V^n]_{\mathbf{d}}$ and $u \in \Wedge^m V$.
    
    For each $i \in [d]$, let 
    \begin{equation}
        \label{i bar}
        \bar{\imath} \coloneqq \text{the unique $j \in [n]$ such that $d_1 + \cdots + d_{j-1} < i \leq d_j$}.
    \end{equation}
    In other words, if one partitions $[d]$ into consecutive ``blocks'' of lengths $d_1, \ldots, d_n$, then $\bar{\imath}$ designates the block containing $i$.
    Let
    \[
    \mathcal{M}(\mathbf{d}|m) \coloneqq \left\{ \bm{\mu} \in \mathcal{M}(d+m) : 
    \begin{array}{l} 
    \text{$\{i_\ell, d+\ell\} \in \bm{\mu}$ implies $i_\ell \in [d]$} \\
    \text{for all $1 \leq \ell \leq m$},\\
    \text{and $\overline{\imath_1}, \ldots, \overline{\imath_m}$ are all distinct}
    \end{array}
    \right\}.
    \]
    For $\bm{\mu} \in \mathcal{M}(\mathbf{d}|m)$, define
    \begin{align}
        \label{mu notation}
        \begin{split}
        \widehat{\bm{\mu}} & \coloneqq \Big\{ \{i_\ell,\: d + \ell\} \in \bm{\mu} : 1 \leq \ell \leq m \Big\},\\
        i(\widehat{\bm{\mu}}) & \coloneqq \{ i_\ell : 1 \leq \ell \leq m \} \subseteq [d],\\
        \overline{\imath}(\widehat{\bm{\mu}}) & \coloneqq \{ \overline{\imath_\ell} : 1 \leq \ell \leq m \} \subseteq [n],\\
        {\rm sgn}(\widehat{\bm{\mu}}) & \coloneqq \text{signature of the sequence } (\overline{\imath_\ell})_{1 \leq \ell \leq m},
        \end{split}
    \end{align}
    meaning the signature of the permutation in $\mathfrak{S}_m$ required to arrange $(\overline{\imath_\ell})$ in increasing order.
    It is straightforward to verify from~\eqref{Phi Ok proof} that for every ${\bm{\mu} \in \mathcal{M}(d+m)}$,
    \begin{equation}
        \label{Phi(theta mu)}
    \Phi(\theta_{\bm{\mu}}) = \begin{cases}
        {\rm sgn}(\widehat{\bm{\mu}}) \left(\displaystyle \prod_{ \{i,j\} \in \bm{\mu} \backslash \widehat{\bm{\mu}}} f_{\overline{\imath}, \overline{\jmath}} \right) \phi_{\overline{\imath}(\widehat{\bm{\mu}})}, & \bm{\mu} \in \mathcal{M}(\mathbf{d}|m),\\[4ex]
        0 & \text{otherwise}.
    \end{cases}
    \end{equation}
    Since $\Phi$ is surjective, the first fundamental theorem~\eqref{FFT O} (where we take $s =(d+m)/2$) implies that $M_{\sigma, \mathbf{d}}$ is spanned by $\{ \Phi(\theta_{\bm{\mu}}) : \bm{\mu} \in \mathcal{M}(\mathbf{d}|m)\}$.
    Thus by~\eqref{Phi(theta mu)} we have
    \begin{equation}
        \label{spanning set O}
        M_{\sigma, \mathbf{d}} = {\rm span}\Big\{ \mathbf{f} \phi_T : \text{$\mathbf{f} \in \C[W]_{\mathbf{d}}$ is a monomial in the $f_{ij}$'s, and $T \in  {\textstyle \binom{[n]}{m}}$} \Big\}.
    \end{equation}
    This leads us to impose the following monomial ordering on $M_{\sigma, \mathbf{d}}$.
    We declare $\mathbf{f} \phi_T < \mathbf{f}' \phi_{T'}$ if $T < T'$ with respect to lexicographic order, or if $T = T'$ and $\mathbf{f} < \mathbf{f}'$ with respect to the ordering 
    \begin{equation}
        \label{monomial ordering}
        f_{11} > \cdots > f_{1n} > f_{22} > \cdots f_{2n} > f_{33} > \cdots > f_{n-1,n} > f_{nn}.
    \end{equation}

    Taking $s = (d+m)/2$ in~\eqref{SFT O} and applying the map $\Phi$ in~\eqref{Phi Ok proof}, we see that any linear relation among the $\Phi(\theta_{\bm{\mu}})$'s (or equivalently, among the $\mathbf{f} \phi_T$'s) is a linear consequence of the relations $\Phi(\Theta) = 0$, for all $\Theta \in \mathcal{R}'$.
    But in fact, we now show that certain of these $\Theta$'s are redundant.
    Given $\Theta(X,Y,\bm{\nu}) \in \mathcal{R}'$, we define the set partitions $X = \widecheck{X} \sqcup \widehat{X}$ and $Y = \widecheck{Y} \sqcup \widehat{Y}$ and $\bm{\nu} = \widecheck{\bm{\nu}} \sqcup \widehat{\bm{\nu}}$ as follows:
    \[
    \begin{array}{lll}
    \widecheck{X} \coloneqq \{x_\ell : y_\ell \leq d \}, & & \widehat{X} \coloneqq \{x_\ell : y_\ell > d \},\\
    \widecheck{Y} \coloneqq \{y : y \leq d \}, & & \widehat{Y} \coloneqq \{y : y > d \},\\
    \widecheck{\bm{\nu}} \coloneqq \{ \{i,j\} \in \bm{\nu} : i,j \leq d \}, & & \widehat{\bm{\nu}} \coloneqq \{ \{i,j\} \in \bm{\nu} : \text{$i$ or $j$ is $> d$} \}.
    \end{array}
    \]
    We also recall the notation $i(\widehat{\bm{\nu}})$ from~\eqref{mu notation}.
    We claim that in order to generate all linear relations among the monomials $\mathbf{f} \phi_T$ in $M_{\sigma, \mathbf{d}}$, it suffices to consider the images under $\Phi$ of the elements of the subset
    \begin{equation}
        \label{R''}
        \mathcal{R}'' \coloneqq 
        \left\{ 
        \Theta(X,Y,\bm{\nu}) \in \mathcal{R}' : 
        \begin{array}{l}
            \text{$X \subseteq [d]$ with distinct $\overline{x}$'s for all $x \in X$,}\\
            \text{distinct $\overline{y}$'s for all $y \in \widecheck{Y}$,}\\
            \text{$\bm{\nu} \subseteq \bm{\mu}$ for some $\bm{\mu} \in \mathcal{M}(\mathbf{d}|m)$}, \\
            \text{and $\overline{\imath} < \overline{x}$ for all $i \in i(\widehat{\bm{\nu}})$ and $x \in \widehat{X}$}
        \end{array}
    \right\}.
    \end{equation}
    We justify the claim as follows.
    If $X \not\subseteq [d]$, then the $m$-antisymmetrization by $\Phi$ (along with Laplace expansion) allows the determinant in~\eqref{det for Theta} to be written as a linear combination of determinants $\det [b(v_{x'}, v_{y'})]_{x' \in X', y' \in Y'}$ such that $X' \subseteq [d]$; therefore we may as well insist that $X \subseteq [d]$.
    If two of the $\overline{x}$'s (or two of the $\overline{y}$'s for $y \in \widecheck{Y}$) are equal, then the $\mathbf{d}$-symmetrization by $\Phi$ kills $\det [b(v_{x}, v_{y})]_{x \in X, y \in Y}$ in the expression~\eqref{det for Theta} for $\Theta(X,Y, \bm{\nu})$; hence $\Phi(\Theta(X,Y,\bm{\nu})$ is identically zero, and yields only the trivial relation $0=0$.
    Likewise, if $\bm{\nu}$ is not a submatching of some $\bm{\mu} \in \mathcal{M}(\mathbf{d}|m)$, then by~\eqref{Phi(theta mu)}, every term in $\Theta(X,Y,\bm{\nu})$ is killed by $\Phi$ (due to the combination of the $\mathbf{d}$-symmetrization and the $m$-antisymmetrization), so we obtain only the trivial relation.
    Finally, recall that the $m$-antisymmetrization by $\Phi$ is a signed average over all permutations of the indices $d+1, \ldots, d+m$, which are precisely the indices paired with the elements of $\widehat{X} \sqcup i(\widehat{\bm{\nu}})$ in the matching $\{ \{x_\ell, y_\ell\}\}_{\ell=1}^{k+1} \sqcup \bm{\nu}$.
    Hence, we may as well restrict our attention to those triples $(X,Y,\bm{\nu})$ such that all the elements $i \in i(\widehat{\bm{\nu}})$ are less than all the elements $x \in \widehat{X}$.
    The $\mathbf{d}$-symmetrization by $\Phi$ then forces $\overline{\imath} < \overline{x}$, as in~\eqref{R''}.
    This establishes the claim above, and we conclude that in $M_{\sigma, \mathbf{d}}$,
    \begin{equation}
        \label{Phi SFT O}
        \begin{array}{l}
        \text{any linear relation among the $\mathbf{f} \phi_T$'s is a linear consequence} \\
        \text{of the relations $\Phi(\Theta) = 0$, for all $\Theta \in \mathcal{R}''$}.
        \end{array}
    \end{equation}

    By comparing~\eqref{R''} with Definition~\ref{def:split O}, we observe that in $\mathcal{R}''$, the equivalence classes modulo $\widecheck{\bm{\nu}}$ are parametrized by the $\O_k$-splits:
    \[
    \Theta(X,Y,\bm{\nu}) \leadsto A = \overline{\imath}(\widehat{\bm{\nu}}), \: B = \{ \overline{x} : x \in \widecheck{X} \}, \: C = \{ \overline{y} : y \in \widecheck{Y}\}, \: D = \{ \overline{x} : x \in \widehat{X}\}.
    \]
    In the context of Definition~\ref{def:split O}, we thus have $r = \# \widecheck{X} = \# \widecheck{Y}$ and $t = \# \widehat{X} = \# \widehat{Y}$ and $v = \# \widehat{\bm{\nu}}$.
    Set
    \[
      \overline{\tau}(\widehat{X}) \coloneqq \{\overline{x_{\tau(\ell)}} : x_\ell \in \widehat{X} \}.
    \]
    By~\eqref{Phi(theta mu)}, we have the following for each $\Theta(X,Y,\bm{\nu}) \in \mathcal{R}''$ with corresponding split $(A,B,C,D)$:  
    \begin{align}
        0 = \Phi\Big( \Theta(X,Y,\bm{\nu}) \Big) &= \Phi \left( \sum_{\tau \in \mathfrak{S}_{k+1}} {\rm sgn}(\tau) \: \theta_{\tau(X,Y,\bm{\nu})} \right) \nonumber \\
        &= \sum_{\tau \in \mathfrak{S}_{k+1}} {\rm sgn}(\tau) \: \Phi \! \left(\theta_{\tau(X,Y,\bm{\nu})} \right) \nonumber \\
        &= \sum_{\tau \in \mathfrak{S}_{k+1}} {\rm sgn}(\tau) \: {\rm sgn}(\widehat{\bm{\nu}}) \Bigg( \prod_{\{i,j\} \in \widecheck{\bm{\nu}}} f_{\overline{\imath}, \overline{\jmath}} \Bigg) \Bigg( \prod_{\substack{
        x \in \widecheck{X}, \\ 
        y \in \widecheck{Y} 
        }} 
        f_{\overline{x}, \overline{y}}\Bigg) \: \phi_{\overline{\imath}(\widehat{\bm{\nu}}) \cup \overline{\tau}(\widehat{X})} \nonumber \\
        &= {\rm sgn}(\widehat{\bm{\nu}}) \prod_{\{i,j\} \in \widecheck{\bm{\nu}}} f_{\overline{\imath}, \overline{\jmath}} \cdot \Bigg[ \sum_{\tau \in \mathfrak{S}_{k+1}}{\rm sgn}(\tau) \Bigg( \prod_{        \substack{
        x \in \widecheck{X}, \\ 
        y \in \widecheck{Y} 
        }
        } 
        f_{\overline{x}, \overline{y}}\Bigg)\: \phi_{\overline{\imath}(\widehat{\bm{\nu}}) \cup \overline{\tau}(\widehat{X})} \Bigg] \nonumber \\      & = 
    {\rm sgn}(\widehat{\bm{\nu}}) \prod_{\{i,j\} \in \widecheck{\bm{\nu}}} f_{\overline{\imath}, \overline{\jmath}} \cdot 
    \left[ 
    \sum_{\pi: B \leftrightarrow D} {\rm sgn}(\pi) \: \pi. \!\left(\det [f_{ij}]_{
    \begin{subarray}{l}
        i \in B\\
        j \in C
    \end{subarray}
    } 
    \cdot \phi_{A \cup D} \right) \right], \label{Phi Theta final}
    \end{align}
   where the sum ranges over all exchanges $\pi$ between $B$ and $D$;
   that is, each $\pi$ acts by exchanging any number $p$ of elements of $B$ with $p$ elements of $D$ (with $0 \leq p \leq \min\{r,t\}$), and ${\rm sgn}(\pi)$ equals $(-1)^p$ times the signatures of the resulting lists $B'$ and $D'$.
   
   The right-hand side of~\eqref{Phi Theta final} contains $r!$ many terms for each exchange $\pi$, corresponding to the terms in the $r \times r$ determinant.
    With respect to the monomial ordering in~\eqref{monomial ordering}, the leading term on the right-hand side of~\eqref{Phi Theta final} occurs within the summand corresponding to $\pi = {\rm id}$, since in a split the elements of $B$ are already less than all the elements of $D$.
    Furthermore, in the summand corresponding to $\pi = {\rm id}$, the leading monomial is the one that contains the main diagonal in the determinant.
    Hence the leading monomial in~\eqref{Phi Theta final} is the monomial of the split $(A,B,C,D)$, times the monomial $\prod_{\{i,j\} \in \widecheck{\bm{\nu}}} f_{\overline{\imath}, \overline{\jmath}}$.
    It follows that if a monomial $\mathbf{f} \phi_T \in M_{\sigma, \mathbf{d}}$ is divisible by the monomial of an $\O_k$-split, then it can be rewritten as a linear combination of strictly lesser monomials with respect to the term order~\eqref{monomial ordering}.
    Thus $M_{\sigma, \mathbf{d}}$ is spanned by the monomials that are not divisible by the monomial of a split, and by~\eqref{Phi SFT O}, these monomials are linearly independent.
    Hence these monomials form a linear basis for $M_{\sigma, \mathbf{d}}$, and the result follows from~\eqref{M sigma d decomp}.
\end{proof}

\begin{proof}[Proof of Theorem~\ref{thm:Stanley decomps and HS} \textup{(}$H = \O_k$\textup{)}]

We adapt the argument used in the proof for $\Sp_{2k}$ and $\GL_k$ above.
We equip $\mathbf{P}$ with a different partial order:
\begin{equation}
    \label{partial order for O}
    (i,j) \leq (i', j') \text{ if and only if } i \geq i' \text{ and } j \leq j'.
\end{equation}
Then for each $T = \{T_1, \ldots, T_m\} \in \mathcal{T}(\sigma) = \binom{[n]}{m}$ where $n+1 \eqcolon T_0 > T_1 > \cdots > T_m$, define the abstract simplicial complex
\begin{equation}
        \label{Delta T O}
        \Delta_T \coloneqq
        \left\{ \mathbf{S} \subseteq \mathbf{P} : \operatorname{width} \left( \mathbf{S}|_{i < T_{t}} \right) \leq k-t, \text{ for all } 0 \leq t \leq m \right\}.
    \end{equation}
By Lemma~\ref{lemma:Ok standard monomials}, we have $M_\sigma = \bigoplus_{T \in \mathcal{T}(\sigma)} ( \bigoplus_{\mathbf{f}} \: \C \, \mathbf{f} ) \phi_T$, 
where the inside direct sum ranges over all monomials $\mathbf{f}$ that are not divisible by the monomial of an $\O_k$-split.
Letting $(*)$ denote this inside direct sum, we claim that $(*) = \C[\Delta_T]$.
To this end, let ${\rm supp}(\mathbf{f}) \coloneqq \{ (i,j) \in \mathbf{P} : \text{$f_{ij}$ divides $\mathbf{f}$} \}$.
Since the monomials of splits are squarefree in the $f_{ij}$'s, if $\mathbf{f} \phi_T$ is divisible by the monomial of a split, then so is $f_{\rm supp(\mathbf{f})} \phi_T$.
In this case, where the split in question is $(A,B,C,D)$, we have $T = A \sqcup D$ in~\eqref{monomial of O split}.
Thus there is some $t$, where $0 \leq t \leq m$, such that $d_1 < b_1 < \cdots < b_r$, where $r + t = k+1$ (and if $t=0$, then omit the $d_1$), and $\{(b_i, c_i)\}_{i=1}^r$ is an antichain contained in ${\rm supp}(\mathbf{f})$.
Equivalently, $\operatorname{width}({\rm supp}(\mathbf{f})|_{i < T_t}) \geq k - t + 1$.
Therefore $\mathbf{f}$ occurs in $(*)$ if and only if $\operatorname{width}({\rm supp}(\mathbf{f})|_{i < T_t}) \leq k - t$, for all $0 \leq t \leq m$.
By~\eqref{Delta T O}, this is equivalent to ${\rm supp}(\mathbf{f}) \in \Delta_T$.
Hence we have $(*) = \C[\Delta_T]$, as claimed, which yields
    \begin{equation}
        \label{M sigma SR O}
        M_\sigma = \bigoplus_{T \in \mathcal{T}(\sigma)} \C[\Delta_T] \phi_T.
    \end{equation}

    For $\mathbf{E} = \{(i_t, n)\}_{t=1}^k \in \E$, where $n+1 \eqcolon i_0 > i_1 > \cdots > i_k$, define the abstract simplicial complex
    \[
    \Delta_{\mathbf{E}} \coloneqq 
    \left\{ \mathbf{S} \subseteq \mathbf{P} : \operatorname{width} \left( \mathbf{S}|_{i < i_t} \right) \leq k-t, \text{ for all } 0 \leq t \leq k \right\}.
    \]
    It remains to establish the three properties~\eqref{Delta T = Delta end T}--\eqref{res = cor}.    
    The property~\eqref{Delta T = Delta end T}, namely that $\Delta_T = \Delta_{{\rm end}(T)}$, follows from the definition of ${\rm end}(T)$ in~\eqref{end T details O} and from the fact that $\operatorname{width}(\mathbf{P}|_{i < a}) = a-1$.
    Property~\eqref{F Delta E = F E} follows from the same argument used in the proof for $\Sp_{2k}$ and $\GL_k$ above;
    recall, however, that for $\O_k$, each $\F \in \mathcal{F}_{\mathbf{E}}$ is determined by $k$ many \emph{north}east lattice paths, which are saturated chains with respect to the partial order~\eqref{partial order for O}.
    To show~\eqref{res = cor}, namely that ${\rm res}(\F) = \cor(\F)$, let $\F = \mathbf{L}_1 \sqcup \cdots \sqcup \mathbf{L}_k \in \mathcal{F}(\Delta_{\mathbf{E}})$, where each $\mathbf{L}_t$ is a northeast lattice path from some diagonal element $(b_t, b_t)$ to $(i_t, n)$.
    Then each $\mathbf{L}_t$ is contained in the lower order ideal $\mathbf{Q}_t \coloneqq \{(i,j) \in \mathbf{P} : (i,j) \leq (i_t, n) \}$.
    Let $\widehat{\mathbf{Q}}_t \coloneqq \mathbf{Q}_t \cup \{\mathbf{0}\}$, where $\mathbf{0}$ is an auxiliary element covered by every diagonal element $(i,i)$.
    Thus $\widehat{\mathbf{Q}}_t$ is a bounded poset, with minimal element $\mathbf{0}$ and maximal element $(i_t, n)$.
    Define the labeling $\alpha_t$ on $\widehat{\mathbf{Q}}_t$ as follows:
    \[
        \label{O alpha}
        \alpha_t\Big( \mathbf{0}, \: (i,i) \Big) = \begin{cases}
            0, & i = i_t,\\
            1, & i > i_t,
        \end{cases}
        \qquad
        \alpha_t\Big( (i,j), \: (i-1, j) \Big) = 0,
        \qquad
        \alpha_t \Big( (i,j), \: (i, j+1) \Big) = 1.
    \]
    Then $\alpha_t$ is an EL-labeling of $\widehat{\mathbf{Q}}_t$, since between any two comparable elements there is a unique saturated chain with label sequence $(0, \ldots, 0)$.
    Thus by Lemma~\ref{lemma:EL labeling induces shelling}, $\alpha_t$ induces a shelling of the order complex $\Delta(\widehat{\mathbf{Q}}_t)$ where the restrictions of the facets are given by the descents in label sequences, and it is clear that the descents correspond to the corners defined in~\eqref{corner O}.
    Each $\mathbf{L}_t$ is a facet of $\Delta(\widehat{\mathbf{Q}}_t)$ with $\mathbf{0}$ removed.
    By concatenating the label sequences of $\mathbf{L}_1, \ldots, \mathbf{L}_k$, and then ordering these concatenations lexicographically (among all $\F$'s), we obtain a shelling of $\Delta_{\mathbf{E}}$ in which the restrictions are the corners, which proves~\eqref{res = cor}.
    The rest of the proof proceeds in the same way as for $\Sp_{2k}$ and $\GL_k$, starting from~\eqref{M sigma SR O} and using the chain of equalities displayed after~\eqref{res = cor} (replacing $z_{ij}$ by $f_{ij}$, and $w_T$ by $\phi_T$).
\end{proof}

\section{Invariant rings for $\SO(V)$ and $\SL(V)$}
\label{sec:Special cases}

In this section, we adapt our jellyfish model to the special linear group $\SL_k \coloneqq \{h \in \GL_k : \det h = 1\}$ and special orthogonal group $\SO_k \coloneqq \{h \in \O_k : \det h = 1\}$, in order to write down Stanley decompositions and Hilbert series of their rings of invariants.
All notation in this section ($M_\sigma$, $f_{ij}$, $\det_T$, $\mathbf{P}$, etc.) in the context of $\SL_k$ (\resp $\SO_k$) is the same as for $\GL_k$ (\resp $\O_k$) above.
We begin with $\SO_k$, which is the much simpler case.

\subsection*{The special orthogonal group}
\label{sub:invariants SO}

Let $V = \C^k$ be the defining representation of $\O_k$, and let $W \coloneqq V^n$ as before.
Recall that for $H = \O_k$, the irreducible representation $U_{(1^k)} \cong \Wedge^k V$ is the one-dimensional representation on which $h \in H$ acts by the scalar $\det h$.
Therefore, we have
\begin{align}
    \C[W]^{\SO_k} & \coloneqq \{ f \in \C[W] : f(hw) = f(w) \text{ for all $h \in \SO_k$ and $w \in W$} \} \nonumber \\
    & = \left\{ f \in \C[W] : \begin{array}{l} f(hw) = (\det h)^s f(w) \text{ for all $h \in \O_k$ and $w \in W$}, \\ \text{where $s=0$ or $1$} \end{array} \right\} \nonumber \\
    &= M_0 \oplus M_{(1^k)}, \label{M0 M1}
\end{align}
where $M_0$ and $M_{(1^k)}$ are modules of covariants~\eqref{M sigma} for $H = \O_k$.
(The module $M_{(1^k)}$ is sometimes called the \emph{module of semiinvariants} for $\O_k$.)
In the proposition below, given some $J \in \binom{[n]}{k}$ containing elements $j_1 < \cdots < j_k$, we write $\det_J$ to denote the function $(v_1, \ldots, v_n) \mapsto \det(v_{j_1}, \ldots, v_{j_k})$.

\begin{prop}
\label{prop:SO invariants}
    Let $V = \C^k$, and set $\widehat{\mathbf{E}} \coloneqq \{(i,n) : 1 \leq i \leq \min\{k,n\} \}$, with all other notation as in the $H = \O_k$ setting of Theorem~\ref{thm:Stanley decomps and HS}.
    We have a Stanley decomposition
     \[
     \C[V^n]^{\SO_k} = 
     \bigoplus_{\mathclap{\F \in \mathcal{F}_{\widehat{\mathbf{E}}}}} \C[f_{ij} : (i,j) \in \F] \: f_{\cor(\F)} \oplus \bigoplus_{J \in \binom{[n]}{k}} \scalebox{.88}{$\displaystyle \left( \bigoplus_{\F \in \mathcal{F}_{\{(j,n): j \in J\}}} \hspace{-2ex} \C[f_{ij} : (i,j) \in \F] \: f_{\cor(\F)} \: {\textstyle \det_{J}} \right)$}.
     \]
     Furthermore, we have the Hilbert--Poincar\'e series
    \[
         P\!\left(\C[V^{n}]^{\SO_k}; t\right) = 
         P_{\widehat{\mathbf{E}}}(t) + t^k \sum_{J \in \binom{[n]}{k}} P_{\{(j,n) : j \in J\}}(t).
    \]
    In particular, if $k > n$, then $\C[V^n]^{\SO_k} = \C[V^n]^{\O_k} = \C[f_{ij} : (i,j) \in \mathbf{P}]$ with Hilbert--Poincar\'e series $1/(1-t^2)^{\binom{n+1}{2}}$.
\end{prop}

\begin{proof}
This follows directly from~\eqref{M0 M1}, upon applying the $H = \O_k$ case of Theorem~\ref{thm:Stanley decomps and HS} and Corollary~\ref{cor:HS} in the two extreme cases $m=0$ and $m=k$.
Recalling from~\eqref{size P delta P} that $\# \delta(\mathbf{P}) = n$, in applying Theorem~\ref{thm:Stanley decomps and HS}, we treat separately the two cases $k < n$ and $k \geq n$.

Suppose $k < n$.
Then $\widehat{\mathbf{E}} = \{ (1,n), \ldots, (k,n) \}$ consists of the $k$ northernmost endpoints in $\delta(\mathbf{P})$.
When $m=0$, we have $\mathcal{T}(0) = \mathcal{T}_{\widehat{\mathbf{E}}}(0) =  \{ \varnothing \}$, by~\eqref{end T details O}.
Thus $\mathcal{J}(0) = \mathcal{F}_{\widehat{\mathbf{E}}} \times \{ \varnothing \}$.
Since $\phi_{\varnothing} = 1$, we apply Theorem~\ref{thm:Stanley decomps and HS} to conclude that $M_0$ equals the first direct sum (ranging over $\F \in \mathcal{F}_{\widehat{\mathbf{E}}}$) in this proposition.
Likewise, when $m=k$, we have $\mathcal{T}(1^k) = \binom{[n]}{k}$, upon identifying a one-column tableau $T \in \SSYT((1^k), n)$ with the set $J = \{T_{1,1}, \ldots, T_{k,1}\}$.
Then we have $\mathcal{T}_{\{(j,n) : j \in J\}}(1^k) = \{J\}$, by~\eqref{end T details O}.
Thus $\mathcal{J}(1^k) = \coprod_{J \in \binom{[n]}{k}} \mathcal{F}_{\{(j,n) : j \in J \}} \times \{ J \}$.
Since $\phi_J = \det_J$ by~\eqref{phi O}, we conclude that $M_1$ equals the second direct sum (ranging over $J$) in this proposition.

Suppose $k \geq n$.
When $m=0$, we have $\mathcal{T}(0) = \{ \varnothing \}$, so that $M_0 = \C[ f_{ij} : (i,j) \in \mathbf{P} ]$ by Theorem~\ref{thm:Stanley decomps and HS}.
But since $\widehat{\mathbf{E}} = \delta(\mathbf{P})$, we have $\mathcal{F}_{\widehat{\mathbf{E}}} = \{ \mathbf{P} \}$, and $\cor(\mathbf{P}) = \varnothing$; thus the first direct sum in this proposition equals $M_0$.
When $m=k$, we have that $M_{(1^k)}$ equals the second direct sum in this proposition, by the same argument as in the $k < n$ case above.

In particular, if $k > n$, then $\mathcal{T}(1^k) = \varnothing$, so that $\mathcal{J}(1^k) = \varnothing$ and we have $M_{(1^k)} = 0$.
Thus in this range, we have $\C[V^n]^{\SO_k} = M_0$, and the given Hilbert series follows from the fact~\eqref{size P delta P} that $\#\mathbf{P} = \binom{n+1}{2}$.
\end{proof}

\subsection*{The special linear group}
\label{sub:invariants SL}

Let $V = \C^k$ be the defining representation of $\GL_k$, and let $W \coloneqq V^{*p} \oplus V^q$ as before.
Recall that for $H = \GL_k$, and for $s \in \mathbb{Z}$, the irreducible representation $U_{(s^k)}$ is the one-dimensional representation on which $h \in H$ acts by the scalar $(\det h)^s$.
Therefore, we have
\begin{align}
    \C[W]^{\SL_k} & \coloneqq \{f \in \C[W] : f(hw) = f(w) \text{ for all $h \in \SL_k$ and $w \in W$} \} \nonumber \\
    &= \left\{f \in \C[W] : \begin{array}{l} f(hw) = (\det h)^s f(w) \text{ for all $h \in \GL_k$ and $w \in W$}, \\ \text{for some $s \in \mathbb{Z}$} \end{array} \right\} \nonumber \\
    &= \bigoplus_{s \in \mathbb{Z}} M_{(s^k)} \nonumber \\
    & = 
    M_0 \oplus \bigoplus_{s > 0} M_{(s^k)} \oplus \bigoplus_{s < 0} M_{(s^k)}, \label{SL invariants = sum GL}
\end{align}
where $M_{(s^k)}$ is a module of covariants~\eqref{M sigma} for $H = \GL_k$.
(These modules, where $s \neq 0$, are sometimes called \emph{modules of semiinvariants} for $\GL_k$.)

In what follows, we use the letters $J$ and $I$ in a consistent way, such that  $J \in \binom{[q]}{k}$ and $I \in \binom{[p]}{k}$.
Note that these sets are empty if $k>q$ (\resp $k>p$).
In our notation, the $J$'s are accompanied by plus signs, and the $I$'s by minus signs; this will make the proof of Proposition~\ref{prop:SL invariants} as transparent as possible.
Recalling the sets $\mathcal{F}_{\mathbf{E}}$ and $\mathcal{T}_{\mathbf{E}}(\sigma)$ from the $\GL_k$ setting of~\eqref{F_E} and~\eqref{T_E}, respectively, we introduce the shorthand
\begin{equation}
    \label{F T shorthand}
    \begin{array}{lll}
     \mathcal{F}^+_J \coloneqq \mathcal{F}_{\{(p,j) : j \in J\}}, & & 
     \mathcal{F}^-_I \coloneqq \mathcal{F}_{\{(i,q) : i \in I\}},\\[2ex]
     \mathcal{T}^+_J(\sigma) \coloneqq \mathcal{T}_{\{(p,j) : j \in J\}}(\sigma), & & \mathcal{T}^-_I(\sigma) \coloneqq \mathcal{T}_{\{(i,q) : i \in I\}}(\sigma).
    \end{array}
\end{equation}

We will write $\mathbf{C}$ to denote a maximal chain in $\binom{[q]}{k}$ or $\binom{[p]}{k}$, with respect to the tableau order; equivalently, $\mathbf{C}$ denotes a facet of the order complex of either poset.
We have the covering relation $J \lessdot J'$ if and only if $J'$ is obtained by adding 1 to exactly one element of $J$.
Thus we can define the following labeling $\alpha$, for all $J \lessdot J'$:
\begin{equation}
    \label{alpha SL}
    \alpha(J, J') = \ell \Longleftrightarrow \text{$J'$ is obtained by adding 1 to the $\ell$th smallest element of $J$}.
\end{equation}
We abbreviate this by $
J \overset{\ell}{\lessdot} J'$.
By Definition~\ref{def:EL-labeling}, this $\alpha$ is an EL-labeling of both $\binom{[q]}{k}$ and $\binom{[p]}{k}$, since the lexicographically minimal path between two elements $J < J'$ is obtained by repeatedly incrementing the elements of $J$ as necessary, from smallest to largest.
Thus by Lemma~\ref{lemma:EL labeling induces shelling}, $\alpha$ induces a shelling on the order complex of $\binom{[q]}{k}$ or $\binom{[p]}{k}$.
By~\eqref{R equals descents}, the restrictions are given by descents:
\begin{equation}
    \label{res(C)}
    {\rm res}(\mathbf{C}) = \Big\{ J \in \mathbf{C} : \text{$\mathbf{C}$ contains $A \overset{\ell}{\lessdot} J \overset{m}{\lessdot} B$ such that $\ell > m$} \Big\}.
\end{equation}

Let $J \in \binom{[q]}{k}$, and consider the upper order ideal $\binom{[q]}{k}_{\geq J}$ consisting of those elements $J'$ such that $J' \geq J$.
Define $\binom{[p]}{k}_{\geq I}$ in the same way.
Let
\begin{align}
    \label{CI CJ}
    \begin{split}
    \mathcal{C}^+_J &\coloneqq \Big\{ \mathbf{C} :  \text{$\mathbf{C}$ is a maximal chain in $\textstyle\binom{[q]}{k}_{\geq J}$} \Big\},\\
    \mathcal{C}^-_I &\coloneqq \Big\{ \mathbf{C} :  \text{$\mathbf{C}$ is a maximal chain in $\textstyle\binom{[p]}{k}_{\geq I}$} \Big\}.
    \end{split}
\end{align}
For $\mathbf{C} \in \mathcal{C}^+_J$ or $\mathcal{C}^-_I$, respectively, define
\begin{equation}
    \label{det C}
     {\textstyle \det_{\mathbf{C}}} \coloneqq {\textstyle \det_J} \prod_{J' \in {\rm res}(\mathbf{C})} {\textstyle \det_{J'}}, \qquad {\textstyle \det^*_{\mathbf{C}}} \coloneqq {\textstyle \det^*_I} \prod_{I' \in {\rm res}(\mathbf{C})} {\textstyle \det^*_{I'}}.
\end{equation}
Analogously to the set of jellyfish from Definition~\ref{def:jellyfish}, we define
\begin{equation}
\label{KJ KI}
\mathcal{K}^+ \coloneqq \coprod_{\mathclap{J \in \binom{[q]}{k}}} \mathcal{F}^+_J \times \mathcal{C}^+_J, \qquad \mathcal{K}^- \coloneqq \coprod_{\mathclap{I \in \binom{[p]}{k}}} \mathcal{F}^-_I \times \mathcal{C}^-_I.
\end{equation}
A straightforward calculation shows that every element of $\mathcal{C}^+_{J}$ (respectively, $\mathcal{C}^-_I$) has the same cardinality $c^+_J$ (\resp $c^-_I$), namely
\begin{equation}
    \label{cI cJ}
     c^+_J = \frac{k}{2}(2q - k -1) - \sum_{j \in J} j \qquad \text{and} \qquad c^-_I = \frac{k}{2}(2p - k -1) - \sum_{i \in I} i.
\end{equation}
Finally, recalling the rational function $P_\mathbf{E}(t)$ from Corollary~\ref{cor:HS}, we define the shorthand
\begin{equation}
    \label{PI PJ}
    P^+_J(t) \coloneqq P_{\{(p,j):j \in J\}}(t), \qquad P^-_I(t) \coloneqq P_{\{(i,q):i\in I\}}(t),
\end{equation}
along with the new rational functions
\begin{equation}
    \label{QI QJ}
    Q^+_J(t) \coloneqq \frac{\sum_{\mathbf{C} \in \mathcal{C}^+_J} (t^k)^{\# {\rm res}(\mathbf{C})+1}}{(1-t^k)^{c^+_J}}, \qquad Q^-_I(t) \coloneqq \frac{\sum_{\mathbf{C} \in \mathcal{C}^-_I} (t^k)^{\# {\rm res}(\mathbf{C}) +1}}{(1-t^k)^{c^-_I}}.
\end{equation}

\begin{prop}
\label{prop:SL invariants}
Let $V = \C^k$.
If $k < \min\{p,q\}$, then we have a Stanley decomposition
\begin{align}
    \label{SD SL}
    \begin{split}
    \C[V^{*p} \oplus V^q]^{\SL_k} = \phantom{\oplus} &\hspace{2ex} \bigoplus_{\mathclap{\F \in \mathcal{F}^-_{\{p-k+1, \ldots, p\}}}} \hspace{2.5ex} \C[f_{ij} : (i,j) \in \F] \: f_{\cor(\F)}\\
    \oplus &\bigoplus_{(\F, \mathbf{C}) \in \mathcal{K}^+} \C\Big[\{f_{ij} : (i,j) \in \F\} \cup \{{\textstyle \det_J} : J \in \mathbf{C} \} \Big] \: f_{\cor(\F)} \: {\textstyle \det_{\mathbf{C}}}\\
    \oplus &\bigoplus_{(\F, \mathbf{C}) \in \mathcal{K}^-} \C\Big[\{f_{ij} : (i,j) \in \F\} \cup \{{\textstyle \det^*_I} : I \in \mathbf{C} \} \Big] \: f_{\cor(\F)} \: {\textstyle \det^*_{\mathbf{C}}},
    \end{split}
    \end{align}
where $f_{ij}$, $\det^*_I$, and $\det_J$ are defined for $H = \GL_k$ in~\eqref{table P fij} and~\eqref{table dets GL Sp}, and $\mathcal{K}^\pm$ in~\eqref{KJ KI}, and $\det^*_{\mathbf{C}}$ and $\det^*_{\mathbf{C}}$ in~\eqref{det C}. Furthermore, we have the Hilbert--Poincar\'e series
\begin{equation}
    \label{Hilbert series SL invariants}
    P \! \left(\C[V^{*p} \oplus V^q]^{\SL_k}; t \right) = P^+_{\{q-k+1, \ldots, q\}}(t) + \sum_{\mathclap{\substack{J \in \binom{[q]}{k}}}} P^+_{J}(t) Q^+_{J}(t) + \sum_{\mathclap{\substack{I \in \binom{[p]}{k}}}} P^-_I(t) Q^-_I(t).
\end{equation}
If $k \geq \min\{p,q\}$, then replace the first direct sum in~\eqref{SD SL} by $\C[ f_{ij} : (i,j) \in \mathbf{P} ]$, and replace the first term in~\eqref{Hilbert series SL invariants} by $\frac{1}{(1-t^2)^{pq}}$.

\noindent In particular, if $k > q$ (\resp $k > p$), then the second (\resp third) direct sum in~\eqref{SD SL} is empty, and thus the second (\resp third) term in~\eqref{Hilbert series SL invariants} is $0$.
\end{prop}

\begin{proof}
Just as in Proposition~\ref{prop:SO invariants}, this result will follow from~\eqref{SL invariants = sum GL} and Theorem~\ref{thm:Stanley decomps and HS};
here, however, we must first use Stanley--Reisner theory in order to collapse the infinite direct sums in~\eqref{SL invariants = sum GL} into finite direct sums.

Note that for $T \in \mathcal{T}(s^k)$, we have $T = (T^+, \varnothing)$ if $s \geq 0$, and $T = (\varnothing, T^-)$ if $s < 0$.
We may identify $\phi_T$ with the function $\det_T$ defined in~\eqref{table dets GL Sp}.
It follows that for all $T \in \mathcal{T}(s^k)$, we have 
\begin{equation}
    \label{phi in SL proof}
    \phi_T = {\textstyle \det_T} = \begin{cases}
        1, & s = 0,\\
        \prod_{\text{columns $J$ of $T^+$}} \textstyle \det_J , & s > 0, \\

        \prod_{\text{columns $I$ of $T^-$}} \textstyle \det^*_I, & s < 0.
    \end{cases}
\end{equation}
We now apply Theorem~\ref{thm:Stanley decomps and HS} to the three components in~\eqref{SL invariants = sum GL}, corresponding to $s=0$, $s>0$, and $s<0$.
We suppose $k < \min\{p,q\}$.

For $s=0$, note that $i_t = \hat{\imath}_t = p-t+1$ in~\eqref{end T GL details}, so that $\mathcal{T}(0) = \mathcal{T}_{\widehat{\mathbf{E}}}(0) = \{ \varnothing \}$, where $\widehat{\mathbf{E}} \coloneqq \{(p-t+1, q) : 1 \leq t \leq k\}$.
Thus $\mathcal{J}(0) = \mathcal{F}_{\widehat{\mathbf{E}}} \times \{ \varnothing \}$.
In the notation of~\eqref{F T shorthand}, we have $\mathcal{F}_{\widehat{\mathbf{E}}} = \mathcal{F}_{\{p-t+1 : 1 \leq t \leq k\}}$, and so we conclude that $M_0$ equals the first direct sum in~\eqref{SD SL}.

It remains to show that the $s>0$ and $s<0$ components in~\eqref{SL invariants = sum GL} equal the second and third direct sums in~\eqref{SD SL}.
We will prove only the $s>0$ case, since the $s<0$ case is identical.
Let
\[
\Delta_{\geq J} \coloneqq \text{the order complex of $\binom{[q]}{k}_{\geq J}$}
\]
with Stanley--Reisner ring $\C[\Delta_{\geq J}]$ in the indeterminates $z_{J'}$, for all $J' \geq J$.
Note that $\mathcal{C}^+_J$, defined in~\eqref{CI CJ}, is precisely the facet set $\mathcal{F}(\Delta_{\geq J})$.
By~\eqref{cI cJ}, $\Delta_{\geq J}$ is pure.
On one hand, a monomial~$\mathbf{z}$ lies in $\C[\Delta_{\geq J}]$ if and only if one can write $\mathbf{z} = z_{J_1} \cdots z_{J_r}$ such that the columns $J_1 \leq \cdots \leq J_r$ form a tableau in $\SSYT((r^k), q)$, with $J \leq J_1$.
On the other hand, the shelling induced by the labeling $\alpha$ in~\eqref{alpha SL} restricts to a shelling of $\Delta_{\geq J}$.
This gives us two ways to decompose $\C[\Delta_{\geq J}]$:
\begin{equation}
    \label{T sum = C sum}
    \C[\Delta_{\geq J}] = \bigoplus_{r \geq 0} \Bigg(\bigoplus_{\substack{T \in \SSYT((r^k), q): \\ T_{\bullet,1} \geq J}} \C \cdot \prod_{\mathclap{\substack{\text{cols.~$J'$} \\ \text{of $T$}}}} z_{J'} \Bigg) = \bigoplus_{\mathbf{C} \in \mathcal{C}^+_J} \C[ z_{J'} : J' \in \mathbf{C} ] \: z_{{\rm res}(\mathbf{C})},
\end{equation}
where the Stanley decomposition on the right-hand side follows from~\eqref{Stanley decomp SR ring}.

Let $s>0$.
It follows from the ``end'' map in~\eqref{end T GL details} that $\mathcal{T}_{\mathbf{E}}(s^k)$ is nonempty if and only if $\mathcal{T}_{\mathbf{E}}(s^k) = \mathcal{T}^+_{J}(s^k)$ for some $J \in \binom{[q]}{k}$.
Therefore by Definition~\ref{def:jellyfish}, we have
\begin{equation}
    \label{Jmk}
    \mathcal{J}(s^k) = \coprod_{J \in \binom{[q]}{k}} \mathcal{F}^+_J \times \mathcal{T}^+_J(s^k).
\end{equation}
Explicitly, we have 
\begin{equation}
    \label{T plus J}
    \mathcal{T}^+_J(s^k) = \Big\{ (T, \varnothing) : T \in \SSYT((s^k), q) \text{ with $T_{\bullet,1} = J$} \Big\}.
\end{equation}
Thus the $s>0$ component in~\eqref{SL invariants = sum GL} can be decomposed using Theorem~\ref{thm:Stanley decomps and HS} followed by~\eqref{Jmk}:
\begin{align}
    \bigoplus_{s>0} M_{(m^k)} & = \bigoplus_{s>0} \left(\bigoplus_{(\F, T) \in \mathcal{J}(s^k)} \C[f_{ij} : (i,j) \in \F] \: f_{\cor(\F)} \: \phi_T \right) \nonumber \\
    & = \bigoplus_{s>0} \left(\bigoplus_{J \in \binom{[q]}{k}} \left(
    \bigoplus_{\F \in \mathcal{F}^+_J} \C[f_{ij} : (i,j) \in \F] \:f_{\cor(\F)}  \left( \bigoplus_{T \in \mathcal{T}^+_J(s^k)} \C \phi_T \right)\right)\right) \nonumber \\
    &= \bigoplus_{J \in \binom{[q]}{k}} \Bigg( \bigoplus_{\F \in \mathcal{F}^+_J} \C[f_{ij} : (i,j) \in \F] \:f_{\cor(\F)} \underbrace{\left( \bigoplus_{s>0} \left( \bigoplus_{T \in \mathcal{T}^+_J(s^k)} \C \phi_T  \right) \right)}_{(*)}  \Bigg), \label{end of M decomp}
\end{align}
where
\begin{align*}
    (*) = \bigoplus_{s>0} \Bigg( \bigoplus_{T \in \mathcal{T}^+_J(s^k)} \C \phi_T  \Bigg) &= \bigoplus_{s > 0} \Bigg( \bigoplus_{T \in \mathcal{T}^+_J(s^k)} \C \cdot \prod_{\mathclap{\substack{\text{cols.~$J'$}\\ \text{of $T$}}}} {\textstyle \det_{J'}} \Bigg) & \text{by~\eqref{phi in SL proof}} \\
    &= \bigoplus_{s>0} \Bigg( \bigoplus_{\substack{T \in \SSYT((s^k), q): \\ T_{\bullet,1} = J}} \C \cdot \prod_{\mathclap{\substack{\text{cols.~$J'$}\\ \text{of $T$}}}} {\textstyle \det_{J'}}   \Bigg) & \text{by~\eqref{T plus J}} \\
    &= \bigoplus_{r \geq 0} \Bigg( \bigoplus_{\substack{T \in \SSYT((r^k), q): \\ T_{\bullet,1} \geq J}} \C \cdot \prod_{\mathclap{\substack{\text{cols.~$J'$}\\ \text{of $T$}}}} {\textstyle \det_{J'}}   \Bigg) {\textstyle \det_J} & \text{($r = s-1$)} \\
    &= \Bigg(\bigoplus_{\mathbf{C} \in \mathcal{C}^+_J} \C[ {\textstyle \det_{J'}} : J' \in \mathbf{C}] \: \prod_{\mathclap{J' \in {\rm res}(\mathbf{C})}} {\textstyle \det_{J'}} \Bigg) {\textstyle \det_{J}} & \text{by~\eqref{T sum = C sum}} \\
    &= \bigoplus_{\mathbf{C} \in \mathcal{C}^+_J} \C[{\textstyle \det_{J'}} : J' \in \mathbf{C}] \: {\textstyle \det_{\mathbf{C}}} & \text{by~\eqref{res(C)}}.
\end{align*}
Substituting this direct sum for $(*)$ in~\eqref{end of M decomp}, we have
\begin{align*}
    \bigoplus_{s>0} M_{(s^k)} & = \bigoplus_{J \in \binom{[q]}{k}} \left( \bigoplus_{\F \in \mathcal{F}^+_J} \C[ f_{ij} : (i,j) \in \F] \: f_{\cor(\F)} \left( \bigoplus_{\mathbf{C} \in \mathcal{C}^+_J} \C[{\textstyle \det_{J'}} : J' \in \mathbf{C}] \: {\textstyle \det_{\mathbf{C}}}  \right) \right) \\
    &= \bigoplus_{(\F, \mathbf{C}) \in \mathcal{K}^+} \C[\{f_{ij} : (i,j) \in \F\} \cup \{ {\textstyle \det_J} : J \in \mathbf{C} \} ] \: f_{\cor(\F)} \: {\textstyle \det_{\mathbf{C}}},
\end{align*}
where the last equality follows from the definition of $\mathcal{K^+}$ in~\eqref{KJ KI}.
Thus the $s>0$ component in~\eqref{SL invariants = sum GL} equals the second direct sum in~\eqref{SD SL}.
The proof for the $s<0$ component is identical.
The Hilbert--Poincar\'e series follows immediately from the Stanley decomposition~\eqref{SD SL}, using~\eqref{Hilbert series from Stanley decomp general}, along with the definitions~\eqref{PI PJ} and~\eqref{QI QJ}, and the fact that the det functions in this setting all have degree $k$.

Finally, suppose that $k \geq \min\{p,q\}$.
Consider the $s=0$ case, and recall that $\mathcal{T}(0) = \{ \varnothing \}$.
Theorem~\ref{thm:Stanley decomps and HS} splits into the two cases $k < \#\delta(\mathbf{P})$ and $k \geq \#\delta(\mathbf{P})$, where we recall from~\eqref{size P delta P} that $\#\delta(\mathbf{P}) = p + q - 1$.
If $k < p + q - 1$, then it is easy to check using~\eqref{end T GL details} that for the unique $\widehat{\mathbf{E}} \in \E$ such that $\mathcal{T}(0) = \mathcal{T}_{\widehat{\mathbf{E}}}(0)$, we have $\mathcal{F}_{\widehat{\mathbf{E}}} = \{ \mathbf{P} \}$.
Thus $\mathcal{J}(0) = \{(\mathbf{P}, \varnothing)\}$, and $\cor(\mathbf{P}) = \varnothing$, so that $M_0 = \C[ f_{ij} : (i,j) \in \mathbf{P}]$.
Likewise, if $k > p+q-1$, we obtain the same result directly by Theorem~\ref{thm:Howe duality}.
The Hilbert series $P(M_0;t) = 1/(1-t^2)^{pq}$ follows from the fact~\eqref{size P delta P} that $\#\mathbf{P} = pq$.
\end{proof}

\begin{exam}
    \label{ex:SL_jellyfish}
    Let $k=3$, $p=8$, and $q=10$.
    Below is an example of a pair $(\F, \mathbf{C}) \in \mathcal{K}^-$ indexing one of the Stanley spaces in Proposition~\ref{prop:SL invariants}:
    \tikzstyle{corner}=[rectangle,draw=black,thin, minimum size = 5pt, inner sep=0pt]
\tikzstyle{endpt}=[circle,fill=lightgray, minimum size = 5pt, inner sep=0pt]
\tikzstyle{dot}=[circle,fill=black, minimum size = 3pt, inner sep=0pt]

\begin{center}
\begin{tikzpicture}[scale=.3]

\draw [white,fill=lightgray] (.5,8.5) -- ++ (2,0) -- ++(0,-1) -- ++ (-1,0) -- ++(0,-1) -- ++ (-1,0) -- cycle;

\draw [densely dotted]
(.5,8.5) rectangle (10.5,.5);

\draw[line width=2pt, lightgray] (3,8) -- ++(6,0) node[corner]{} -- ++(0,-1) -- ++ (1,0) \foreach \x/\y in {1/0,1/0,1/0,1/0,1/-1,1/-1,1/0,1/0,1/-1,1/0,1/-1}{
-- ++(\x,\y) node[endpt]{}
}
(2,7) -- ++(4,0) node[corner]{} -- ++(0,-1) -- ++(2,0) node[corner]{} -- ++(0,-1) -- ++ (2,0) \foreach \x/\y in {1/0,1/0,1/-1,1/0,1/0,1/0,1/-1,1/0,1/0,1/-1,1/0}{
-- ++(\x,\y) node[endpt]{}
}
(1,6) -- ++(3,0) node[corner]{} -- ++(0,-1) -- ++(2,0) node[corner]{} -- ++(0,-1) -- ++(4,0) \foreach \x/\y in {1/0,1/-1,1/0,1/-1,1/0,1/0,1/0,1/-1,1/0,1/0,1/0}{
-- ++(\x,\y) node[endpt]{}
};

\foreach \x in {1,2,...,10}{         \foreach \y in {1,2,...,8}{        \node [dot] at (\x,\y) {};}}

\node (A) at (12,7) {};
\node (B) at (12,3) {};
\node (C) at (14,2) {};
\node (D) at (14,7) {};
\node (E) at (18,1) {};
\node (F) at (18,5) {};
\node (G) at (20,1) {};
\node (H) at (20,4) {};
\node[draw, thick, fit=(A.center) (B.center)] {};
\node[draw, thick, fit=(C.center) (D.center)] {};
\node[draw, thick, fit=(E.center) (F.center)] {};
\node[draw, thick, fit=(G.center) (H.center)] {};

\end{tikzpicture}
\end{center}

    \noindent 
    In particular, $(\F, \mathbf{C})$ belongs to $\mathcal{F}^-_{I} \times \mathcal{C}^-_I$, where $I = \{2,4,5\}$.
    Hence $\mathbf{C}$ is a maximal chain $\{2,4,5\} \lessdot \cdots \lessdot \{6,7,8\}$ in $\binom{[p]}{k}_{\geq I}$.
    Here we depict $\mathbf{C}$ in the same way as we depicted tableaux $T \in \mathcal{T}(\sigma)$ in Section~\ref{sec:Stanley decomps main}, so as to extend the lattice paths of $\F$ into ``tentacles.''
    Thus each $I' \in \mathbf{C}$ is depicted as a column of dots.
    The long vertical rectangles indicate the elements of ${\rm res}(\mathbf{C})$, defined in~\eqref{res(C)}:
    at each such element, the left-hand descending tentacle lies below the right-hand descending tentacle.
    In particular, with respect to the labeling $\alpha$ in~\eqref{alpha SL}, $\mathbf{C}$ has label sequence $(3,2,3,1,1,2,1,2,1)$, whose four descents correspond to the four elements of ${\rm res}(\mathbf{C})$.
\end{exam}

\begin{exam}
    Let $k=3$, $p=3$, and $q=4$.
    We abbreviate $P^\pm_{\{a,b,c\}}(t)$ by $P^\pm_{abc}(t)$.
    Implementing~\eqref{Hilbert series SL invariants} leads to the following Hilbert series of $\C[W]^{\SL_k}$:
    \begin{align*}
        & \phantom{+} P^-_{123}(t) \\
        & + P^-_{123}(t) \: Q^-_{123}(t) \\
        & + P^+_{123}(t) \: Q^+_{123}(t) + P^+_{124}(t) \: Q^+_{124}(t) + P^+_{134}(t) \: Q^+_{134}(t) + P^+_{234}(t) \: Q^+_{234}(t) \\[2ex]
        = & \phantom{+} 
        \frac{1}{(1-t^2)^{12}} \\
        & + \frac{1}{(1-t^2)^{12}} \cdot \frac{t^3}{(1-t^3)}\\
        & + \scalebox{.8}{$\displaystyle \frac{1}{(1-t^2)^{9}} \cdot \frac{t^3}{(1-t^3)^4} + \frac{1 + 2t^2}{(1-t^2)^{10}} \cdot \frac{t^3}{(1-t^3)^{3}} + \frac{1+t^2+t^4}{(1-t^2)^{11}} \cdot \frac{t^3}{(1-t^3)^2} + \frac{1}{(1-t^2)^{12}} \cdot \frac{t^3}{(1-t^3)}$} \\[2ex]
        = & \phantom{+} \frac{1 + 3 t^2 + 2 t^3 + 6 t^4 + 3 t^5 + 8 t^6 + 3 t^7 + 6 t^8 + 2 t^9 +  3 t^{10} + t^{12}}{(1 - t^2)^9 (1 - t^3)^3 (1 - t^6)}.
    \end{align*}
    (We also verified this using Macaulay2.)
    This reduced form is somewhat surprising, since the denominator contains the factor $(1-t^6)$ in addition to the expected factors $(1-t^2)$ and $(1-t^3)$.
    Furthermore, the numerator is palindromic but not unimodal.
    
\end{exam}

\section{Weight bases via arc diagrams}
\label{sec:linear bases}

Recall that via Howe duality (Theorem~\ref{thm:Howe duality}), the modules of covariants $M_\sigma$ for a classical group $H$ can be viewed as $(\g,K)$-modules $L_{\la(\sigma)}$ of unitary highest weight representations of a real reductive group $G_{\mathbb{R}}$.
In this section, we manipulate our jellyfish diagrams in order to realize a weight basis for these modules $L_{\la(\sigma)}$.
Recall from~\eqref{M tilde} that when viewing $M_\sigma$ as a $(\g,K)$-module, we write $\widetilde{M}_\sigma$ to emphasize that the $\k$-action is given by tensoring the natural $\k$-action with a one-dimensional $\k$-module $F_{-kc\zeta}$. 
By Theorem~\ref{thm:Stanley decomps and HS}, $\widetilde{M}_\sigma$ has an infinite linear basis consisting of \emph{standard monomials}, where a standard monomial $\xi$ is the product of a monomial in the $f_{ij}$'s with some $\phi_T$.
In particular, for each  standard monomial $\xi$ in $\widetilde{M}_\sigma$, there is a unique jellyfish $(\F, T) \in \mathcal{J}(\sigma)$ whose corresponding Stanley space contains $\xi$.
Thus one can produce standard monomials from the diagram of a jellyfish $(\F, T)$ by arbitrarily ``populating'' the points $(i,j) \in \F$ with their corresponding $f_{ij}$'s.
It turns out that these standard monomials are weight vectors with respect to the $\g$-action, and by ``flattening'' this picture into an arc diagram, we can read off the weight of a standard monomial directly from its degree sequence.

\begin{defi}[Arc diagram of a standard monomial in $\widetilde{M}_\sigma$]
\label{def:arc diagram}
    Assume the hypotheses of Theorem~\ref{thm:Stanley decomps and HS}.
    Let $(\F, T) \in \mathcal{J}(\sigma)$, and let 
    \[
    \mathbf{f} = \prod_{(i,j) \in \F} f_{ij}^{d_{ij}}
    \]
    for nonnegative integers $d_{ij}$.
    The \emph{arc diagram} of the standard monomial
    \[
    \xi = \mathbf{f} \cdot f_{\cor(\F)} \cdot \phi_{T}.
    \]
    is constructed as follows (see Examples~\ref{ex:arc diagram GL} and~\ref{ex:arc diagram Sp}):
    \begin{enumerate}
        \item Starting with the usual diagram of the jellyfish $(\F, T)$, place a box labeled $d_{ij}$ at each point $(i, j) \in \F$ such that $d_{ij} \neq 0$.
        (Make these boxes small enough so that the corners in $\cor(\F)$ are still visible outside them.)

        \item ``Flatten'' the resulting picture into an arc diagram as follows:

        \begin{itemize}
            \item If $H = \GL_k$, then draw vertices $1^*, \ldots, p^*, 1, \ldots, q$ in a line from left to right.
            The ``tentacles'' attached outside eastern edge of $\mathbf{P}$ remain attached to their corresponding (starred) vertices from below;
            likewise, the tentacles attached outside the southern edge of $\mathbf{P}$ remain attached to their corresponding (unstarred) vertices from below.

            On either side (starred and unstarred) of the arc diagram, each horizontal cross section of the tentacles is viewed as a hyperedge connecting the vertices lying directly above its dots.
            
            For all $i$ and $j$, draw $d_{ij}$ many arcs from vertex $i^*$ to vertex $j$.
            Then for each $(i,j) \in \cor(\F)$, draw an additional arc from vertex $i^*$ to vertex $j$.
            
            \item If $H = \O_k$ or $\Sp_{2k}$, then draw vertices $1, \ldots, n$ in a line from left to right.
            The ``tentacles'' attached outside the eastern edge of $\mathbf{P}$ remain attached to their corresponding vertices from below.

            Each horizontal cross section of the tentacles is viewed as a hyperedge connecting the vertices directly above its dots.

            For all $i$ and $j$, draw $d_{ij}$ many arcs from vertex $i$ to vertex $j$.
            Then for each $(i,j) \in \cor(\F)$, draw an additional arc from vertex $i$ to vertex $j$.           
       \end{itemize}
    \end{enumerate}
    As usual, the \emph{degree sequence} of an arc diagram is the sequence 
\[
    \begin{cases} 
    (\deg 1^*, \ldots, \deg p^* \mid \deg 1, \ldots, \deg q), \quad & H = \GL_k, \\
    (\deg 1, \ldots, \deg n), \quad & H = \O_{k} \text{ or } \Sp_{2k},
    \end{cases}
\]
where $\deg i$ is the number of edges (including hyperedges) that are incident to vertex~$i$.
\end{defi}

\begin{prop}
\label{prop:weight from degree sequence}
    Let $(H, \g)$ be one of the three pairs in Theorem~\ref{thm:Howe duality}.
    Let $\xi \in \widetilde{M}_{\sigma}$ be a standard monomial as in Definition~\ref{def:arc diagram}.
    Then $\xi$ is a weight vector under the action of $\g$, whose weight ${\rm wt}_\g(\xi)$ is obtained from the degree sequence of its arc diagram as follows:
    \[
    {\rm wt}_\g(\xi) = \begin{cases}
    (- \deg 1^*, \ldots, - \deg p^* \mid \deg 1, \ldots, \deg q) - (k^p \mid 0^q), & (H, \g) = (\GL_k, \gl_{p+q}),\\
    (- \deg 1, \ldots, - \deg n) - \left((\tfrac{k}{2})^n\right), & (H,\g) = (\O_k, \sp_{2n}), \\
    (- \deg 1, \ldots, - \deg n) - (k^n), &  (H,\g) = (\Sp_{2k}, \so_{2n}).
   \end{cases}
   \]
\end{prop}

Before proving Proposition~\ref{prop:weight from degree sequence}, we illustrate it with the following examples.

\begin{exam}
\label{ex:arc diagram GL}
    Let $H = \GL_k$, where $k = 4$, $p=4$, $q=6$, and $\sigma = (4,3,2,-3)$.
    On the left-hand side below, we implement Step 1 of Definition~\ref{def:arc diagram}, starting with the diagram of a jellyfish $(\F, T) \in \mathcal{J}(\sigma)$.
    The numbers along the lattice paths are the exponents $d_{ij}$ in Definition~\ref{def:arc diagram}, which determine a standard monomial $\xi$ lying in the Stanley space indexed by $(\F, T)$.
    On the right-hand side, we implement Step 2 of Definition~\ref{def:arc diagram} to obtain the arc diagram of $\xi$, with the hyperedges highlighted below the vertices.
    Beneath the arc diagram, we use Proposition~\ref{prop:weight from degree sequence} to compute ${\rm wt}_\g(\xi)$:
    \tikzstyle{sq}=[rectangle,fill=black, minimum size = 10pt, inner sep=1pt,text=white, font=\tiny\sffamily\bfseries]
\tikzstyle{newcorner}=[rectangle,draw=black,thick, minimum size = 14pt, inner sep=0pt]
\tikzstyle{endpt}=[circle, fill=lightgray,minimum size = 7pt, inner sep=0pt]
\tikzstyle{highlight}=[rounded corners,line width=.8em,green!50!blue,nearly transparent,cap=round]
\tikzstyle{dot}=[circle,fill=black, minimum size = 5pt, inner sep=0pt]

\begin{center}
\begin{tikzpicture}[scale=.45,baseline=(current bounding box.center),every node/.style={scale=.7}]

\draw[lightgray,line width = 2pt] (4,4) -- ++(2,0) -- ++(0,-1) \foreach \x/\y in {1/0,1/-2,1/0}{
-- ++(\x,\y) node[endpt]{}
}
(3,3) -- ++(2,0) -- ++(0,-2) 
\foreach \x/\y in {0/-1,1/-1}{
-- ++(\x,\y) node[endpt]{}
}
(2,2) -- ++(1,0) -- ++(0,-1) -- ++(1,0)
\foreach \x/\y in {0/-1,0/-1,1/-1}{
-- ++(\x,\y) node[endpt]{}
}
(1,1) 
\foreach \x/\y in {0/-1,2/-1,0/-1,1/-1}{
-- ++(\x,\y) node[endpt]{}
};

\draw [white, fill=lightgray] 
(.5,4.5) -- ++ (3,0) -- ++ (0,-1) -- ++ (-1,0) -- ++ (0,-1) -- ++ (-1,0) -- ++ (0,-1) -- ++ (-1,0) -- cycle;

\foreach \x in {1,...,6}{    \foreach \y in {1,...,4}{        \node [dot] at (\x,\y) {};}}

\draw [densely dotted] (.5,4.5) rectangle (6.5,.5);

\node at (0,4) {$1^*$};
\node at (0,3) {$2^*$};
\node at (0,2) {$3^*$};
\node at (0,1) {$4^*$};
\node at (1,5) {$1$};
\node at (2,5) {$2$};
\node at (3,5) {$3$};
\node at (4,5) {$4$};
\node at (5,5) {$5$};
\node at (6,5) {$6$};

\node at (6,4) [newcorner] {};
\node at (3,2) [newcorner] {};
\node at (5,3) [newcorner] {};
\node at (6,4) [sq] {2};
\node at (2,3) [sq] {1};
\node at (1,1) [sq] {3};
\node at (5,2) [sq] {1};

\node[below=5pt of current bounding box.south,anchor=north, align=center,scale=1.3
    ]{$\xi = \underbrace{f_{16}^2 f_{22} f_{35} f_{41}^3}_{\mathbf{f} \in \C[f_{ij} : (i,j) \in \mathbf{F}]}\underbrace{f_{16}f_{25}f_{33}}_{f_{\cor(\F)}}\phi_{T}$ \\[2ex] $T = \left( \: \ytableaushort{1334,445,56}_{\textstyle ,} \ytableaushort{244,\none,\none} \:\right)$};

\end{tikzpicture}
\qquad $\leadsto$ \qquad
\begin{tikzpicture}[-,auto,
  thick,plain node/.style={minimum size=12pt,circle,draw,font=\tiny, fill = gray!30, inner sep=0pt}, 
  painted hypernode/.style={minimum size=4pt,inner sep=0pt,circle,draw,fill=black},bend left = 60, scale=.5,baseline=(current bounding box.center)]

\draw[highlight] (-2,-1) -- ++(0,0);
\draw[highlight] (0,-2) -- ++(0,0);
\draw[highlight] (0,-3) -- ++(0,0);
\draw[highlight] (2,-1) -- (6,-1);
\draw[highlight] (4,-2) -- (7,-2);
\draw[highlight] (4,-3) -- (6,-3);
\draw[highlight] (5,-4) -- (5,-4);

  \draw[ultra thick] (-2,0) \foreach \x/\y in {0/-1,2/-1,0/-1}{
-- ++(\x,\y) node[painted hypernode]{}
}
(2,0) \foreach \x/\y in {0/-1,2/-1,0/-1,1/-1}{
-- ++(\x,\y) node[painted hypernode]{}
}
(5,0) \foreach \x/\y in {0/-1,0/-1,1/-1}{
-- ++(\x,\y) node[painted hypernode]{}
}
(6,0) \foreach \x/\y in {0/-1,1/-1}{
-- ++(\x,\y) node[painted hypernode]{}
};

\node[plain node] (4*) at (0,0) {4*};
\node[plain node] (3*) at (-1,0) {3*};
\node[plain node] (2*) at (-2,0) {2*};
\node[plain node] (1*) at (-3,0) {1*};
\node (divider) at (1,0) {$\Bigg|$};
\node[plain node] (1) at (2,0) {1};
\node[plain node] (2) at (3,0) {2};
\node[plain node] (3) at (4,0) {3};
\node[plain node] (4) at (5,0) {4};
\node[plain node] (5) at (6,0) {5};
\node[plain node] (6) at (7,0) {6};

\draw [bend left=70] (1*) to (6);
\draw [bend left = 65] (1*) to (6);
\draw [bend left=60] (1*) to (6);
\draw [bend left=50] (2*) to (2);
\draw [bend left=65] (2*) to (5);
\draw [bend left=50] (3*) to (3);
\draw [bend left=60] (3*) to (5);
\draw [bend left=60] (4*) to (1);
\draw [bend left=50] (4*) to (1);
\draw [bend left=40] (4*) to (1);

\node[below=-10pt of current bounding box.south,anchor=north,align=center
    ]{${\rm wt}_\g(\xi) = $\\[1.5ex]
    $\phantom{-}(-3,-3,-2,-5 \mid 4,1,3,3,4,4)$\\\underline{$-(\phantom{-}4,\phantom{-}4,\phantom{-}4,\phantom{-}4, \mid 0,0,0,0,0,0)$} \\ $\phantom{-}(-7,-7,-6,-9 \mid 4,1,3,3,4,4)$};
    
\end{tikzpicture}
\end{center}
    
\end{exam}

\begin{exam}
\label{ex:arc diagram Sp}
We give another example of Definition~\ref{def:arc diagram} and Proposition~\ref{prop:weight from degree sequence}, this time where $H = \Sp_{2k}$, with $k = 3$ and $n = 8$, and $\sigma = (3,3,2)$:

\begin{center}
\tikzstyle{sq}=[rectangle,fill=black, minimum size = 10pt, inner sep=1pt,text=white, font=\tiny\sffamily\bfseries]
\tikzstyle{newcorner}=[rectangle,draw=black,thick, minimum size = 14pt, inner sep=0pt]
\tikzstyle{endpt}=[circle, fill=lightgray,minimum size = 7pt, inner sep=0pt]
\tikzstyle{highlight}=[rounded corners,line width=.8em,green!50!blue,nearly transparent,cap=round]
\tikzstyle{dot}=[circle,fill=black, minimum size = 5pt, inner sep=0pt]
\begin{tikzpicture}[scale=.45,baseline=(current bounding box.center),every node/.style={scale=.7}]
\draw [white,fill=lightgray] (1.5,8.5) -- ++ (0,-1) -- ++ (1,0) -- ++ (0,-1) -- ++ (1,0) -- ++ (0,-1) -- ++ (1,0) -- ++ (0,1)  -- ++ (1,0) -- ++ (0,1) -- ++ (1,0) -- ++ (0,1) -- cycle;
\draw [densely dotted] (1.5,8.5) -- ++(7,0) -- ++(0,-7) -- ++(-1,0) -- ++(0,1) -- ++(-1,0) -- ++(0,1) -- ++(-1,0) -- ++(0,1) -- ++(-1,0) -- ++(0,1) -- ++(-1,0) -- ++(0,1) -- ++(-1,0) -- ++(0,1) -- ++(-1,0) -- ++(0,1);
\draw[lightgray, line width = 2pt] (5,6) -- ++(0,-1) -- ++(1,0) -- ++(0,-1) -- ++(2,0) \foreach \x/\y in {1/0,1/-2}{
-- ++(\x,\y) node[endpt]{}
}
(6,7) -- ++(1,0) -- ++(0,-1) -- ++(1,0) node [newcorner] {} -- ++(0,-1) \foreach \x/\y in {1/0,1/0,1/-1}{
-- ++(\x,\y) node[endpt]{}
}
(7,8) -- ++(1,0) -- ++(0,-1) \foreach \x/\y in {1/0,1/-1,1/0}{
-- ++(\x,\y) node[endpt]{}
};

\node at (8,8) [newcorner] {};
\node at (7,7) [newcorner] {};
\node at (8,6) [newcorner] {};
\foreach \x in {2,...,8}{\foreach \y in {\x,...,8}{\node [dot] at (10-\x,\y) {};}}
\node at (0,8) {$1$};
\node at (0,7) {$2$};
\node at (0,6) {$3$};
\node at (0,5) {$4$};
\node at (0,4) {$5$};
\node at (0,3) {$6$};
\node at (0,2) {$7$};
\node at (0,1) {$8$};
\node at (1,9) {$1$};
\node at (2,9) {$2$};
\node at (3,9) {$3$};
\node at (4,9) {$4$};
\node at (5,9) {$5$};
\node at (6,9) {$6$};
\node at (7,9) {$7$};
\node at (8,9) {$8$};

\node at (2,8) [sq] {1};
\node at (4,8) [sq] {2};
\node at (7,7) [sq] {2};
\node at (4,6) [sq] {1};
\node at (5,5) [sq] {1};
\node at (8,4) [sq] {2};

\end{tikzpicture}
\qquad $\leadsto$ \qquad
\begin{tikzpicture}[-,auto,
  thick,plain node/.style={minimum size=12pt,circle,draw,font=\scriptsize, fill = gray!30, inner sep=0pt}, 
  painted hypernode/.style={minimum size=4pt,inner sep=0pt,circle,draw,fill=black},bend right = 70, scale=.55,baseline=(current bounding box.center)]

\draw[highlight] (-3,-.75) -- ++(-3,0);
\draw[highlight] (-1,-1.5) -- ++(-4,0);
\draw[highlight] (-3,-2.25) -- ++(-2,0);

\draw[ultra thick] (-6,0) \foreach \x/\y in {0/-.75,1/-.75,0/-.75}{
-- ++(\x,\y) node[painted hypernode]{}
};
\draw[ultra thick] (-4,0) \foreach \x/\y in {0/-.75,0/-.75,1/-.75}{
-- ++(\x,\y) node[painted hypernode]{}
};
\draw[ultra thick] (-3,0) \foreach \x/\y in {0/-.75,2/-.75}{
-- ++(\x,\y) node[painted hypernode]{}
};

\node[plain node] (8) at (0,0) {8};
\node[plain node] (7) at (-1,0) {7};
\node[plain node] (6) at (-2,0) {6};
\node[plain node] (5) at (-3,0) {5};
\node[plain node] (4) at (-4,0) {4};
\node[plain node] (3) at (-5,0) {3};
\node[plain node] (2) at (-6,0) {2};
\node[plain node] (1) at (-7,0) {1};

\draw [bend right = 80] (8) to (1);
\draw [bend right = 50] (2) to (1);
\draw [bend right = 60] (4) to (1);
\draw [bend right = 70] (4) to (1);
\draw [bend right = 60] (7) to (2);
\draw [bend right = 70] (7) to (2);
\draw [bend right = 80] (7) to (2);
\draw [bend right = 75] (8) to (3);
\draw [bend right = 50] (4) to (3);
\draw [bend right = 35] (5) to (4);
\draw [bend right = 70] (8) to (5);
\draw [bend right = 60] (8) to (5);

\node[below=5pt of current bounding box.south,anchor=north,align=center
    ]{${\rm wt}_\g(\xi) = $\\[1.5ex]
    $-(4,5,4,6,5,0,4,4)$\\\underline{$-(3,3,3,3,3,3,3,3)$} \\ $-(7,8,7,9,8,3,7,7)$};
\end{tikzpicture}
\end{center}

\end{exam}

\begin{proof}[Proof of Proposition~\ref{prop:weight from degree sequence}]

The proof for all three groups $H$ is identical, up to the details of the $\g$-action on $\C[W]$ given in Appendix~\ref{app:Howe}.
By~\eqref{Psi on f det}, when computing the weight of $\xi$ we may as well replace $\phi_T$ by $\det_T$, as defined in~\eqref{table dets GL Sp}.
We write ${\rm wt}_\g$ to denote the weight with the respect to the action of $\g$, upon restriction to its Cartan subalgebra consisting of diagonal matrices.
In writing out weight formulas below, we use the degree notation defined in~\eqref{deg notation}, and we write $\epsilon_i$ to denote the tuple with 1 in the $i$th coordinate and 0's elsewhere.

Let $H = \GL_k$.
The action of $\g = \gl_{p+q}$ on $\C[W]$ is given explicitly in Appendix~\ref{sub:appendix GL}, from which we observe that the weight of a monomial
\[
\mathbf{m} \in \C[W] = \C\Big[\{y_{ij} : (i,j) \in [p] \times [k]\} \cup \{x_{ij} : (i,j) \in [k] \times [q]\}\Big]
\]
is given by the $(p+q)$-tuple
\begin{equation}
\label{g weight GL}
{\rm wt}_\g(\mathbf{m}) = \Big( - \deg_{y_{1 \bullet}}(\mathbf{m}), \ldots, - \deg_{y_{p \bullet}}(\mathbf{m}) \: \Big| \: \deg_{x_{\bullet 1}}(\mathbf{m}), \ldots, \deg_{x_{\bullet q}}(\mathbf{m})  \Big) - (k^p \mid 0 ).
\end{equation}
Each arc $\{i^*,j\}$ in the arc diagram of $\xi$ contributes $(\epsilon_i \mid \epsilon_{j})$ to the degree sequence;
it also represents a factor $f_{ij}$ in $\mathbf{f} \cdot f_{\cor(\F)}$, which contributes $(-\epsilon_i \mid  \epsilon_j)$ to ${\rm wt}_\g(\xi)$, by~\eqref{table coordinates} and~\eqref{g weight GL}.
Therefore, upon negating the degrees of the starred vertices $i^*$, the arcs combine to contribute ${\rm wt}_\g(\mathbf{f} \cdot f_{\cor(\F)})$ to the degree sequence of the arc diagram.
Within a hyperedge, a dot below vertex $i^*$ (\resp vertex $j$) contributes $(\epsilon_i \mid 0 )$ (\resp $(0 \mid \epsilon_j)$) to the degree sequence of the arc diagram;
it also corresponds to an entry $i$ in $T^-$ (\resp an entry $j$ in $T^+$), which by~\eqref{table dets GL Sp} and~\eqref{g weight GL} contributes $(-\epsilon_i \mid 0)$ (\resp $(0 \mid \epsilon_{j})$) to ${\rm wt}_\g(\det_T)$.
 Therefore, upon negating the degrees of the starred vertices $i^*$, the hyperedges combine to contribute precisely ${\rm wt}_\g(\det_T)$ to the degree sequence.
 Thus, by combining the arcs and the hyperedges, we conclude that ${\rm wt}_\g(\xi)$ is obtained from the degree sequence of the arc diagram of $\xi$ by negating and then subtracting $k$ from the degrees of the starred vertices.

Let $H = \Sp_{2k}$.
The action of $\g = \so_{2n}$ on $\C[W]$ is given explicitly in Appendix~\ref{sub:appendix Sp}, from which we observe that the weight of a monomial
\[
    \mathbf{m} \in \C[W] = \C\Big[x_{ij} : (i,j) \in [2k] \times [n]\Big]
\]
is given by the $n$-tuple
\begin{equation}
    \label{g weight Sp}
    {\rm wt}_\g(\mathbf{m}) = \Big( - \deg_{x_{\bullet 1}}(\mathbf{m}), \ldots, - \deg_{x_{\bullet n}}(\mathbf{m}) \Big) - (k^n).
\end{equation}
Each arc $\{i,j\}$ in the arc diagram of $\xi$ contributes $\epsilon_i + \epsilon_j$ to the degree sequence;
it also represents a factor $f_{ij}$ in $\mathbf{f} \cdot f_{\cor(\F)}$, which contributes $-(\epsilon_i + \epsilon_j)$ to ${\rm wt}_\g(\xi)$, by~\eqref{table coordinates} and~\eqref{g weight Sp}.
Therefore the arcs together contribute $-{\rm wt}_\g(\mathbf{f} \cdot f_{\cor(\F)})$ to the degree sequence of the arc diagram.
Within a hyperedge, a dot below vertex $i$ contributes $\epsilon_i$ to the degree sequence;
 it also corresponds to an entry $i$ in $T$, which by~\eqref{table dets GL Sp} and~\eqref{g weight Sp} contributes $-\epsilon_i$ to ${\rm wt}_\g(\det_T)$.
 Therefore the hyperedges together contribute precisely $-{\rm wt}_\g(\det_T)$ to the degree sequence.
 Thus, by combining the arcs and the hyperedges, we see from~\eqref{g weight Sp} that ${\rm wt}_\g(\xi)$ is obtained from the negative of the degree sequence of the arc diagram of $\xi$, upon subtracting $k$ from every coordinate.

Let $H = \O_k$.
The action of $\g = \sp_{2n}$ on $\C[W]$ is given explicitly in Appendix~\ref{sub:appendix O}, from which we observe that the weight of a monomial 
\[
\mathbf{m} \in \C[W] = \C\Big[x_{ij} : (i,j) \in [k] \times [n]\Big]
\]
is given by the $n$-tuple
\begin{equation}
\label{g weight O}
    {\rm wt}_\g(\mathbf{m}) = \Big( - \deg_{x_{\bullet 1}}(\mathbf{m}), \ldots, - \deg_{x_{\bullet n}}(\mathbf{m}) \Big) - ((\tfrac{k}{2})^n).
\end{equation}
  The rest of the proof is identical to the $\Sp_{2k}$ case, upon replacing $\det_T$ by $\phi_T$.
\end{proof}

\section{Wallach representations of type ADE}
\label{sec:ADE}

In this final section, we apply our jellyfish approach to certain unitary highest weight modules beyond those modules $L_{\la(\sigma)}$ arising in the Howe duality setting.
In particular, we consider certain distinguished modules known as  \emph{Wallach representations}, in all cases where $\g$ is simply laced (\ie of type $\ssA$, $\ssD$, or $\ssE$), and we reinterpret their Hilbert series in terms of lattice paths.

\subsection*{Hermitian symmetric pairs}

Suppose $G_\R$ is a connected real reductive group 
and $K_\R\subset G_\R$ is a maximal compact subgroup such that $G_\R/K_\R$ is an irreducible Hermitian symmetric space of noncompact type. 
Let $\g_{\R} = \k_{\R} \oplus \p_{\R}$ be a 
Cartan decomposition of the Lie algebra of $G_\R$, and let 
$\g = \k \oplus \p$ be the corresponding decomposition
of the complexified Lie algebra. 
 From the general theory, there exists a distinguished element $h_0 \in \mathfrak{z}(\k)$ such that $\operatorname{ad} h_0$ acts on $\g$ with eigenvalues $0$ and $\pm 1$.  
 We thus have a triangular decomposition 
 \[
    \g = \p^- \oplus \k \oplus \p^+,
 \]
 where $\p^{\pm} = \{ x \in \g : [h_0, x] = \pm x\}$.
 The subalgebra $\q = \k \oplus \p^+$ is a maximal parabolic subalgebra of $\g$, with Levi subalgebra $\k$ and abelian nilradical $\p^+$.
Upon fixing a Cartan subalgebra $\h$ of both $\g$ and $\k$, we write $\Phi$ for the root system of $(\g,\h)$, and $\g_\alpha$ for the root space corresponding to the root $\alpha \in \Phi$.
Then we put
\[
\Phi(\p^+) \coloneqq \{\alpha \in \Phi : \g_\alpha \subseteq \p^+\}.
\]

\label{r} Let $r$ denote the real rank of $G_{\R}$.
Recall that two roots $\alpha, \beta$ are called \emph{strongly orthogonal} if neither $\alpha + \beta$ nor $\alpha - \beta$ is a root.
We adopt Harish-Chandra's maximal set $\{\gamma_1, \ldots, \gamma_r\}$ of strongly orthogonal roots in $\Phi(\p^+)$, which are defined recursively as follows.
Let $\gamma_1$ be the lowest root in $\Phi(\p^+)$;
then for $1 < i \leq r$, let $\gamma_i$ be the lowest root in $\Phi(\p^+)$ that is strongly orthogonal to each of $\gamma_1, \ldots, \gamma_{i-1}$.  
For any nonnegative integers $n_1 \geq \cdots \geq n_r \geq 0$, the weight $-\sum_i n_i \gamma_i$ is $\k$-dominant and integral.  
We have the following multiplicity-free decomposition due to Schmid~\cite{Schmid}:
\[
    \C[\p^+] \cong \bigoplus_{\mathclap{n_1 \geq \cdots \geq n_r \geq 0}} F_{-n_1 \gamma_1 -\cdots - n_r\gamma_r},
\]
where $F_\mu$ denotes the simple $\k$-module with highest weight $\mu$.

As a specific example, we revisit the Howe duality setting of Section~\ref{sec:proofs}, where we have the following concrete realizations (with $\SM_n$ and $\AM_n$ denoting the spaces of $n \times n$ symmetric and alternating matrices, respectively):

\begin{center}
\label{new Howe table}
%\resizebox{\linewidth}{!}{
\begin{tblr}{colspec={|Q[m,c]|Q[m,c]|Q[m,c]|Q[m,c]|Q[m,c]|Q[m,c]|Q[m,c]|Q[m,c]|},stretch=1.5}

\hline

$H$ & $\g$ & $\g = \left\{\left[\begin{smallmatrix}A&B\\C&D\end{smallmatrix}\right] : \ldots \right\}$ & $\k = \left\{\left[\begin{smallmatrix}A&0\\0&D\end{smallmatrix}\right]\right\}$ & $\p^+ =  \left\{\left[\begin{smallmatrix}
    0 & B \\ 0 & 0
\end{smallmatrix}\right]\right\}$ & $r$ & $\gamma_i$ \\ \hline[2pt]

$\GL_k$ & $\gl_{p+q}$ & no constraints & $\gl_p \oplus \gl_q$ & $\M_{p,q}$ & $\min\{p,q\}$ & $( \overleftarrow{\epsilon_i} \mid -\overrightarrow{\epsilon_i})$ \\ \hline

$\O_k$ & $\sp_{2n}$ & {$A=-D^t$, \\ $B=B^t$, \\ $C=C^t$} & $\gl_n$ & $\SM_n$ & $n$ & $\overleftarrow{2\epsilon_i}$ \\ \hline

$\Sp_{2k}$ & $\so_{2n}$ & {$A=-D^t$, \\ $B=-B^t$, \\ $C=-C^t$} & $\gl_n$ & $\AM_n$ & $\lfloor n/2 \rfloor$ & $\overleftarrow{\epsilon_{2i-1} + \epsilon_{2i}}$ \\ \hline

\end{tblr}
%}
\end{center}

\noindent In this setting, one can define an $H$-invariant polynomial map $\pi: W \longrightarrow \p^+$ as follows:
\begin{alignat*}{3}
    &(H = \GL_k) \qquad && \pi: \M_{p,k} \oplus \M_{k,q} \longrightarrow \M_{p,q}, \qquad && (Y,X) \longmapsto YX,\\
    &(H = \O_k) \qquad && \pi: \M_{k,n} \longrightarrow \SM_n, \qquad && X \longmapsto X^t X,\\
    &(H = \Sp_{2k}) \qquad && \pi: \M_{2k,n} \longrightarrow \AM_n, \qquad && X \longmapsto X^t \left( \begin{smallmatrix}
    0 & I \\ -I & 0
\end{smallmatrix}\right) X.
    \end{alignat*}
 Then $\pi(W) \subseteq \p^+$ is either all of $\p^+$ (if $k \geq r$), or else is the determinantal variety $\M^{\leqslant k}_{p,q}$ or $\SM^{\leqslant k}_n$ or $\AM^{\leqslant 2k}_n$, respectively; here $\M^{\leqslant k}_{p,q}$ denotes the set of matrices with rank at most $k$, and similarly in the other two cases.
 Thus for $k<r$, the map $\pi$ descends to an isomorphism $W /\!\!/ H \cong \M^{\leqslant k}_{p,q}$ (\resp $\SM^{\leqslant k}_n$ or $\AM^{\leqslant 2k}_n$).
Moreover, $\pi$ induces a surjective comorphism
\begin{align}
    \label{pi star again}
    \begin{split}
    \pi^* : \C[\mathfrak{p}^+] & \longrightarrow \C[W]^H, \\
    z_{ij} & \longmapsto f_{ij},
    \end{split}
\end{align}
where the $z_{ij}$'s are the standard matrix coordinates on $\p^+$, and the $f_{ij}$'s are the  contractions in~\eqref{table P fij}.
Note that $\pi^*$ is the same map introduced in~\eqref{pi star map}, since we have $S = \C[\mathbf{P}] = \C[\p^+]$.
In all three Howe duality settings, we have $\mathbf{P} \cong \Phi(\p^+)$ as sets;
moreover, in the two cases where $\g$ is simply laced (\ie where $\g = \gl_{p+q}$ or $\so_{2n}$), we have a \emph{poset} isomorphism $\mathbf{P} \cong \Phi(\p^+)$,
where the partial order on $\mathbf{P}$ is given by~\eqref{partial order on P}, and the partial order on $\Phi(\p^+)$ is inherited from the standard partial order on $\Phi$.
Explicitly, the poset isomorphism $\mathbf{P} \longrightarrow \Phi(\p^+)$ is given by
\[
    \begin{cases}
        (i,j) \longmapsto 
(\overleftarrow{\epsilon_i} \mid - \overrightarrow{\epsilon_j}), & (H, \g) = (\GL_k, \gl_{p+q}),\\[2ex]
(i,j) \longmapsto \overleftarrow{\epsilon_i + \epsilon_j}, & (H, \g) = (\Sp_{2k}, \so_{2n}).
    \end{cases}
\]

\subsection*{The Wallach representations}

Let $K$ be the complexification of $K_\R$.
Then $\k$ is the Lie algebra of $K$, and the adjoint action of $\k$ on
$\g$ exponentiates to a $K$-action.
In particular, $K$ acts on $\p^+$.
The $K$-orbits in $\p^+$ are  
$\scr{O}_0:=\{0\}$ and
$\scrO_k:=K\cdot(e_{\gamma_1}+\dots +e_{\gamma_k})$ for $1\leq k\leq r$, where $e_{\gamma_i}\in \g_{\gamma_i}$ denotes a root vector corresponding to $\gamma_i$.
The closures of the $K$-orbits in $\p^+$ form a chain of algebraic varieties
\[
\{0\} = \overline{\scrO}_0 \subset \overline{\scrO}_1 \subset \cdots \subset \overline{\scrO}_r = \p^+.
\]
For each $0 \leq k \leq r$, the coordinate ring of $\overline{\scrO}_k$ can be viewed as a truncation of the Schmid decomposition above~\cite{EHW}:
\[
    \C[\overline{\scrO}_k] \cong \bigoplus_{k\geq n_1 \geq \cdots \geq n_r \geq 0} F_{-n_1 \gamma_1 -\cdots - n_r\gamma_r.}
\]

Let $c \coloneqq \frac{1}{2}(\rho, \: \gamma_2^\vee - \gamma_1^\vee)$, where as usual $\rho$ is half the sum of the positive roots of $\g$, and $( \:, \:)$ is the nondegenerate bilinear form on $\h^*$ induced from the Killing form of $\g$.
Let $\zeta$ be the unique fundamental
weight of $\g$ that is orthogonal to the roots of $\k$.
Then for each $1 \leq k < r$, the \emph{$k$th Wallach representation} is the simple $\g$-module $L_{-kc\zeta}$.
The determinantal variety $\overline{\scrO}_k$ is the associated variety of the $k$th Wallach representation.

In the Howe duality setting of Section~\ref{sec:proofs}, one can check (using the $\la(\sigma)$ column of the table in Theorem~\ref{thm:Howe duality}) that the $k$th Wallach representation is given by
\[
L_{-kc\zeta} = L_{\lambda(0)} \cong \widetilde{M_0},
\]
where $M_0 = \C[W]^H$ is the ring of invariants.
Therefore, when $(\g,\k) = (\gl_{p+q}, \gl_p \oplus \gl_q)$ or $(\sp_{2n}, \gl_n)$ or $(\so_{2n}, \gl_n)$, the Stanley decompositions and Hilbert series for the Wallach representations are obtained by specializing to $\sigma = 0$ in Theorem~\ref{thm:Stanley decomps and HS} and Corollary~\ref{cor:HS}.
(One should replace the $f_{ij}$'s by their preimages $z_{ij}$ in $S = \C[\p^+]$, and thus also replace each $t^2$ by $t$ in the Hilbert series.)
Note that in this $\sigma = 0$ case, we have $\mathcal{J}(0) = \mathcal{F}(\Delta_k(\mathbf{P})) \times \{ \varnothing \}$, where the \emph{$k$th order complex} is defined to be
\[
    \Delta_k(\mathbf{P}) \coloneqq \{ \mathbf{S} \subseteq \mathbf{P} : {\rm width}(\mathbf{S}) \leq k \}.
\]
As a result, in the special case $\sigma = 0$, for the two simply laced cases (where $\mathbf{P} \cong \Phi(\p^+)$), Theorem~\ref{thm:Stanley decomps and HS} can be rewritten as an isomorphism of graded vector spaces:
\begin{equation}
    \label{Wallach rep iso}
    L_{-kc\zeta} \cong \bigoplus_{\F \in \mathcal{F}(\Delta_k(\Phi(\p^+)))} \C[\F] z_{{\rm res}(\F)}.
\end{equation}
Indeed, it is a general fact for Hermitian symmetric pairs~\cite{Jakobsen}*{Lemma~4.1} that the Hasse diagram of $\Phi(\p^+)$ is a planar distributive lattice, and therefore its order complexes are shellable~\cite{Bjorner80}*{Thm.~7.1}.
This suggests that the Stanley decomposition~\eqref{Wallach rep iso} might be valid for other Hermitian symmetric pairs $(\g,\k)$, which would yield the Hilbert series
\begin{equation}
    \label{Hilbert series Wallach}
    P(L_{-kc\zeta};t) = \frac{\sum_{\F \in \mathcal{F}(\Delta_k(\Phi(\p^+)))} t^{\#{\rm res}(\F)}}{(1-t)^d},
\end{equation}
where $d$ is the common size of all facets $\F$.

We have found that~\eqref{Hilbert series Wallach} does indeed hold true whenever $\g$ is simply laced (\ie when $\g$ is of Killing--Cartan type $\ssA$, $\ssD$, or $\ssE$).
Below, we verify~\eqref{Hilbert series Wallach} for all Hermitian symmetric pairs of simply laced type.
The Hilbert series for the Wallach representations were previously computed in~\cite{EW} and~\cite{EnrightHunziker04}*{\S6.6--6.8} and~\cite{EnrightHunzikerExceptional}, by using Enright--Shelton reduction to compute generalized BGG resolutions.
Our formula~\eqref{Hilbert series Wallach}, on the other hand, provides a new combinatorial interpretation of these Hilbert series in terms of nonintersecting lattice paths in $\Phi(\p^+)$.
As mentioned above, types $(\ssA_{p+q-1}, \ssA_{p-1} \times \ssA_{q-1})$ and $(\ssD_n, \ssA_{n-1})$ occur in the Howe duality setting and therefore are included already in Corollary~\ref{cor:HS}, for $H = \GL_k$ and $\Sp_{2k}$, respectively.
This leaves three simply laced cases: $(\ssD_n, \ssD_{n-1})$ and $(\ssE_6, \ssD_5)$, for $k<r=2$, and $(\ssE_7, \ssE_6)$, for $k<r=3$.
In each case, we rotate the Hasse diagram of $\Phi(\p^+)$ so that the minimal element $\gamma_1$ is in the upper-left.
In this way, saturated chains are southeast lattice paths.
We continue to suppress the edges in the Hasse diagrams, in order to make the lattice paths clearly visible; there is no loss of information, since each point is covered by the point(s) immediately to its south and its east.
We also continue to indicate the elements of ${\rm res}(\F)$ by drawing squares around them; these are easily determined by imposing a binary EL-labeling on $\Phi(\p^+)$ as we did in Section~\ref{sec:proofs}, and then using our characterization of corners in~\eqref{corner Sp}.

\subsection*{First Wallach representation of $(\ssD_n, \ssD_{n-1})$}
\label{sub:Wallach Dn}

For $(\g,\k) = (\ssD_n, \ssD_{n-1})$, 
the Hasse diagram of $\Phi(\p^+)$ has $2(n-1)$ many vertices, arranged as two chains of length $n-2$ joined by a square.
Since $r=2$, there is only one Wallach representation, corresponding to $k=1$, where $-kc\zeta = -(n+2)\epsilon_1$.
The facets of $\Delta_1(\Phi(\p^+))$ are precisely the maximal southeast lattice paths in $\Phi(\p^+)$.
There are only two facets $\F_0$ and $\F_1$ (shown below in the case $n=6$):

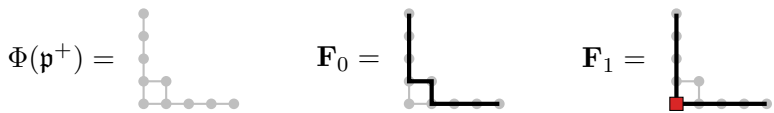
\tikzstyle{dot}=[scale=.7,circle,fill=black, minimum size = 4.5pt, inner sep=0pt]
    \tikzstyle{corner}=[rectangle,draw=black,thin, minimum size = 6pt, inner sep=2pt]

\begin{center}
\begin{tikzpicture}[scale=.3, baseline]

\draw [densely dotted] (.5,.5) -- ++(0,5) --++(1,0) -- ++(0,-3) -- ++(1,0) -- ++(0,-1) -- ++(3,0) -- ++(0,-1) -- cycle;

\node at (1,1) [dot] {};
\node at (1,2) [dot] {};
\node at (1,3) [dot] {};
\node at (1,4) [dot] {};
\node at (1,5) [dot] {};
\node at (2,1) [dot] {};
\node at (3,1) [dot] {};
\node at (4,1) [dot] {};
\node at (5,1) [dot] {};
\node at (2,2) [dot] {};

\node[left=5pt of current bounding box.west,anchor=east
    ]{$\Phi(\p^+) =$};
\end{tikzpicture}
\qquad
\begin{tikzpicture}[scale=.3, baseline]

\draw [densely dotted] (.5,.5) -- ++(0,5) --++(1,0) -- ++(0,-3) -- ++(1,0) -- ++(0,-1) -- ++(3,0) -- ++(0,-1) -- cycle;

\draw[lightgray, line width=3pt] (1,5) -- (1,1) -- (5,1);

\node at (1,1) [dot] {};
\node at (1,2) [dot] {};
\node at (1,3) [dot] {};
\node at (1,4) [dot] {};
\node at (1,5) [dot] {};
\node at (2,1) [dot] {};
\node at (3,1) [dot] {};
\node at (4,1) [dot] {};
\node at (5,1) [dot] {};
\node at (2,2) [dot] {};

\node[left=5pt of current bounding box.west,anchor=east
    ]{$\F_0 =$};
\end{tikzpicture}
\quad
\begin{tikzpicture}[scale=.3, baseline]

\draw [densely dotted] (.5,.5) -- ++(0,5) --++(1,0) -- ++(0,-3) -- ++(1,0) -- ++(0,-1) -- ++(3,0) -- ++(0,-1) -- cycle;

\draw[lightgray, line width=3pt] (1,5) -- (1,2) -- (2,2) node [corner] {} -- (2,1) -- (5,1);

\node at (1,1) [dot] {};
\node at (1,2) [dot] {};
\node at (1,3) [dot] {};
\node at (1,4) [dot] {};
\node at (1,5) [dot] {};
\node at (2,1) [dot] {};
\node at (3,1) [dot] {};
\node at (4,1) [dot] {};
\node at (5,1) [dot] {};
\node at (2,2) [dot] {};

\node[left=5pt of current bounding box.west,anchor=east
    ]{$\F_1 =$};
\end{tikzpicture}
\end{center}

\noindent Note that each facet has cardinality $d = 2n-3$.
Following~\eqref{Hilbert series Wallach} yields the Hilbert series below, which coincides with~\cite{EnrightHunziker04}*{Thm.~26}:
\[
P(L_{-kc\zeta};t) = \frac{1+t}{(1-t)^{2n-3}}.
\]

\subsection*{First Wallach representation of $(\ssE_6, \ssD_5)$}
\label{sub:Wallach E6}

For $(\g,\k) = (\ssE_6, \ssD_5)$, since $r=2$, there is only one Wallach representation, corresponding to $k=1$, where $-kc\zeta = -3\zeta$.
The facets of $\Delta_1(\Phi(\p^+))$ are again just maximal southeast lattice paths in $\Phi(\p^+)$.

\tikzstyle{dot}=[scale=.7,circle,fill=black, minimum size = 4.5pt, inner sep=0pt]
\tikzstyle{corner}=[rectangle,draw=black,thin, minimum size = 6pt, inner sep=2pt]

\begin{center}

\begin{tikzpicture}[scale=.3, baseline]

\draw [densely dotted] (.5,1.5) -- ++(5,0) -- ++(0,-2) -- ++(1,0) -- ++(0,-1) -- ++(2,0) -- ++ (0,-1) -- ++(-5,0) -- ++(0,2) -- ++(-1,0) -- ++(0,1) -- ++ (-2,0) -- cycle;

\node at (1,1) [dot] {};
\node at (2,1) [dot] {};
\node at (3,1) [dot] {};
\node at (4,1) [dot] {};
\node at (5,1) [dot] {};
\node at (3,0) [dot] {};
\node at (4,0) [dot] {};
\node at (5,0) [dot] {};
\node at (4,-1) [dot] {};
\node at (5,-1) [dot] {};
\node at (6,-1) [dot] {};
\node at (4,-2) [dot] {};
\node at (5,-2) [dot] {};
\node at (6,-2) [dot] {};
\node at (7,-2) [dot] {};
\node at (8,-2) [dot] {};

\node[left=0pt of current bounding box.west,anchor=east
    ]{$\Phi(\p^+) =$};
\end{tikzpicture}
\qquad
\begin{tikzpicture}[scale=.3, baseline]

\draw[lightgray, line width=3pt] (1,1) -- ++(2,0) -- ++(0,-1) -- ++(2,0) node [corner] {} -- ++(0,-1) -- ++(1,0) node [corner] {} -- ++(0,-1) -- ++(2,0);

\draw [densely dotted] (.5,1.5) -- ++(5,0) -- ++(0,-2) -- ++(1,0) -- ++(0,-1) -- ++(2,0) -- ++ (0,-1) -- ++(-5,0) -- ++(0,2) -- ++(-1,0) -- ++(0,1) -- ++ (-2,0) -- cycle;

\node at (1,1) [dot] {};
\node at (2,1) [dot] {};
\node at (3,1) [dot] {};
\node at (4,1) [dot] {};
\node at (5,1) [dot] {};
\node at (3,0) [dot] {};
\node at (4,0) [dot] {};
\node at (5,0) [dot] {};
\node at (4,-1) [dot] {};
\node at (5,-1) [dot] {};
\node at (6,-1) [dot] {};
\node at (4,-2) [dot] {};
\node at (5,-2) [dot] {};
\node at (6,-2) [dot] {};
\node at (7,-2) [dot] {};
\node at (8,-2) [dot] {};

\node[left=0pt of current bounding box.west,anchor=east
    ]{Example of a facet:};
\end{tikzpicture}

\end{center}

\noindent All twelve facets have cardinality $d = 11$.
Following~\eqref{Hilbert series Wallach} yields the Hilbert series below, which agrees with~\cite{EnrightHunziker04}*{Thm.~28}:
\[
P(L_{-kc\zeta};t) = \frac{1+5t+5t^2+t^3}{(1-t)^{11}}.
\]

\subsection*{First Wallach representation of $(\ssE_7, \ssE_6)$}
\label{sub:Wallach E7 k1}

For $(\g, \k) = (\ssE_7, \ssE_6)$, since $r=3$, there are two Wallach representations.
For $k=1$, we have $-kc\zeta = -4\zeta$, and the facets of $\Delta_1(\Phi(\p^+))$ are again the maximal southeast lattice paths in $\Phi(\p^+)$:

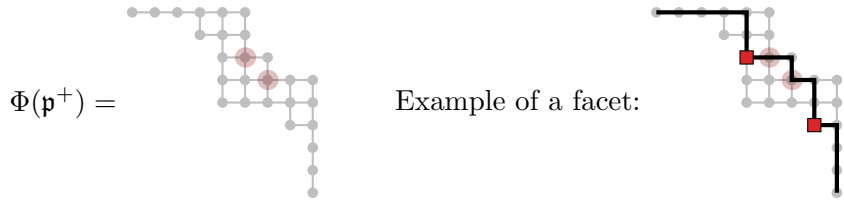
\tikzstyle{dot}=[scale=.7,circle,fill=black, minimum size = 4.5pt, inner sep=0pt]
\tikzstyle{corner}=[rectangle,draw=black,thin, minimum size = 6pt, inner sep=2pt]

\begin{center}

\begin{tikzpicture}[scale=.3, baseline]

\draw [densely dotted] (-.5,1.5) -- ++(6,0) -- ++(0,-2) -- ++(1,0) -- ++(0,-1) -- ++(2,0) -- ++(0,-6) -- ++(-1,0) -- ++(0,3) -- ++(-1,0) -- ++(0,1) -- ++(-3,0) -- ++(0,3) -- ++(-1,0) -- ++(0,1) -- ++(-3,0) -- cycle;

\node at (0,1) [dot] {};
\node at (1,1) [dot] {};
\node at (2,1) [dot] {};
\node at (3,1) [dot] {};
\node at (4,1) [dot] {};
\node at (5,1) [dot] {};
\node at (3,0) [dot] {};
\node at (4,0) [dot] {};
\node at (5,0) [dot] {};
\node at (4,-1) [dot] {};
\node at (5,-1) [dot] {};
\node at (6,-1) [dot] {};
\node at (4,-2) [dot] {};
\node at (5,-2) [dot] {};
\node at (6,-2) [dot] {};
\node at (6,-3) [dot] {};
\node at (4,-3) [dot] {};
\node at (5,-3) [dot] {};
\node at (7,-2) [dot] {};
\node at (8,-2) [dot] {};
\node at (7,-3) [dot] {};
\node at (8,-3) [dot] {};
\node at (7,-4) [dot] {};
\node at (8,-4) [dot] {};
\node at (8,-5) [dot] {};
\node at (8,-6) [dot] {};
\node at (8,-7) [dot] {};

\node[left=0pt of current bounding box.west,anchor=east
    ]{$\Phi(\p^+) =$};
\end{tikzpicture}
\qquad
\begin{tikzpicture}[scale=.3, baseline]

\draw[lightgray, line width=3pt] (0,1) -- ++(4,0) node [corner] {} -- ++(0,-2) -- ++(2,0) node [corner] {} -- ++(0,-1) -- ++(1,0) node [corner] {} -- ++(0,-2) -- ++(1,0) -- ++(0,-3);

\draw [densely dotted] (-.5,1.5) -- ++(6,0) -- ++(0,-2) -- ++(1,0) -- ++(0,-1) -- ++(2,0) -- ++(0,-6) -- ++(-1,0) -- ++(0,3) -- ++(-1,0) -- ++(0,1) -- ++(-3,0) -- ++(0,3) -- ++(-1,0) -- ++(0,1) -- ++(-3,0) -- cycle;

\node at (0,1) [dot] {};
\node at (1,1) [dot] {};
\node at (2,1) [dot] {};
\node at (3,1) [dot] {};
\node at (4,1) [dot] {};
\node at (5,1) [dot] {};
\node at (3,0) [dot] {};
\node at (4,0) [dot] {};
\node at (5,0) [dot] {};
\node at (4,-1) [dot] {};
\node at (5,-1) [dot] {};
\node at (6,-1) [dot] {};
\node at (4,-2) [dot] {};
\node at (5,-2) [dot] {};
\node at (6,-2) [dot] {};
\node at (6,-3) [dot] {};
\node at (4,-3) [dot] {};
\node at (5,-3) [dot] {};
\node at (7,-2) [dot] {};
\node at (8,-2) [dot] {};
\node at (7,-3) [dot] {};
\node at (8,-3) [dot] {};
\node at (7,-4) [dot] {};
\node at (8,-4) [dot] {};
\node at (8,-5) [dot] {};
\node at (8,-6) [dot] {};
\node at (8,-7) [dot] {};

\node[left=0pt of current bounding box.west,anchor=east
    ]{Example of a facet:};
\end{tikzpicture}

\end{center}

\noindent This gives us $78$ facets, all having cardinality $d = 17$.
Counting the sizes of the restrictions, and following~\eqref{Hilbert series Wallach}, we obtain the Hilbert series below, which agrees with~\cite{EnrightHunziker04}*{Thm.~29}: 
\[
P(L_{-kc\zeta}; t) = \frac{ 1+10t+28t^2 +28t^3 +10t^4 +t^5}{(1-t)^{17}}.
\]

\subsection*{Second Wallach representation of $(\ssE_7, \ssE_6)$}
\label{sub:Wallach E7 k2}

The poset $\Phi(\p^+)$ is the same as above; this time, for $k=2$, we have $-kc\zeta = -8\zeta$.
Each facet of $\Delta_2(\Phi(\p^+))$ is a maximal union of two nonintersecting southeast lattice paths in $\Phi(\p^+)$.
There are only three facets, each with cardinality $d=26$:

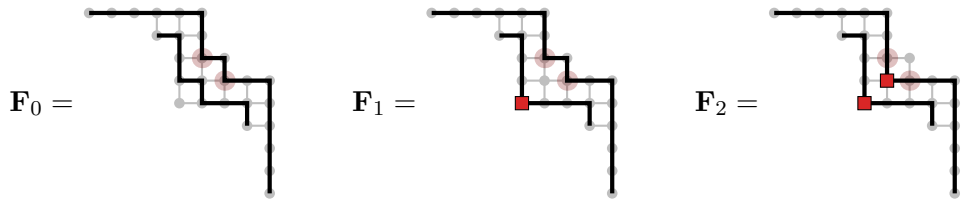
\tikzstyle{dot}=[scale=.7,circle,fill=black, minimum size = 4.5pt, inner sep=0pt]
\tikzstyle{corner}=[rectangle,draw=black,thin, minimum size = 6pt, inner sep=2pt]

\begin{center}

\begin{tikzpicture}[scale=.3, baseline]

\draw[lightgray, line width=3pt] (0,1) -- ++(5,0) -- ++(0,-3) -- ++(3,0) -- ++(0,-5) (3,0) -- ++(1,0) -- ++(0,-3) -- ++(3,0) -- ++(0,-1);

\draw [densely dotted] (-.5,1.5) -- ++(6,0) -- ++(0,-2) -- ++(1,0) -- ++(0,-1) -- ++(2,0) -- ++(0,-6) -- ++(-1,0) -- ++(0,3) -- ++(-1,0) -- ++(0,1) -- ++(-3,0) -- ++(0,3) -- ++(-1,0) -- ++(0,1) -- ++(-3,0) -- cycle;

\node at (0,1) [dot] {};
\node at (1,1) [dot] {};
\node at (2,1) [dot] {};
\node at (3,1) [dot] {};
\node at (4,1) [dot] {};
\node at (5,1) [dot] {};
\node at (3,0) [dot] {};
\node at (4,0) [dot] {};
\node at (5,0) [dot] {};
\node at (4,-1) [dot] {};
\node at (5,-1) [dot] {};
\node at (6,-1) [dot] {};
\node at (4,-2) [dot] {};
\node at (5,-2) [dot] {};
\node at (6,-2) [dot] {};
\node at (6,-3) [dot] {};
\node at (4,-3) [dot] {};
\node at (5,-3) [dot] {};
\node at (7,-2) [dot] {};
\node at (8,-2) [dot] {};
\node at (7,-3) [dot] {};
\node at (8,-3) [dot] {};
\node at (7,-4) [dot] {};
\node at (8,-4) [dot] {};
\node at (8,-5) [dot] {};
\node at (8,-6) [dot] {};
\node at (8,-7) [dot] {};

\node[left=0pt of current bounding box.west,anchor=east
    ]{$\F_0 =$};
\end{tikzpicture}
\qquad
\begin{tikzpicture}[scale=.3, baseline]

\draw[lightgray, line width=3pt] (0,1) -- ++(5,0) -- ++(0,-2) -- ++(1,0) node [corner] {} -- ++(0,-1) -- ++(2,0) -- ++(0,-5) (3,0) -- ++(1,0) -- ++(0,-3) -- ++(3,0) -- ++(0,-1);

\draw [densely dotted] (-.5,1.5) -- ++(6,0) -- ++(0,-2) -- ++(1,0) -- ++(0,-1) -- ++(2,0) -- ++(0,-6) -- ++(-1,0) -- ++(0,3) -- ++(-1,0) -- ++(0,1) -- ++(-3,0) -- ++(0,3) -- ++(-1,0) -- ++(0,1) -- ++(-3,0) -- cycle;

\node at (0,1) [dot] {};
\node at (1,1) [dot] {};
\node at (2,1) [dot] {};
\node at (3,1) [dot] {};
\node at (4,1) [dot] {};
\node at (5,1) [dot] {};
\node at (3,0) [dot] {};
\node at (4,0) [dot] {};
\node at (5,0) [dot] {};
\node at (4,-1) [dot] {};
\node at (5,-1) [dot] {};
\node at (6,-1) [dot] {};
\node at (4,-2) [dot] {};
\node at (5,-2) [dot] {};
\node at (6,-2) [dot] {};
\node at (6,-3) [dot] {};
\node at (4,-3) [dot] {};
\node at (5,-3) [dot] {};
\node at (7,-2) [dot] {};
\node at (8,-2) [dot] {};
\node at (7,-3) [dot] {};
\node at (8,-3) [dot] {};
\node at (7,-4) [dot] {};
\node at (8,-4) [dot] {};
\node at (8,-5) [dot] {};
\node at (8,-6) [dot] {};
\node at (8,-7) [dot] {};

\node[left=0pt of current bounding box.west,anchor=east
    ]{$\F_1 =$};
\end{tikzpicture}
\qquad
\begin{tikzpicture}[scale=.3, baseline]

\draw[lightgray, line width=3pt] (0,1) -- ++(5,0) -- ++(0,-2) -- ++(1,0) node [corner] {} -- ++(0,-1) -- ++(2,0) -- ++(0,-5) (3,0) -- ++(1,0) -- ++(0,-2) -- ++(1,0) node [corner] {} -- ++(0,-1) -- ++(2,0) -- ++(0,-1);

\draw [densely dotted] (-.5,1.5) -- ++(6,0) -- ++(0,-2) -- ++(1,0) -- ++(0,-1) -- ++(2,0) -- ++(0,-6) -- ++(-1,0) -- ++(0,3) -- ++(-1,0) -- ++(0,1) -- ++(-3,0) -- ++(0,3) -- ++(-1,0) -- ++(0,1) -- ++(-3,0) -- cycle;

\node at (0,1) [dot] {};
\node at (1,1) [dot] {};
\node at (2,1) [dot] {};
\node at (3,1) [dot] {};
\node at (4,1) [dot] {};
\node at (5,1) [dot] {};
\node at (3,0) [dot] {};
\node at (4,0) [dot] {};
\node at (5,0) [dot] {};
\node at (4,-1) [dot] {};
\node at (5,-1) [dot] {};
\node at (6,-1) [dot] {};
\node at (4,-2) [dot] {};
\node at (5,-2) [dot] {};
\node at (6,-2) [dot] {};
\node at (6,-3) [dot] {};
\node at (4,-3) [dot] {};
\node at (5,-3) [dot] {};
\node at (7,-2) [dot] {};
\node at (8,-2) [dot] {};
\node at (7,-3) [dot] {};
\node at (8,-3) [dot] {};
\node at (7,-4) [dot] {};
\node at (8,-4) [dot] {};
\node at (8,-5) [dot] {};
\node at (8,-6) [dot] {};
\node at (8,-7) [dot] {};

\node[left=0pt of current bounding box.west,anchor=east
    ]{$\F_2 =$};
\end{tikzpicture}
\end{center}

\noindent Hence, following~\eqref{Hilbert series Wallach} yields the Hilbert series below, which agrees with~\cite{EnrightHunziker04}*{Thm.~30}:
\[
P(L_{-kc\zeta}; t) = \frac{1+t+t^2}{(1-t)^{26}}.
\]

\newpage

\appendix

\section{Explicit maps in Howe duality settings}
\label{app:Howe}

\subsection{Howe duality for $(H, \g) = (\GL_k, \gl_{p+q})$}
\label{sub:appendix GL}

The action of $K = \GL_p\times \GL_q$ on $\C[W] = \C[\M_{p,k}\oplus\M_{k,q}]$ is given by 
\[
(g_1,g_2) \cdot f(Y,X)=\det(g_1)^{-k} f(g_1^{-1}Y, \: X g_2).
\]
The Lie algebra homomorphism $\omega: \mathfrak{gl}_{p+q} \longrightarrow \mathcal{D}(\M_{p,k}\oplus\M_{k,q})^{\GL_k}$ is given by
\begin{equation*}
\begin{split}
\begin{pmatrix}
A & B\\
C & D
\end{pmatrix}\longmapsto
&\ \sum_{i=1}^p\sum_{j=1}^p a_{ij} \left(-\sum_{\ell=1}^k y_{j \ell} \frac{\partial}{\partial y_{i\ell}}-k \partial_{ij}\right)
+\sum_{i=1}^q\sum_{j=1}^q d_{ij}\left(\sum_{\ell=1}^k x_{\ell i} \frac{\partial}{\partial x_{\ell j}}\right)\\
&\quad + \sum_{i=1}^p\sum_{j=1}^q b_{ij}\bigg( -\underbrace{\sum_{\ell=1}^k  \frac{\partial^2}{\partial y_{i\ell}\partial x_{\ell j}}}_{\Delta_{ij}}\bigg)
 +\sum_{i=1}^q\sum_{j=1}^p c_{ij} \bigg(\underbrace{\sum_{\ell=1}^k y_{j\ell} x_{\ell i}}_{f_{ji}}\bigg). 
\end{split}
\end{equation*}

\subsection{Howe duality for $(H,\g) = (\O_k, \sp_{2n})$}
\label{sub:appendix O}

The action of $K = \widetilde{\GL}_n$ on $\C[W] = \C[\M_{k,n}]$ is given by 
\[
(g,s) \cdot f(X)=s^{-k} f( X (g^{-1})^t),
\]
where $\widetilde{\GL}_n:=\{(g,s)\in \GL_n \times  \, \C^{\times} : \det(g) = s^2\}$
and hence $s^{-k}$ can be interpreted as ``{$\,\det(g)^{-k/2}$}.''
The Lie algebra homomorphism  $\omega: \mathfrak{sp}_{2n}\longrightarrow \mathcal{D}(\M_{k,n})^{\O_k}$ is given by
\begin{equation*}
\begin{split}
\begin{pmatrix}
A & B\\
C & -A^t
\end{pmatrix}\longmapsto
&\ \sum_{i=1}^n\sum_{j=1}^n a_{ij} \bigg(-\sum_{\ell=1}^k x_{\ell j} \frac{\partial}{\partial x_{\ell i}}-\frac{k}{2} \partial_{ij}\bigg)\\
&\quad +\sum_{i=1}^n\sum_{j=1}^n b_{ij}\bigg( -\frac{1}{4}\underbrace{\sum_{\ell=1}^k  \frac{\partial^2}{\partial x_{\ell i}\partial x_{\ell j}}}_{\Delta_{ij}}\bigg)
+\sum_{i=1}^n\sum_{j=1}^n c_{ij} \bigg(\underbrace{\sum_{\ell=1}^k  x_{\ell j} x_{\ell i}}_{f_{ji}}\bigg).  
\end{split}
\end{equation*}

\subsection{Howe duality for $(H,\g) = (\Sp_{2k}, \so_{2n})$}
\label{sub:appendix Sp}

The action of $K = \GL_n$ on $\C[W] = \C[\M_{2k,n}]$ is given by
\[
g \cdot f(X)=\det(g)^{-k} f(X (g^{-1})^t).
\]
The Lie algebra homomorphism  $\omega: \mathfrak{so}_{2n}\longrightarrow \mathcal{D}(\M_{2k,n})^{\Sp_{2k}}$ is given by
\begin{equation*}
\begin{split}
\begin{pmatrix}
A & B\\
C & -A^t
\end{pmatrix}\longmapsto
&\ \sum_{i=1}^n\sum_{j=1}^n a_{ij} \bigg(-\sum_{\ell=1}^{2k} x_{\ell j} \frac{\partial}{\partial x_{\ell i}}-k \partial_{ij}\bigg)\\
&\quad +\sum_{i=1}^n\sum_{j=1}^n b_{ij}\bigg(-\frac{1}{4}\underbrace{\sum_{\ell=1}^k \Big( \frac{\partial^2}{\partial x_{\ell i}\partial x_{\ell+k,j}}
- \frac{\partial^2}{\partial x_{\ell+k,i}\partial x_{\ell j}}\Big)}_{\Delta_{ij}}\bigg) \\[-10pt]
&\quad +\sum_{i=1}^n\sum_{j=1}^n c_{ij} \bigg(\underbrace{\sum_{\ell=1}^k  \Big(x_{\ell j} x_{\ell+k, i}-x_{\ell+k,j} x_{\ell i}\Big)}_{f_{ji}}\bigg).  
\end{split}
\end{equation*}

\subsection{Maps between coordinate functions}

In each of the three cases above, identify $\mathfrak{p}^-$ with $(\mathfrak{p}^+)^*$ via the $K$-equivariant linear isomorphism
\[
\begin{pmatrix}
0 & 0\\
C & 0
\end{pmatrix}
\mapsto 
\left[
\begin{pmatrix}
0 & B\\
0 & 0
\end{pmatrix}
\mapsto 
\operatorname{trace}(CB)
\right].
\]
Then define the matrix coordinate functions $z_{ij}\in (\mathfrak{p}^+)^*$  by
\[
z_{ij} :\begin{pmatrix}
0 & B\\
0 & 0
\end{pmatrix}\mapsto b_{ij}.
\]
Via the inverse isomorphism  $(\mathfrak{p}^+)^*\rightarrow \mathfrak{p}^-$,
we have
\[
z_{ij}\mapsto 
\begin{cases}
\begin{pmatrix}
0 & 0\\
E_{ji} & 0
\end{pmatrix}
 & \mbox{if $\mathfrak{g}=\mathfrak{gl}_{p+q}$ and $1\leq i\leq p$, $1\leq j\leq q$;}\\[20pt]
 \begin{pmatrix}
0 & 0\\
\frac{1}{2}(E_{ji}-E_{ij})  & 0
\end{pmatrix}&  \mbox{if $\mathfrak{g}=\mathfrak{so}_{2n}$ and $1\leq i<  j\leq n$;}\\[20pt]
\begin{pmatrix}
0 & 0\\
\frac{1}{2}(E_{ji}+E_{ij})  & 0
\end{pmatrix}
& \mbox{if $\mathfrak{g}=\mathfrak{sp}_{2n}$ and $1\leq i\leq  j\leq n$.}
\end{cases} 
\]
It follows that in all three Howe duality settings above, we have
\[
\pi^*: z_{ij} \longmapsto f_{ij}
\]
via $\mathbb{C}[z_{ij}]=\mathbb{C}[\mathfrak{p}^+]\cong \mathcal{U}(\mathfrak{p}^-)\subset \mathcal{U}(\mathfrak{g}) \rightarrow \mathcal{D}(W)^{H}$.

\longthanks{We thank Gerald Schwarz for his helpful comments that improved the introduction.
We are also grateful to both anonymous referees for their meticulous feedback and valuable suggestions.}

\bibliographystyle{amsplain-ac}
\bibliography{main}
\end{document}